\documentclass[11pt]{amsart}
\usepackage[table]{xcolor}
\usepackage{a4wide, amsmath, amsthm, mathabx, mathtools, amscd, amsfonts, amssymb, caption,
	comment, enumerate, epsfig, float, graphics, graphicx, hyperref, IEEEtrantools,
	mathrsfs, mathtools, multirow, rotating, subcaption, textcomp, tikz, tikz-cd, verbatim, xypic, ytableau}
\usepackage[T1]{fontenc}
\usepackage[shortlabels]{enumitem}

\newtheorem{theorem}{Theorem}[section]
\newtheorem{proposition}[theorem]{Proposition}

\newtheorem{corollary}[theorem]{Corollary}
\newtheorem{lemma}[theorem]{Lemma}
\newtheorem{question}[theorem]{Question}
\newtheorem{definition}[theorem]{Definition}

\theoremstyle{definition}

\newtheorem{remark}[theorem]{Remark}

\newtheorem{notation}[theorem]{Notation}

\makeatletter

\newcommand{\dashedrightarrow}[1][2pt]{%
  \settowidth{\@tempdima}{$\rightarrow$}\rightarrow% typeset arrow
  \makebox[-\@tempdima]{\hskip-1.5ex\color{white}\rule[0.5ex]{#1}{1pt}}% typeset overlay
  \phantom{\rightarrow}% advance appropriate horizontal distance
}
\makeatother

\def\CC{\mathbb{C}}
\def\QQ{\mathbb{Q}}
\def\RR{\mathbb{R}}
\def\NN{\mathbb{N}}
\def\ZZ{\mathbb{Z}}
\def\PP{\mathbb{P}}

\def\dim{\operatorname{dim}}

\def\cF{\mathcal{F}}

\def\cG{\mathcal{G}}

\def\Gr{\operatorname{Gr}}
\def\QH{\operatorname{QH}}
\def\H{\operatorname{H}}
\def\HH{\operatorname{HH}}

\def\CO{\mathcal{CO}}

\def\CF{\operatorname{CF}}
\def\HF{\operatorname{HF}}

\def\cM{\mathcal{M}}
\def\cG{\mathcal{G}}

\def\MC{\operatorname{MC}}
\def\H{\operatorname{H}}
\def\bi{\textbf{i}}
\def\bj{\textbf{j}}
\def\tD{\widetilde{D}}

\definecolor{darkgreen}{RGB}{0,153,0}
\definecolor{darkred}{RGB}{204,0,0}
\definecolor{darkblue}{RGB}{0,51,204}
\definecolor{red}{RGB}{242,43,29}
\hypersetup{colorlinks,linkcolor={darkblue},citecolor={darkblue},urlcolor={darkblue}}

\begin{document}

\title{Curved Fukaya algebras and the Dubrovin spectrum}

\author{Marco Castronovo}
\address{Department of Mathematics, Columbia University}
\email{marco.castronovo@columbia.edu}

\thanks{Partially supported by an AMS-Simons Travel Grant.}

\begin{abstract}

Under simplified axioms on moduli spaces of pseudo-holomorphic curves, we show that
weakly unobstructed Fukaya algebras of Floer-nontrivial Lagrangians in a compact
symplectic manifold must have curvature in the spectrum of an operator introduced by
Dubrovin, which acts on the big quantum cohomology. We use the example of the complex
Grassmannian $\Gr(2,4)$ to illustrate a decoupling phenomenon, where the eigenvalues
of finite energy truncations become simple under explicit bulk-deformations.

\end{abstract}

\maketitle
\thispagestyle{empty}

\section{Introduction}\label{SecIntroduction}

\subsection{Cohomological $q$-deformations}

Let $X$ be a compact symplectic manifold, and denote $\Lambda$
the Novikov field over $\CC$ with formal parameter $q$. The cohomology $\H(X;\CC)$
of $X$ has a $q$-deformation $\QH(X)=\H(X;\CC)\otimes\Lambda$ known as quantum cohomology,
whose product is defined by counts of $J$-holomorphic spheres in $X$. This is an associative
$\Lambda$-algebra, with the classical counterpart being recovered for $q=0$.
See Abouzaid-McLean-Smith \cite{AMS} for the most recent approach to this construction,
which also yields $q$-deformations of many generalized cohomology theories.

\subsection{Categorical $q$-deformations}

Starting with Fukaya-Oh-Ohta-Ono \cite{FOOO} a relative analogue of the construction
above has been studied, whose output is a family of $A_\infty$ categories
$\cF_\lambda(X)$ over $\Lambda$ known as Fukaya categories. General Lagrangians
$L\subset X$ have an obstructed $A_\infty$ algebra of Floer cochains $\CF(L)$. To
deal with this, one takes as objects of $\cF_\lambda(X)$ the pairs $(L,D)$ with
$D\in\MC(L)$ a solution of the projective Maurer-Cartan equation in $\CF(L)$ with
curvature $\lambda\in\Lambda$. Morphisms are given by transverse intersection points, with
composition maps arising from counts of $J$-holomorphic polygons in $X$ with Lagrangian
boundary conditions, twisted by $D$. It is expected that taking Hochschild
cohomology recovers the generalized eigenspace $\HH(\cF_\lambda(X))\cong \QH_\lambda(X)$
of eigenvalue $\lambda\in\Lambda$ for a suitable operator, whose action on quantum
cohomology is multiplication by the first Chern class $c_1$; see Perutz-Sheridan \cite{PS} and
Ganatra \cite{G} for results in this direction.

\subsection{The status of computations}

While computations of quantum cohomology algebras are available in many examples,
this is not the case at the categorical level. One could hope to begin the computation of
a Fukaya category by finding finitely many Lagrangians that
generate the split-closed derived category of its triangulated closure. General results of
this kind were obtained for non-compact targets: see Seidel \cite{Se} for Lefschetz fibrations;
Chantraine-Dimitroglou Rizell-Ghiggini-Golovko \cite{CDGG} and Ganatra-Pardon-Shende
\cite{GPS} for Weinstein sectors. In the compact case, generators have been found
in a handful of examples using a variety of techniques: toric Fano manifolds
(Evans-Lekili \cite{EL}); Calabi-Yau and Fano hypersurfaces (Sheridan \cite{Sh15, Sh16});
curves (Seidel \cite{Se11} and Efimov \cite{E}); some Grassmannians
(Castronovo \cite{C20, C23}).

\subsection{The multiplicity issue}

Suppose that one has found a finite set of Lagrangians $L_i\in\cF_\lambda(X)$
which are nontrivial in the derived category, i.e. with $\HF(L_i)\neq 0$ for all $i$; this
can often be checked by examining the disk potential of $L_i$, see e.g. Sheridan
\cite[Proposition 4.2]{Sh16} for monotone tori. A criterion is available for showing that these
objects generate, due to Abouzaid \cite{Ab} for Weinstein domains and announced in the
compact case by Abouzaid-Fukaya-Oh-Ohta-Ono; see
Sheridan \cite[Proposition 1.10]{Sh16} for a proof in the monotone setting. When $\dim\QH_\lambda(X)=1$, this
criterion implies that any of the $L_i$ generates, so a key difficulty
in computing Fukaya categories is to address what happens if $\dim\QH_\lambda(X)>1$. In
the latter case one could hope that if $\HF(L_i,L_j)=0$ for $i\neq j$ then generation
holds, but Floer cohomologies of pairs of Lagrangians are hard to compute in
practice, unless one knows in advance that the Lagrangians $L_i$ can be disjoined by
Hamiltonian isotopies.

\subsection{Generic simplicity}

The starting point for this article is the work of Dubrovin \cite{D96}.
Assume that $\H(X;\CC)$ is concentrated in even degrees, and replace $\QH(X)$ with the
big quantum cohomology $\QH^B(X)$. In many interesting
examples this algebra is semisimple for some values of the deformation parameter $B$, with
a decomposition into fields given by the generalized eigenspaces of the $B$-dependent
operator of multiplication by
$$K^B = c_1 - \sum_{j\neq 2}\frac{j-2}{2}B_j \quad .$$
Here $c_1$ is corrected by a term depending on $B$ and the degrees of its homogeneous
components $B=\sum_{j}B_j$. When $X$ is a Fano variety, Dubrovin conjectured \cite[Conjecture
4.2.2]{D98} that
$K^B$ has simple spectrum (i.e. all eigenvalues have multiplicity one) for generic $B$ if and only if the derived category of coherent
sheaves on $X$ has a full exceptional collection; see Halpern-Leistner \cite{HL} for some
recent progress.

\subsection{Curvatures are eigenvalues}

The big quantum cohomology $\QH^B(X)$ is a generalization of $\QH(X)$, where one includes
in the $q$-deformed product contributions coming from $J$-holomorphic spheres with extra
marked points constrained to map to a so-called
bulk cycle in $X$ of homology class $B$ (Figure \ref{FigBasicCorrespondences}(A)). This construction has a
relative analogue, yielding a family of $A_\infty$ categories $\cF^B_\lambda(X)$ whose
objects are pairs $(L,D)$ with $D\in\MC^B(L)$ a solution of the projective Maurer-Cartan
equation in the $A_\infty$ algebra of bulk-deformed Floer cochains $\CF^B(L)$ with
curvature $\lambda\in\Lambda$. Composition maps in $\cF^B_\lambda(X)$ arise from counts of
$J$-holomorphic polygons in $X$ with Lagrangian boundary conditions, with extra interior
and boundary marked points mapping to the bulk cycle and the Maurer-Cartan chain respectively
(Figure \ref{FigBasicCorrespondences}(B)). Even after
fixing $L$, the property $\MC^B(L)\neq\emptyset$
may depend on $B$, and one should think of the curvature $\lambda\in\Lambda$ as a function
depending on the bulk-deformation $B$. We prove the following.

\begin{theorem}(Theorem \ref{ThmCurvatureInSpectrum})
Assume that $D\in\MC^B(L)$ is a solution of the projective Maurer-Cartan equation in the
$A_\infty$ algebra of bulk-deformed Floer cochains $\CF^B(L)$, and call $\lambda\in\Lambda$
the curvature of the corresponding $A_\infty$ algebra $\CF^B(L,D)$ twisted by $D$. Then
$\HF^B(L,D)\neq 0$ implies that $\lambda$ is an eigenvalue of Dubrovin's operator $K^B$.
\end{theorem}

This theorem generalizes a result of Auroux \cite[Proposition 6.8]{Au07}, which dealt with
the case where $B=0$, $K^0=c_1$ and $D$ is a local system, corresponding to the
Maurer-Cartan solution $D$ having codegree one; see Auroux
\cite[Lemma 4.1]{Au09} for a comparison between local system and Maurer-Cartan twists.
Our proof follows Auroux's approach closely. First, we choose as model for Floer
cochains $\CF(L)=C(L)\otimes\Lambda$ the normalized smooth singular chains, and spell out a set of simplified axioms on moduli of pseudo-holomorphic
curves that allows us to use Fukaya's general picture \cite{F} of correspondences and
pull-push maps without appealing to virtual techniques; see Section \ref{SecAxioms}
for more on this approach. In Section \ref{SecBMCcap}, we construct a chain level
operation $\cap_{B,D}$ called BMC-deformed $q$-cap, which generalizes Auroux's $q$-deformation of
the classical cap product \cite[Definition 6.2]{Au07} by allowing a bulk-deformation
parameter $B$ coupled with a Maurer-Cartan deformation parameter $D$. When $D\in\MC^B(L)$,
we prove that this operation descends to a map
$\QH^B(X)\otimes\HF^B(L,D)\to\HF^B(L,D)$ (Proposition \ref{PropBMCCapOnHomology}). In
Section \ref{SecBMCmod} we verify that $\cap_{B,D}$ endows the bulk-deformed Floer
cohomology $\HF^B(L,D)$
twisted by $D$ with the structure of a $\QH^B(X)$-module (Proposition \ref{PropModuleStructure}).
Finally, in Section \ref{SecSpectrum} we reduce the main theorem to proving the identity
$K^B\cap_{B,D}[L]=\lambda[L]$, where $\lambda=W^B(D)$ is the curvature
of the $A_\infty$ algebra of bulk-deformed Floer cochains $\CF^B(L,D)$ twisted by $D$.
The notation $W^B(D)$ suggests that $\lambda$ can also be thought of as the value at $D$
of the $B$-deformed disk potential $W^B$ of $L$, for which we provide an explicit formula
(Corollary \ref{CorBMCDiskPotential}). The computation of this BMC-deformed $q$-cap product
(Proposition \ref{PropKCapL}) is the hardest part of the proof, and the one that differs
the most from the case with no bulk or Maurer-Cartan deformations. Indeed, a similar
identity holds at the chain level in that case, while it is only true in homology at this
level of generality. The proof consists in checking that $K^B\cap_{B,D}L - W^B(D)L$ is exact in
$\CF^B(L,D)$, which is done by combining degree arguments on the terms of the projective
Maurer-Cartan equation (Proposition \ref{PropMC}) with Axioms (5) and (6) from Section \ref{SecAxioms}.
These two axioms are identities relating chains obtained from correspondences induced by
moduli spaces of disks with geodesic constraints contributing to $-\cap_{B,D}L$ to
those induced by unconstrained disks, obtained by forgetting the vacuous boundary
constraint of mapping to $L$; see the discussion in Section \ref{SecAxioms}.
One should expect, by analogy with the case of no bulk, that taking Hochschild
cohomology will recover the generalized eigenspace
$\HH(\cF^B_\lambda(X))\cong \QH^B_\lambda(X)$ of eigenvalue $\lambda\in\Lambda$ for
Dubrovin's operator $K^B$.
In light of this, Theorem \ref{ThmCurvatureInSpectrum} gives evidence that the multiplicity
issue arising in the computation of $\cF_\lambda(X)$ will disappear in the bulk-deformed
$\cF^B_\lambda(X)$ whenever Dubrovin's operator $K^B$ has generically simple spectrum, thus
simplifying the problem of describing generators in such cases.

\subsection{Finite energy decoupling}

For explicit calculations, one might want to go beyond Dubrovin's observation that the
spectrum of $K^B$ is often generically simple and ask the following.

\begin{question}\label{QExplicitDecoupling}
Given $X$, for which bulk parameters $B$ is the spectrum of $K^B$ simple?
\end{question}

For example, combining Dubrovin's conjecture with
results of Kapranov \cite{Ka} and Kuznetsov-Polishchuk \cite{KP}, one can expect
that $K^B$ will have generically simple spectrum for homogeneous varieties $G/P$
at least in classical Lie types; see Iritani-Koto \cite[Proposition 5.11]{IK} for
unconditional results in this direction.
In Section \ref{SecDecoupling} we focus on what is perhaps the easiest example
of a new phenomenon called finite energy decoupling, that we believe will be important to
study further in order to address Question \ref{QExplicitDecoupling} in practice. After
fixing an energy level $\alpha\in\RR_{> 0}$, one can explicitly determine those $B$ for
which $\dim\QH^B_\lambda(X)=1$ for all eigenvalues
$\lambda\in\Lambda$ of the truncation of $K^B$ modulo $q^\alpha$.
Truncation modulo a power of the formal parameter $q$
corresponds to considering contributions of $J$-holomorphic spheres satisfying an energy bound independent
of the homotopy class.
While finite energy decoupling is not the same as decoupling the eigenvalues
in the spectrum of $K^B$, the two must be related: see Proposition \ref{PropSimpleTruncations}
for some basic observations. An appealing feature of finite energy truncations of $K^B$
is that they can be explicitly computed by combining the WDVV equations with
the knowledge of a few Gromov-Witten numbers. This allows us to prove the following.

\ytableausetup{boxsize=0.3em}
\begin{theorem}(Theorem \ref{ThmDecouplingGr24})
Consider the Grassmannian $X=\Gr(2,4)$ and the energy level $\alpha=2$. The truncation
of $K^B$ modulo $q^2$ has simple spectrum when $B\neq 0$ is supported on the Schubert
cycles $X_{\ydiagram{2,0}}, X_{\ydiagram{1,1}}, X_{\ydiagram{2,2}}\subset\Gr(2,4)$, while
it doesn't when it is supported on the Schubert cycle $X_{\ydiagram{2,1}}\subset\Gr(2,4)$.
\end{theorem}

See Figure \ref{FigSpectralMovie} for an illustration of this theorem, as well as the
interactive interface \cite{C24}. Theorem \ref{ThmDecouplingGr24}
comes close to being the first case where Question \ref{QExplicitDecoupling} is answered
explicitly, and suggests a qualitative distinction between bulk-deformations supported
on different algebraic cycles. We hope to explore this further in the future,
by looking at energy truncation levels $\alpha > 2$ as well as more intricate
coalescing patterns for Dubrovin eigenvalues of higher dimensional Grassmannians
(see \cite[Figure 1]{C20}).

\begin{figure}[H]
  \centering
    \begin{subfigure}[b]{0.2\textwidth}
        \includegraphics[width=\textwidth]{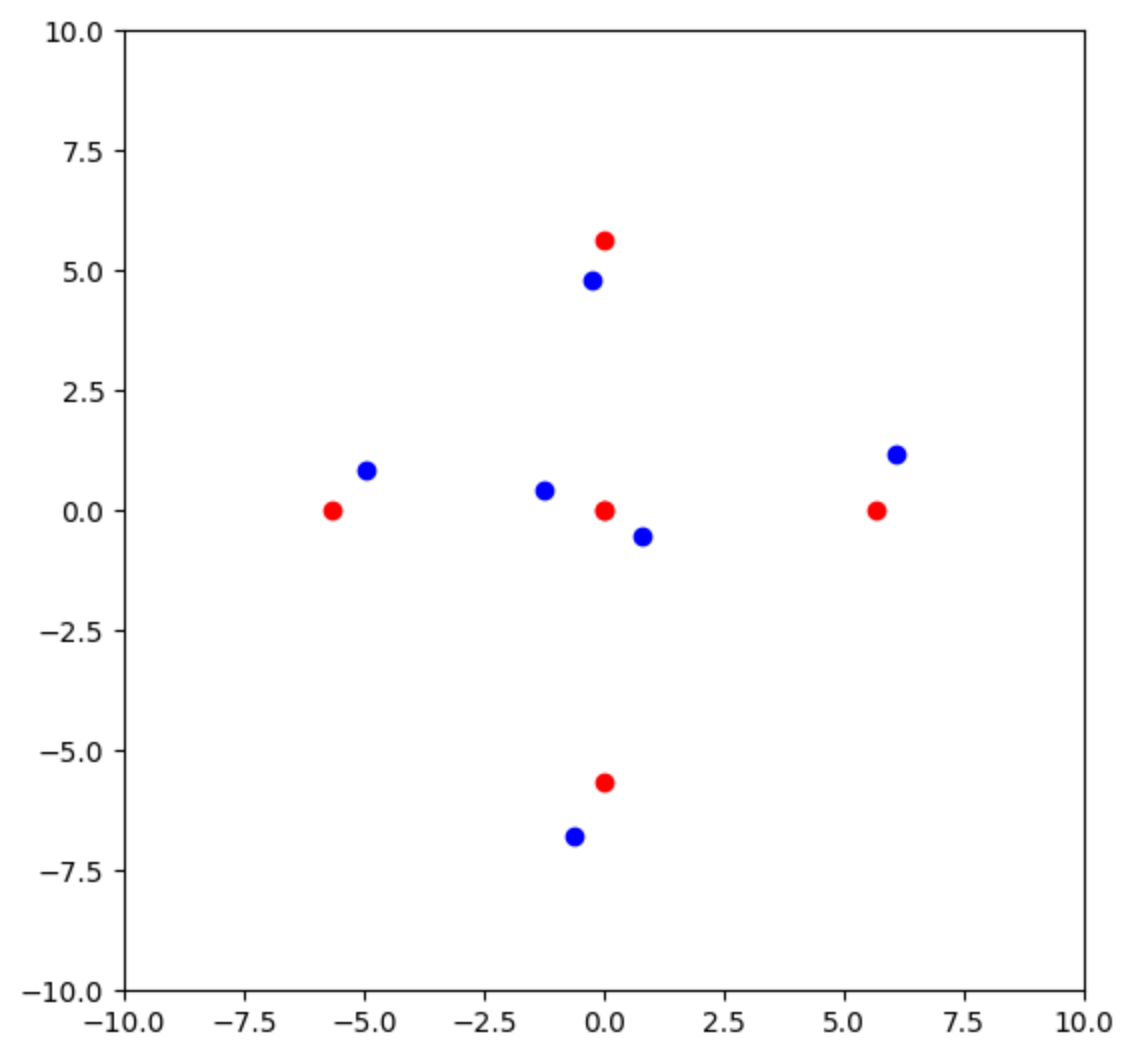}
        \caption{\tabular[t]{@{}l@{}}$X_d=X_{\ydiagram{2,0}}$ \\ $t=0.5+i$\endtabular}
    \end{subfigure}
    \begin{subfigure}[b]{0.2\textwidth}
        \includegraphics[width=\textwidth]{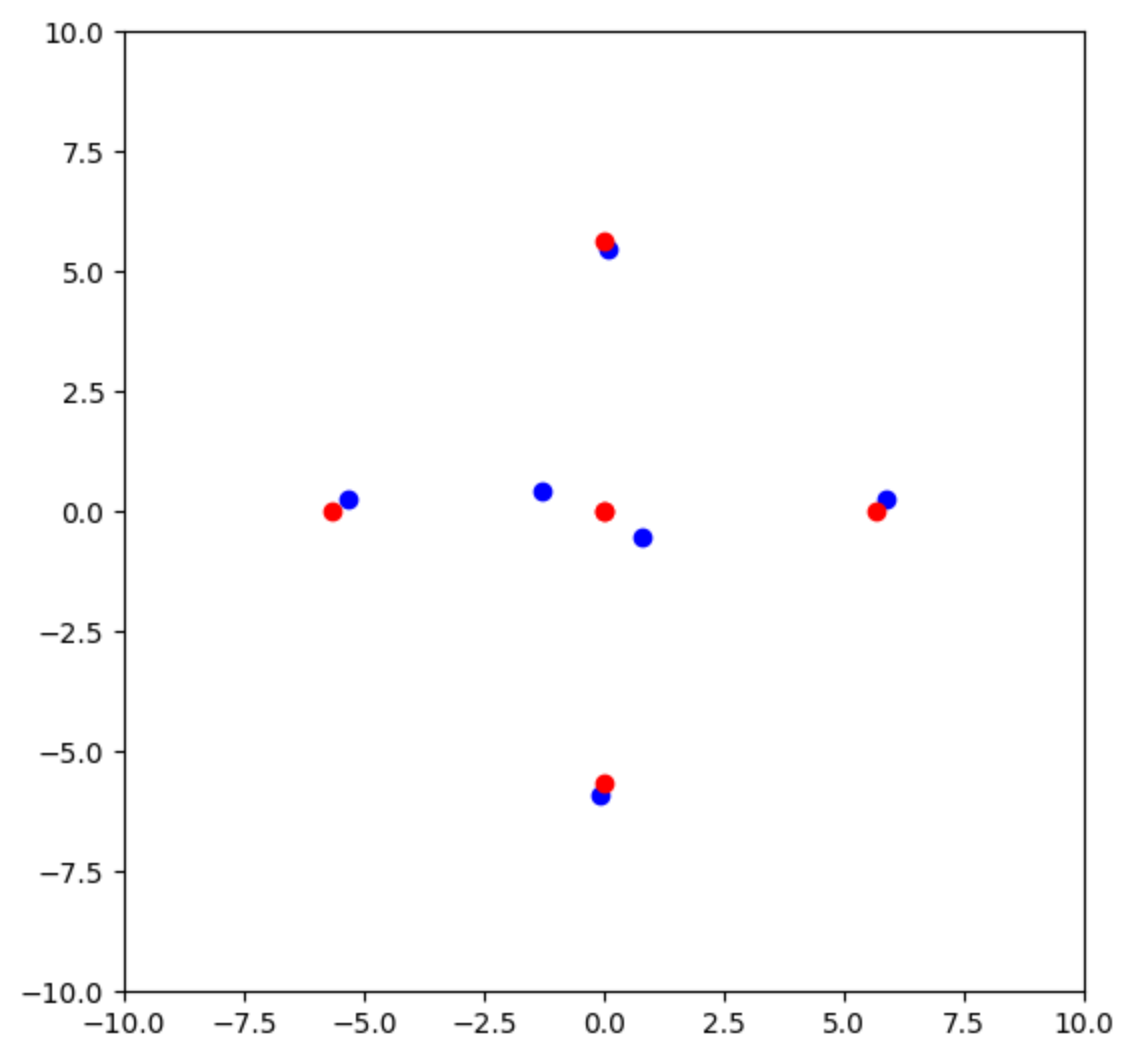}
        \caption{\tabular[t]{@{}l@{}}$X_d=X_{\ydiagram{1,1}}$ \\ $t=0.5+i$\endtabular}
    \end{subfigure}
    \begin{subfigure}[b]{0.2\textwidth}
        \includegraphics[width=\textwidth]{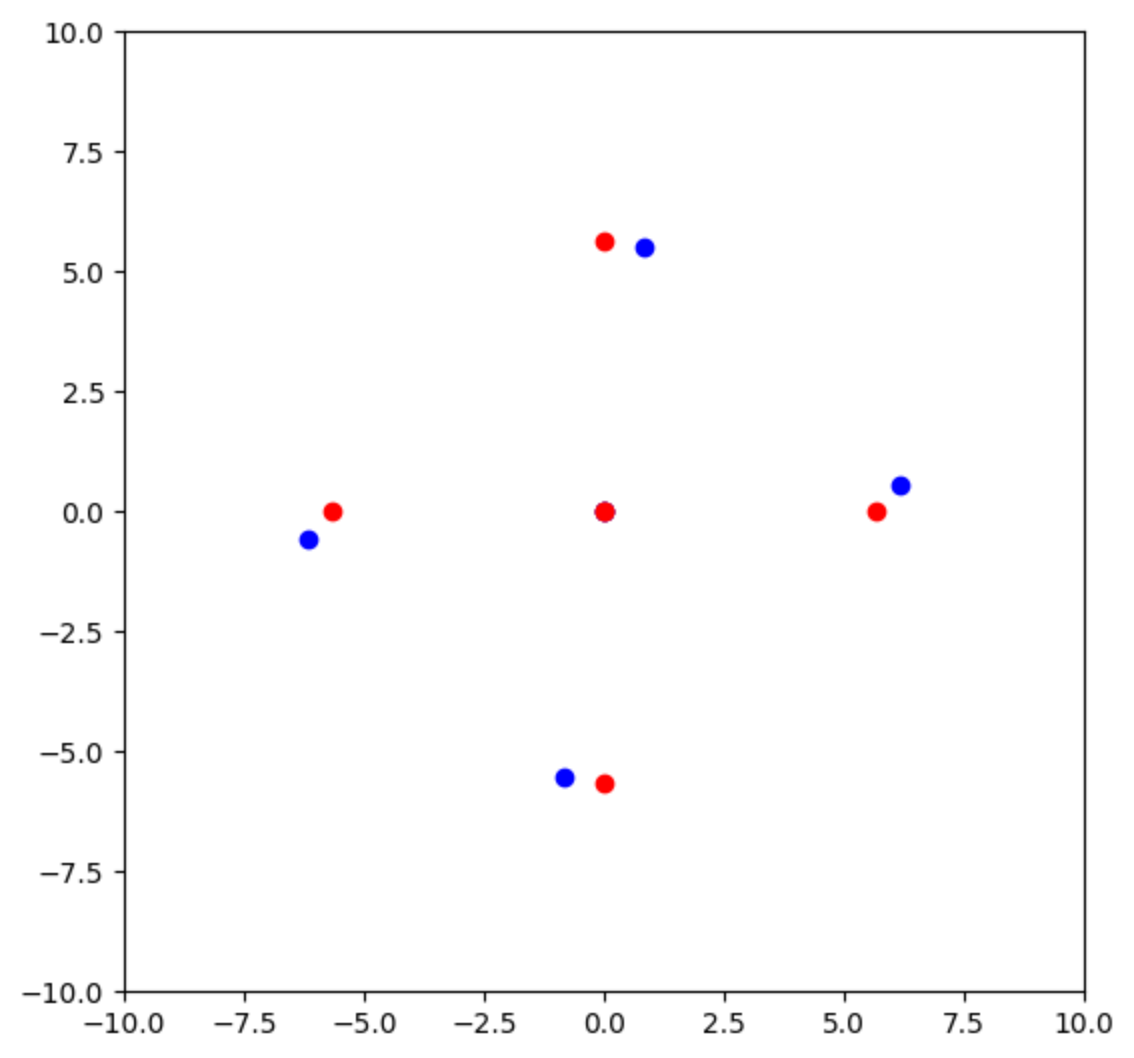}
        \caption{\tabular[t]{@{}l@{}}$X_d=X_{\ydiagram{2,1}}$ \\ $t=0.5+i$\endtabular}
    \end{subfigure}
    \begin{subfigure}[b]{0.2\textwidth}
        \includegraphics[width=\textwidth]{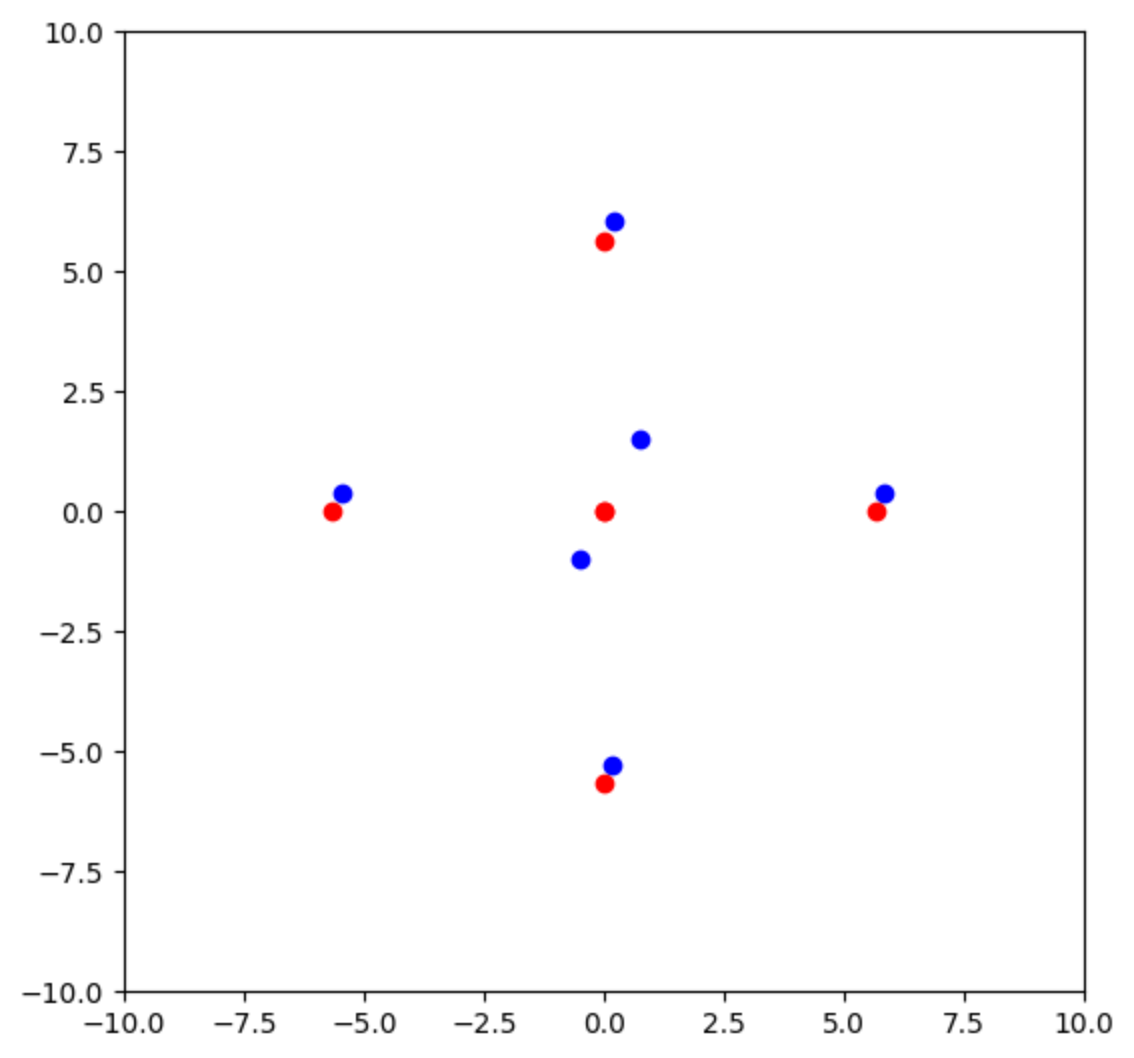}
        \caption{\tabular[t]{@{}l@{}}$X_d=X_{\ydiagram{2,2}}$ \\ $t=0.5+i$\endtabular}
    \end{subfigure}
    
    \begin{subfigure}[b]{0.2\textwidth}
        \includegraphics[width=\textwidth]{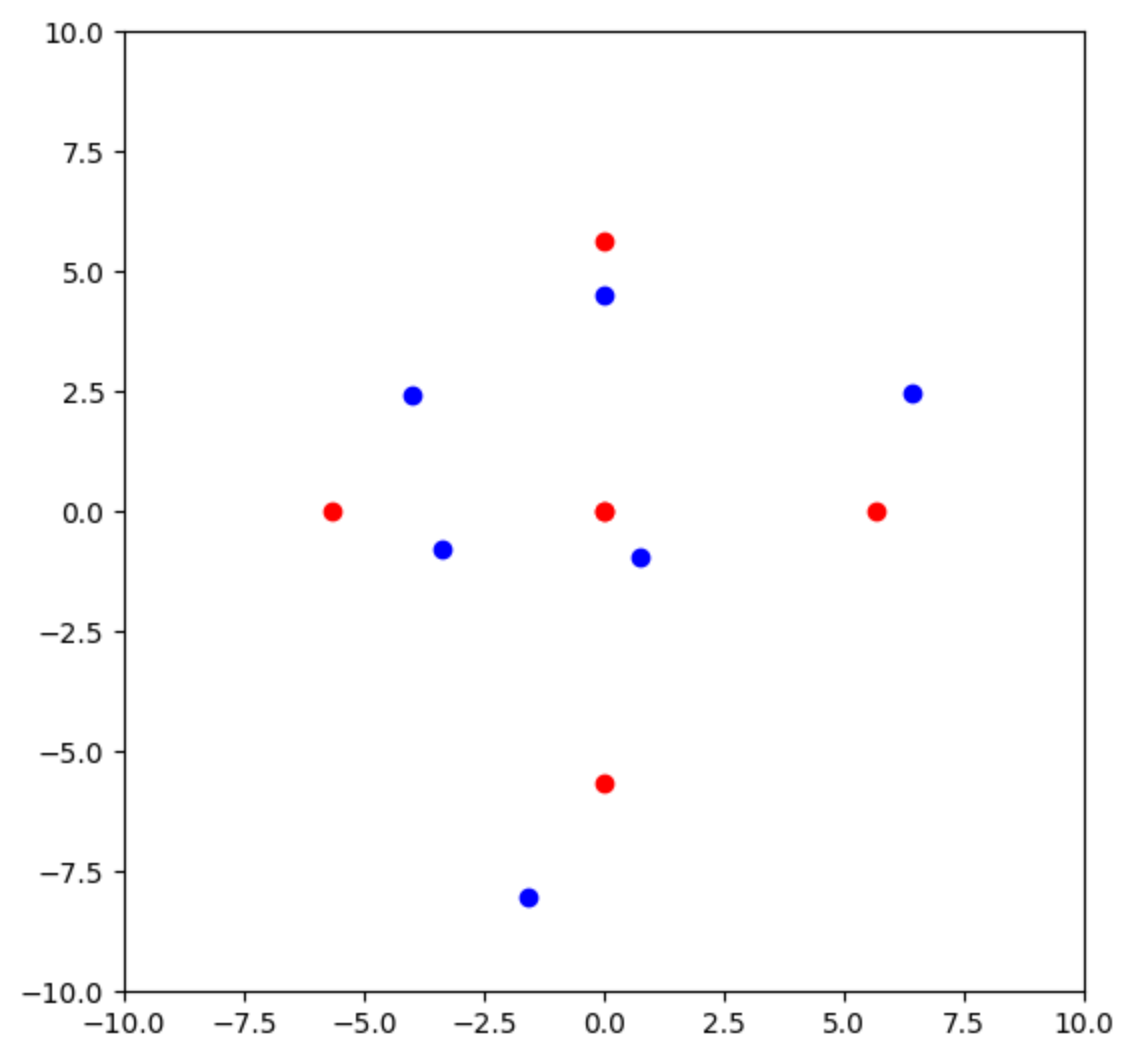}
        \caption{\tabular[t]{@{}l@{}}$X_d=X_{\ydiagram{2,0}}$ \\ $t=1+2i$\endtabular}
    \end{subfigure}
    \begin{subfigure}[b]{0.2\textwidth}
        \includegraphics[width=\textwidth]{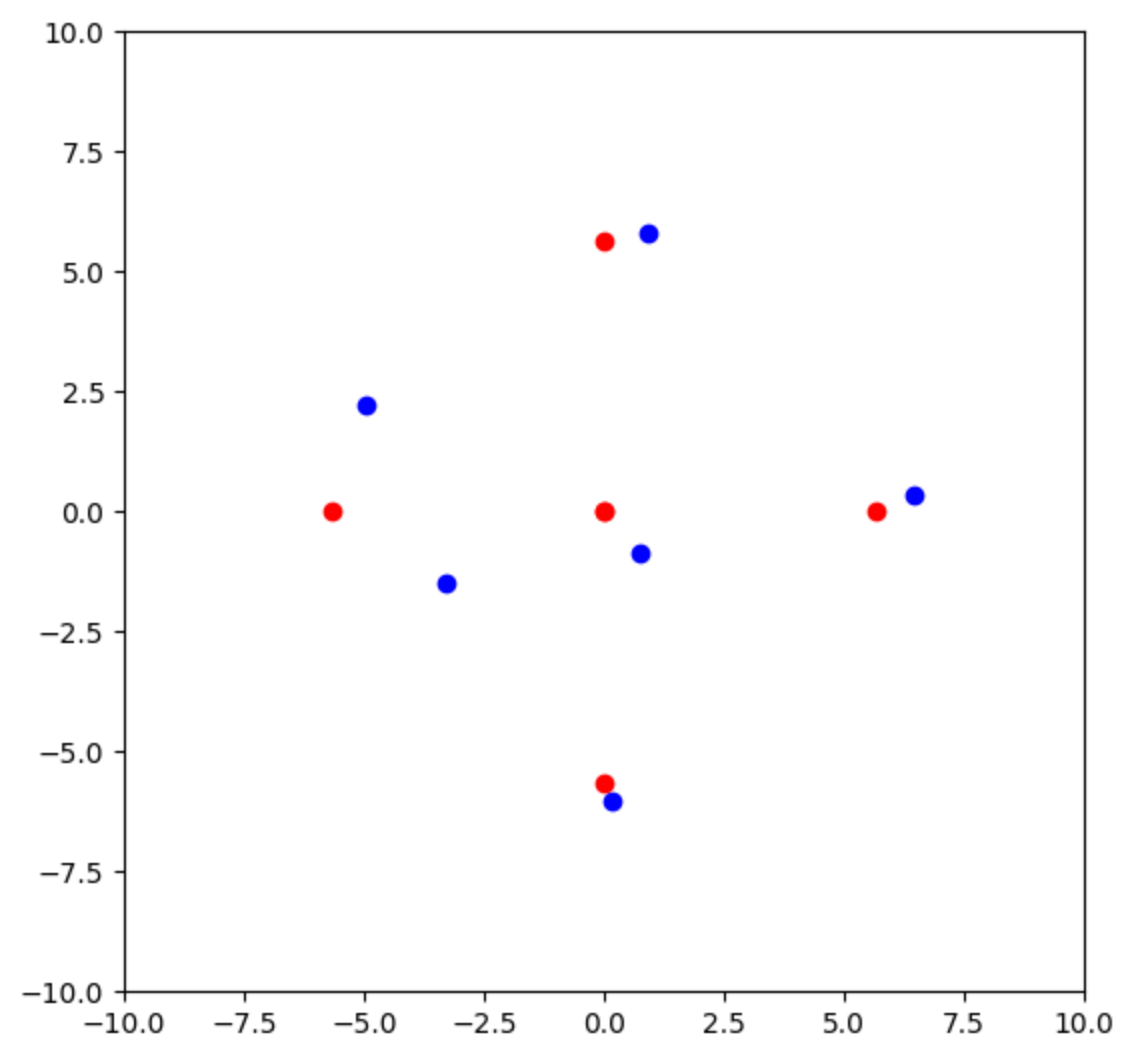}
        \caption{\tabular[t]{@{}l@{}}$X_d=X_{\ydiagram{1,1}}$ \\ $t=1+2i$\endtabular}
    \end{subfigure}
    \begin{subfigure}[b]{0.2\textwidth}
        \includegraphics[width=\textwidth]{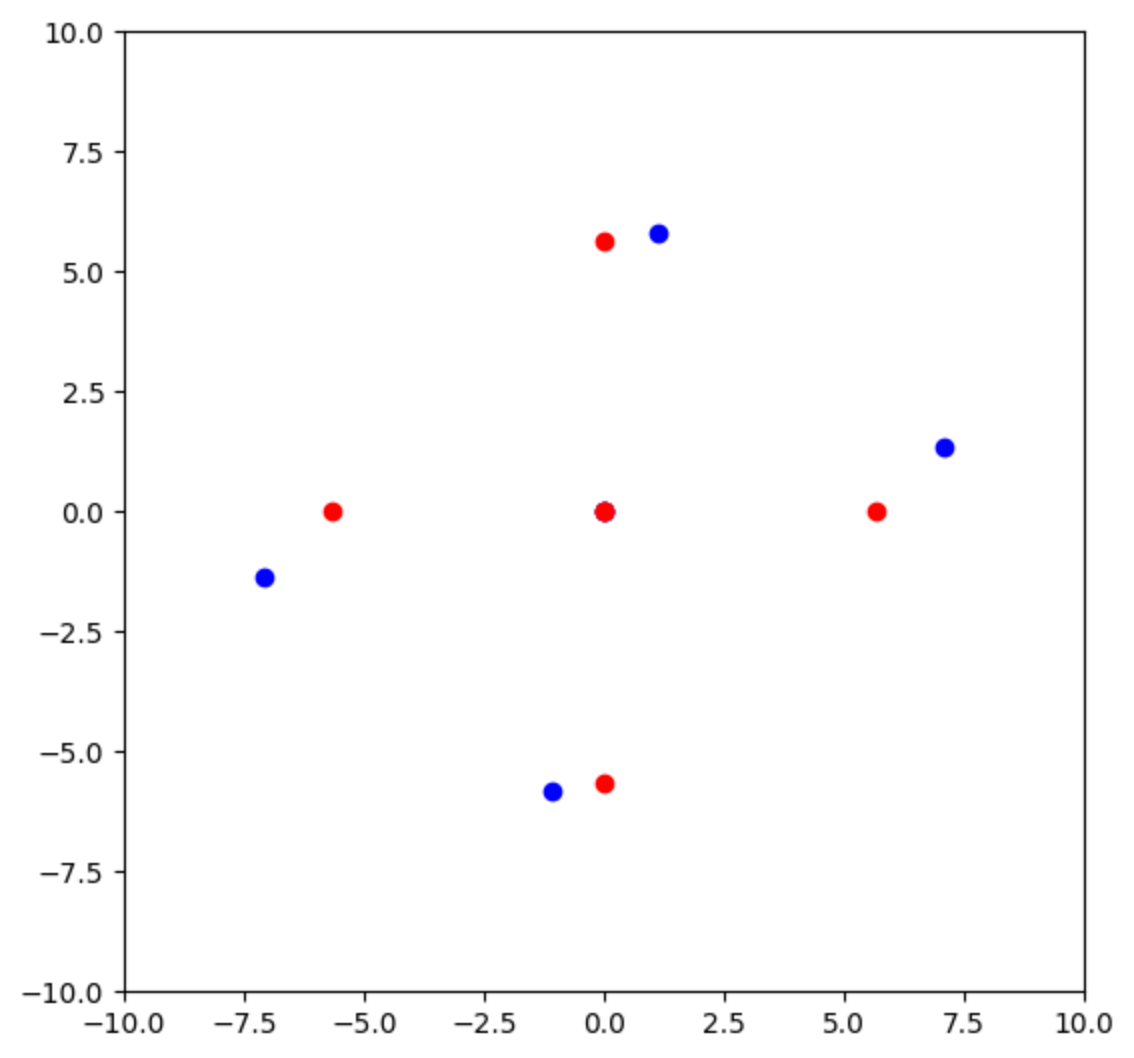}
        \caption{\tabular[t]{@{}l@{}}$X_d=X_{\ydiagram{2,1}}$ \\ $t=1+2i$\endtabular}
    \end{subfigure}
      \begin{subfigure}[b]{0.2\textwidth}
        \includegraphics[width=\textwidth]{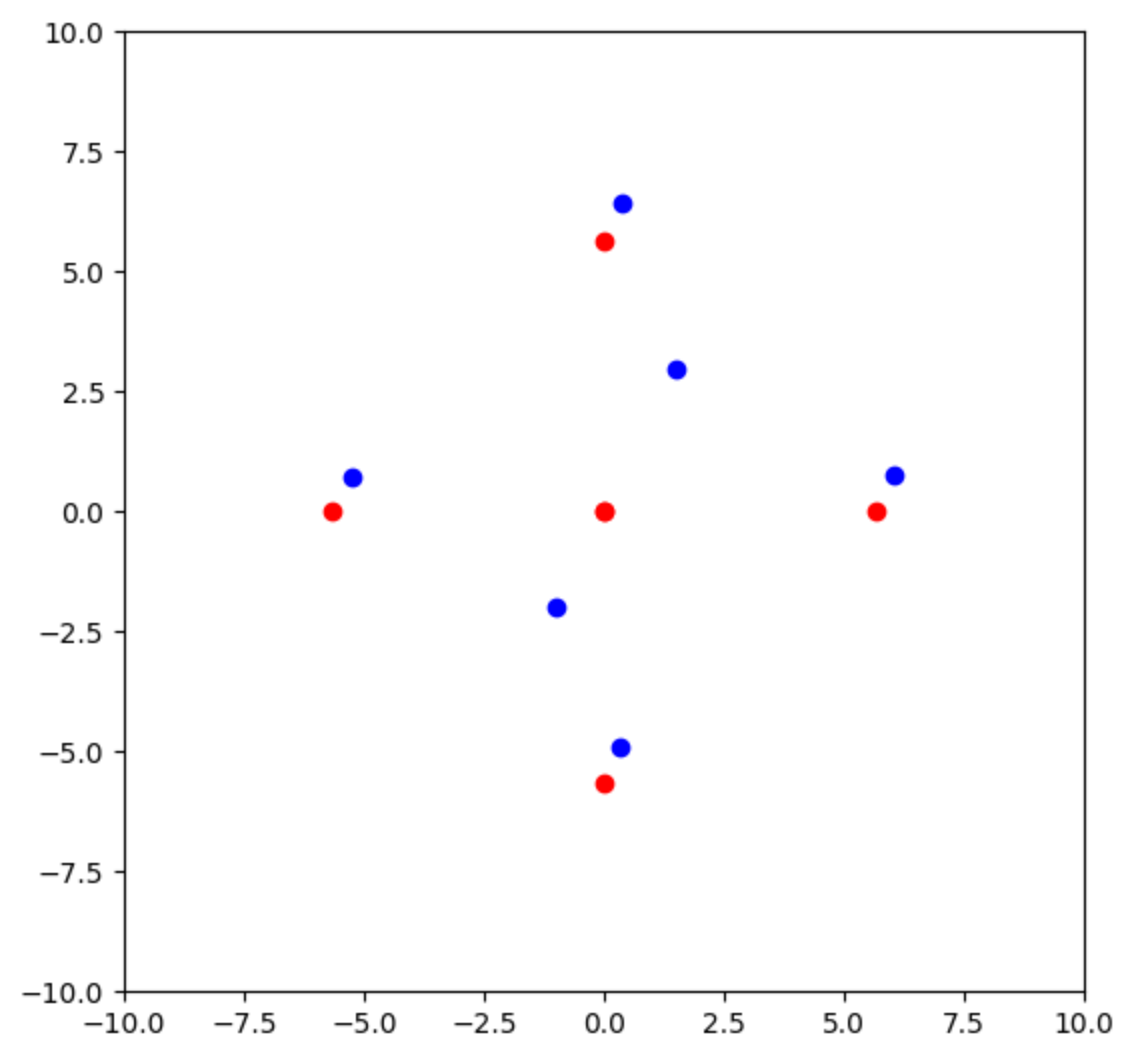}
        \caption{\tabular[t]{@{}l@{}}$X_d=X_{\ydiagram{2,1}}$ \\ $t=1+2i$\endtabular}
    \end{subfigure}
    
    \begin{subfigure}[b]{0.2\textwidth}
        \includegraphics[width=\textwidth]{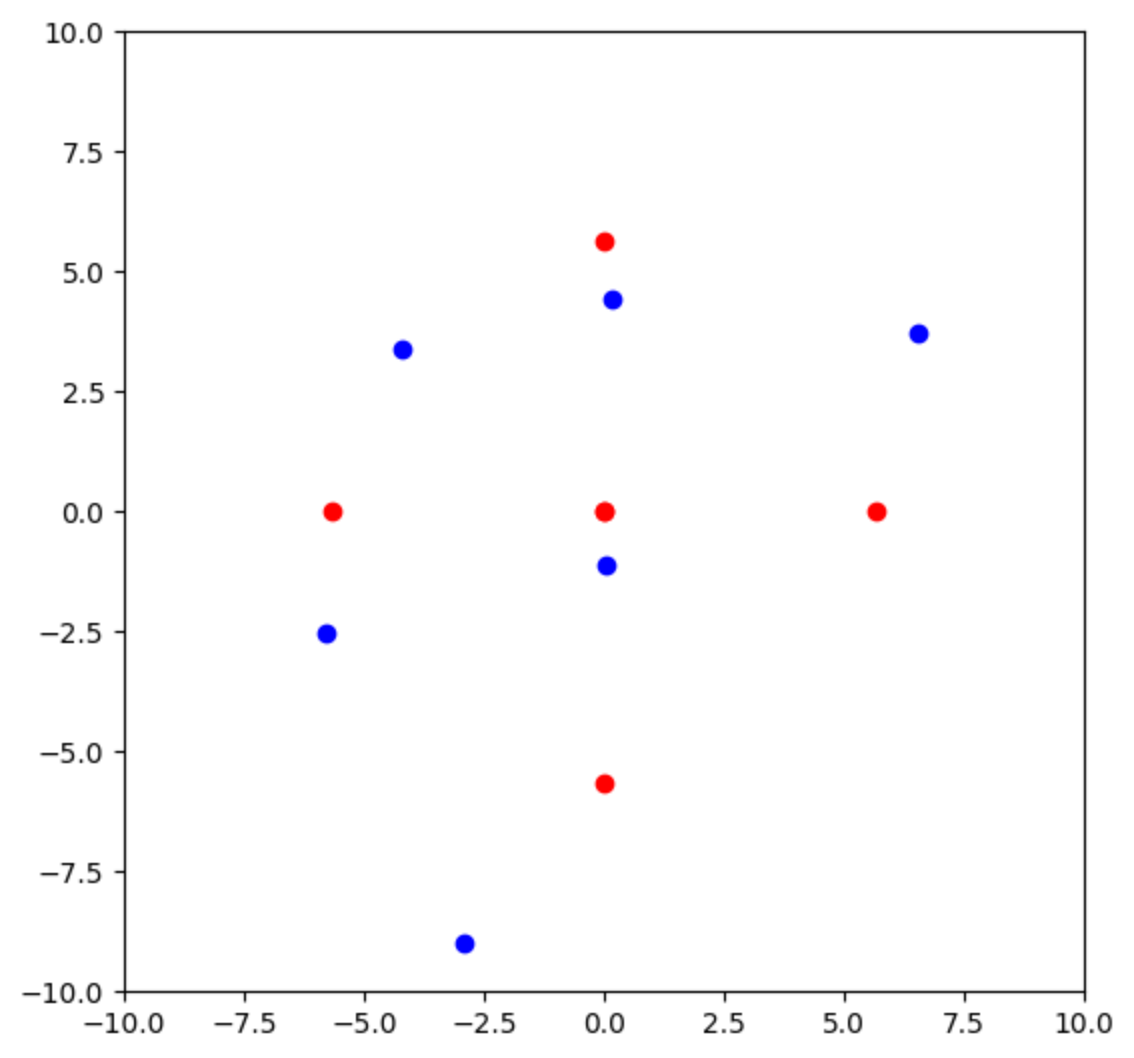}
        \caption{\tabular[t]{@{}l@{}}$X_d=X_{\ydiagram{2,0}}$ \\ $t=1.5+3i$\endtabular}
    \end{subfigure}
    \begin{subfigure}[b]{0.2\textwidth}
        \includegraphics[width=\textwidth]{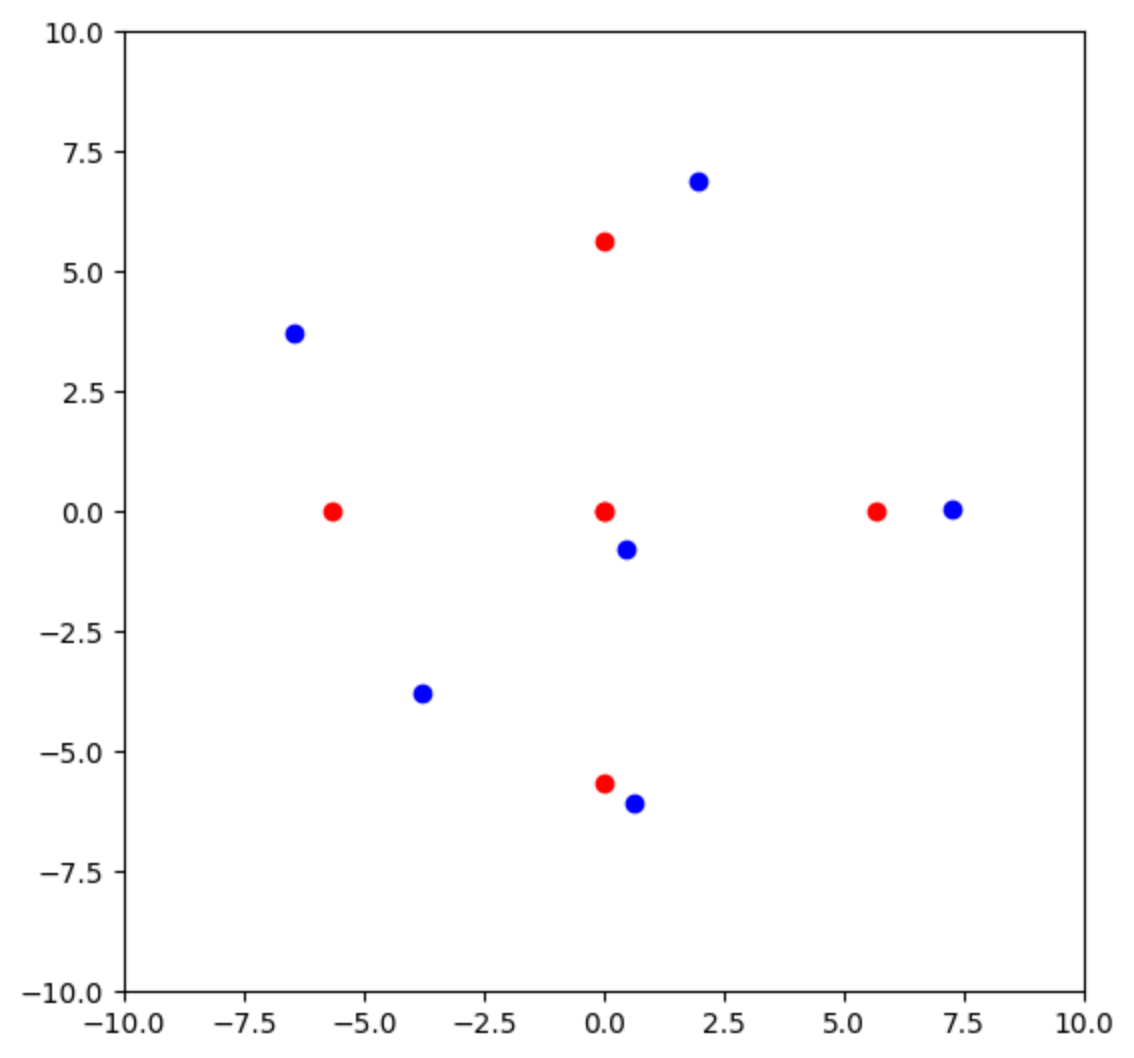}
        \caption{\tabular[t]{@{}l@{}}$X_d=X_{\ydiagram{1,1}}$ \\ $t=1.5+3i$\endtabular}
    \end{subfigure}
    \begin{subfigure}[b]{0.2\textwidth}
        \includegraphics[width=\textwidth]{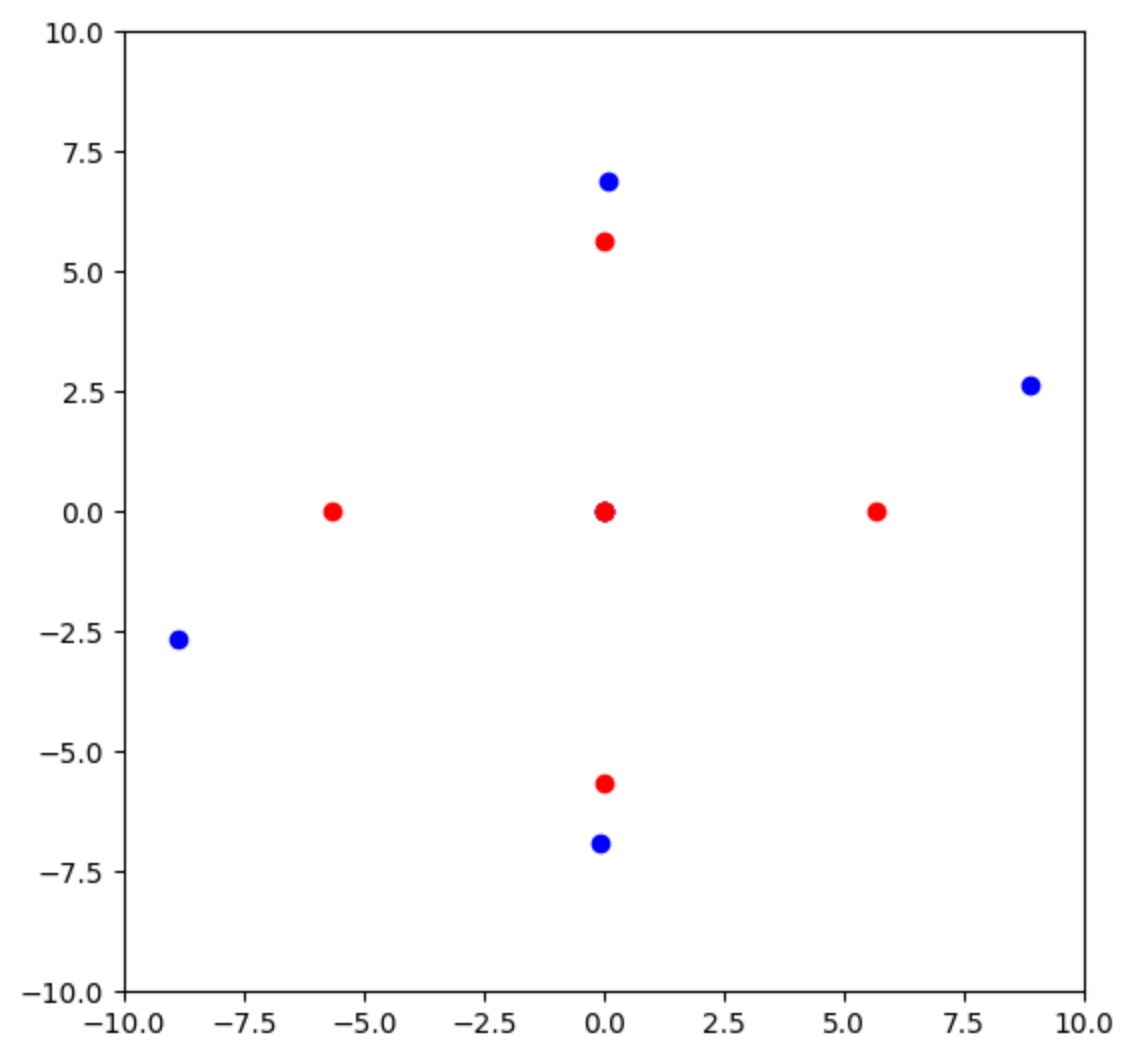}
        \caption{\tabular[t]{@{}l@{}}$X_d=X_{\ydiagram{2,1}}$ \\ $t=1.5+3i$\endtabular}
    \end{subfigure}
      \begin{subfigure}[b]{0.2\textwidth}
        \includegraphics[width=\textwidth]{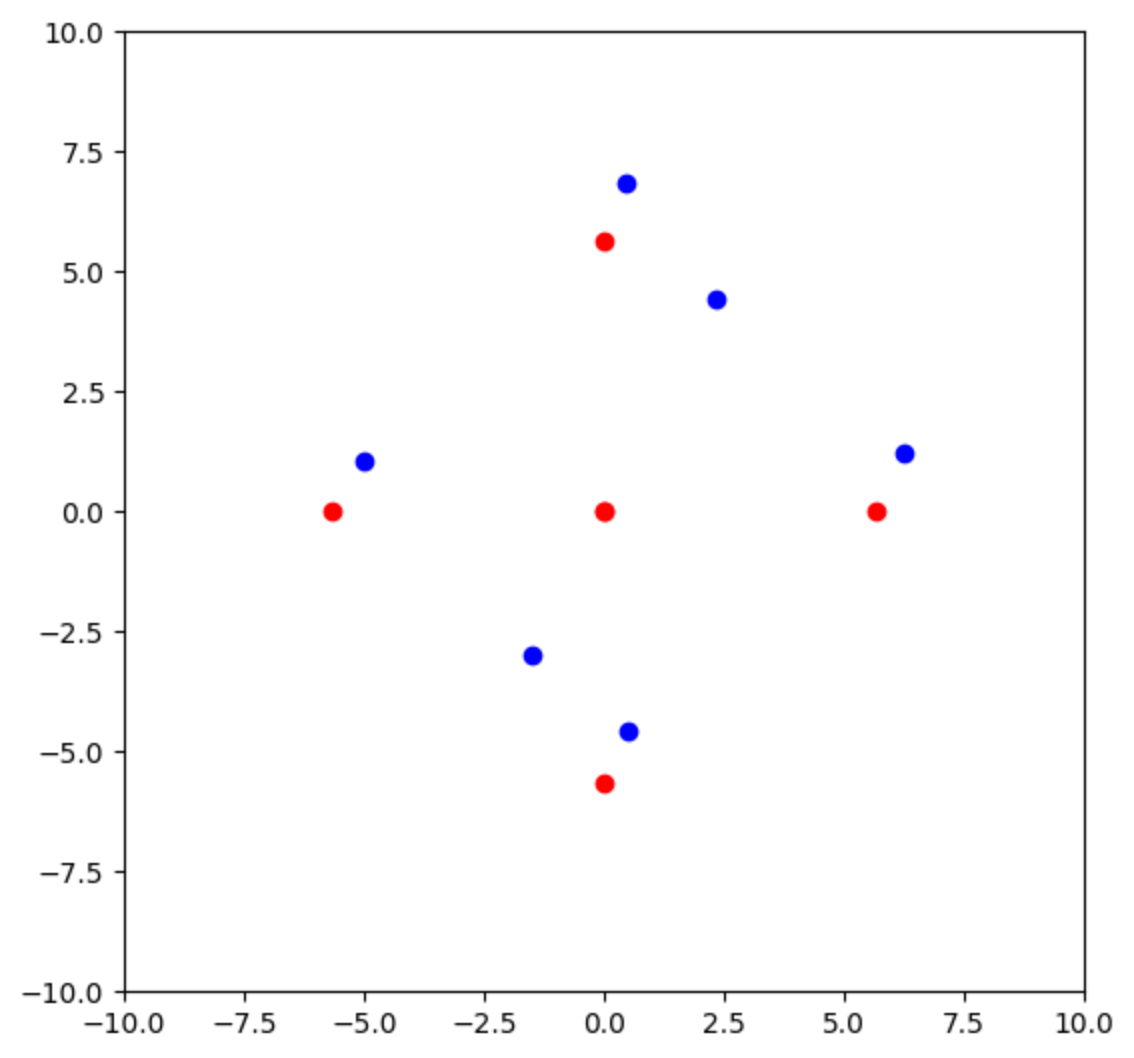}
        \caption{\tabular[t]{@{}l@{}}$X_d=X_{\ydiagram{2,2}}$ \\ $t=1.5+3i$\endtabular}
    \end{subfigure}

    \caption{Spectrum of the $\alpha=2$ truncation of $K^B$ with $B=0$ (red) and $B=t[X_d]$
    (blue) for different Schubert cycles $X_d\subset\Gr(2,4)$ and $t\in\CC^\times$, after
    specializing to $q=1$. The zero eigenvalue has algebraic multiplicity two, and it
    decouples whenever $X_d\neq X_{\ydiagram{2,1}}$.}
    \label{FigSpectralMovie}
\end{figure}

\subsection{Relation with closed-open maps}

By analogy with the case $B=0$, one can expect the existence of algebra maps
$\CO^B_\lambda: \QH^B(X)\to \HH(\cF^B_\lambda(X))$
obtained by bulk-deformation from ordinary closed-open maps, and whose length $0$ part
$(\CO^B_\lambda)^0$ is related to our BMC-deformed $q$-cap operation by
$m^2_{B,D}((\CO^B_\lambda)^0(-),-) = -\cap_{B,D}-$. In this language, our Theorem
\ref{ThmCurvatureInSpectrum} should follow after verifying that: (1) $(\CO^B_\lambda)^0:
\QH^B(X)\to\HF^B(L,D)$ is a unital algebra map; (2) $(\CO^B_\lambda)^0(K^B)=\lambda [L]$,
where $\lambda=W^B(D)$ is as in Corollary \ref{CorBMCDiskPotential}.
See Sheridan \cite[Corollary 2.10]{Sh16} for an alternative proof (in the monotone setting)
of Auroux's result \cite[Proposition 6.8]{Au07} using this approach.
It is our understanding that Hugtenberg \cite{Hu} carries out the construction of length
$0$ bulk-deformed closed-open maps in the de Rham model for Floer cochains, under the same
assumptions used by Solomon-Tukachinsky \cite{ST}. Also, it is our understanding that
Venugopalan-Woodward-Xu \cite{VWX} carry out the construction of full length bulk-deformed
closed-open maps in the pearly model for Floer cochains, under rationality assumptions
that allow the use of stabilizing divisors.
\\\\
\textbf{Acknowledgements} I thank Mohammed Abouzaid for advertising bulk-deformations
to me in 2019. Denis Auroux provided crucial feedback at various points in the
project, especially suggesting initial versions of the forgetful axioms (5)-(6). I
thank Kenji Fukaya for useful comments regarding Section \ref{SecAxioms}. Finally,
I thank Kai Hugtenberg and Chris Woodward for explaining their work on bulk-deformed
closed-open maps.

\section{Axiomatic framework}\label{SecAxioms}

This section introduces some moduli spaces of $J$-holomorphic curves used later
in the article to define various types of operations. We assume that such moduli spaces
satisfy a set of simplified axioms, and discuss the heuristic behind their choice.
Stating and proving virtual analogues of these axioms is an important and interesting
problem, particularly open for the forgetful axioms (5)-(6). See the monograph by
Fukaya-Oh-Ohta-Ono \cite{FOOO20} for the most recent approach to this kind of questions,
based on the notion of Kuranishi structure.

\subsection{Pull-push}

Suppose that $M_\otimes$ and $M_\odot$ are closed
manifolds of dimensions $d_\otimes$ and $d_\odot$, and let $\cM$ be a manifold with corners of
dimension $d$, endowed with maps $\pi_\otimes: M_\otimes \leftarrow \cM$ and $\pi_\odot : \cM \rightarrow M_\odot$.

\begin{definition}\label{DefCorrespondence}
If $\pi_\otimes$ is a submersion, call $\kappa = (\pi_\otimes, \pi_\odot)$ a correspondence
from $M_\otimes$ to $M_\odot$, also denoted $\kappa : M_\otimes\rightsquigarrow M_\odot$.
\end{definition}

Following Fukaya \cite{F}, a correspondence $\kappa:M_\otimes\rightsquigarrow M_\odot$ induces a
map $\kappa_{\#}:C_\bullet (M_\otimes)\to C_{\bullet+d-d_\otimes}(M_\odot)$, sometimes
called pull-push, via the following construction.
A smooth singular chain on $M_\otimes$ is a linear combination of smooth maps $\sigma:\Delta^\bullet\to M_\otimes$
from the $\bullet$-dimensional simplex. In a well-behaved theory of manifolds with corners (see e.g. Joyce \cite{J}), the
assumption that $\pi_\otimes$ is a submersion guarantees the existence of fiber products
$\Delta^\bullet \prescript{}{\sigma}\times_{\pi_\otimes} M_\otimes$ which are manifolds with
corners of dimension $\bullet+d-d_\otimes$. After choosing a triangulation $\sigma_i$
for each of them, the fiber product gives maps $\sigma_i : \Delta^{\bullet+d-d_\otimes} \to \cM$.
One can define $\kappa_{\#}(\sigma)=\sum_i\pi_\odot\circ\sigma_i$ and extend by linearity
to any smooth singular chain.

\subsection{Correspondence axioms}

The following axioms (1)-(4) guarantee that various
moduli spaces of $J$-holomorphic curves used to construct chain-level operations in
the article do give rise to correspondences in the sense of Definition \ref{DefCorrespondence}.
Fix an almost complex structure $J$ compatible with the symplectic structure
of $X$, and suppress it from the notation.
Denote $\cM_l(X,\beta)$ the moduli space of stable $J$-holomorphic spheres in $X$, with $l$ marked points
and class $\beta\in\pi_2(X)$. For each marked point one has a corresponding evaluation
map $\cM_l(X,\beta) \to X$.

\vspace{0.5em}
\fbox{\begin{minipage}{0.9\textwidth}
\textbf{(1) Spheres:}  $\cM_l(X,\beta)$ is a compact manifold of dimension
$2n+2c_1(\beta)+2l-6$, and the evaluation maps are submersions.
\end{minipage}}
\vspace{0.5em}

Denote $\cM_{k,l}(L,\beta)$ the moduli space of stable $J$-holomorphic disks in $X$ with boundary in $L$,
with $k$ boundary marked points and $l$ interior marked points, of class
$\beta\in\pi_2(X,L)$. For each interior marked point one has a corresponding evaluation
map $\cM_{k,l}(L,\beta) \to X$, and similarly for each boundary marked point one has
a corresponding evaluation map $\cM_{k,l}(L,\beta) \to L$.

\vspace{0.5em}
\fbox{\begin{minipage}{0.9\textwidth}
\textbf{(2) Disks:}  $\cM_{k,l}(L,\beta)$ is a compact manifold with corners of dimension
$n+\mu(\beta)+k+2l-3$, and the evaluation maps are submersions.
\end{minipage}}
\vspace{0.5em}

Denote $\cG'_{2+k,1+l}(L,\beta)$ the moduli space of stable $J$-holomorphic
disks in $X$ with boundary in $L$, with $k$ boundary marked points and $l$ interior
marked points, of class $\beta\in\pi_2(X,L)$, with one extra interior and one extra
boundary marked point, both constrained to be on a hyperbolic geodesic from the first
boundary marked point. For each interior marked point one has a corresponding evaluation
map $\cG'_{2+k,1+l}(L,\beta) \to X$, and similarly for each boundary marked point one has
a corresponding evaluation map $\cG'_{2+k,1+l}(L,\beta) \to L$.

\vspace{0.5em}
\fbox{\begin{minipage}{0.9\textwidth}
\textbf{(3) Geodesic$'$:} $\cG'_{2+k,1+l}(L,\beta)$ is a compact manifold with corners of
dimension $n+\mu(\beta)+k+2l$, and the evaluation maps are submersions.
\end{minipage}}
\vspace{0.5em}

Denote $\cG''_{2+k,2+l}(L,\beta)$ the moduli space of stable $J$-holomorphic disks in $X$ with boundary in
$L$, with $k$ boundary marked points and $l$ interior marked points, of class
$\beta\in\pi_2(X,L)$, with two extra interior and one extra boundary marked point, all
constrained to be on a hyperbolic geodesic from the first boundary marked point.
For each interior marked point one has a corresponding evaluation
map $\cG''_{2+k,2+l}(L,\beta) \to X$, and similarly for each boundary marked point one has
a corresponding evaluation map $\cG''_{2+k,2+l}(L,\beta) \to L$.

\vspace{0.5em}
\fbox{\begin{minipage}{0.9\textwidth}
\textbf{(4) Geodesic$''$:} $\cG''_{2+k,2+l}(L,\beta)$ is a compact manifold with corners of
dimension $n+\mu(\beta)+k+2l+1$,
and the evaluation maps are submersions.
\end{minipage}}
\vspace{0.5em}

\begin{notation}
Throughout the paper, correspondences $\kappa : M_\otimes\rightsquigarrow M_\odot$
induced by moduli spaces $\cM$ of stable $J$-holomorphic spheres/disks will be described
by designating some marked points as inputs, denoted $\otimes$, and one marked point
as output, denoted $\odot$. The corresponding evaluations then specify maps
$\pi_\otimes: M_\otimes \leftarrow \cM$ and $\pi_\odot : \cM \rightarrow M_\odot$
to the corresponding products of $L$ and $X$.
\end{notation}

\subsection{Forgetful axioms}

The correspondence axioms listed so far can be grouped in two classes: (1)-(2) involve
no geodesic constraint, while (3)-(4) do. The results of Section \ref{SecSpectrum}
rely on two additional axioms, which relate chains obtained from moduli of disks with
and without geodesic constraints as in (3) and (2) respectively.
Both axioms describe what happens when one forgets the constraint of having the input
boundary marked point on the hyperbolic geodesic mapping to $L\subset X$.
Since disks already have boundary on $L$, this is a vacuous constraint and one expects to
sweep out the same chain by evaluation from a different moduli space where
the constraint is removed. The result of this operation depends on what kind of cycle
the remaining interior marked point on the hyperbolic geodesic was constrained to map to.

\vspace{0.5em}
\fbox{\begin{minipage}{0.9\textwidth}
\textbf{(5) Forget$'$:} If $L\subset X\setminus Z$ with $|Z|^\vee=2$ and $[Z]^\vee = c_1$ then
$$\sum_{k_0+k_1=k}(Z\cap_{B,D}L)_\beta^{k_0,k_1;l}=\frac{\mu(\beta)}{2}q^\beta_{k,l}(D^k,B^l) \quad .$$
\end{minipage}}
\vspace{0.5em}

\vspace{0.5em}
\fbox{\begin{minipage}{0.9\textwidth} 
\textbf{(6) Forget$''$:} If $B=\sum_jB_j$ with $|B_j|^\vee=j\neq 2$ then
$$\sum_{k_0+k_1=k}(B_j\cap_{B,D}L)_\beta^{k_0,k_1;l}=\frac{1}{l+1}\sum_{t=0}^lq^\beta_{k,l+1}(D^k,B^t,B_j,B^{l-t}) \quad .$$
\end{minipage}}
\vspace{0.5em}

\section{Bulk and Maurer-Cartan deformations}\label{SecBMC}

This section reviews the definition of bulk-deformed quantum cup product $\star_B$, as
well as of the coupled bulk/Maurer-Cartan deformation $m^d_{B,D}$ of the $A_\infty$ algebra
of a Lagrangian $L\subset X$. These constructions rely on the correspondence axioms (1)-(2)
from Section \ref{SecAxioms}.

\ytableausetup{boxsize=1em}
\begin{figure}[H]
  \centering
    \begin{subfigure}[t]{0.4\textwidth}
        \includegraphics[width=\textwidth]{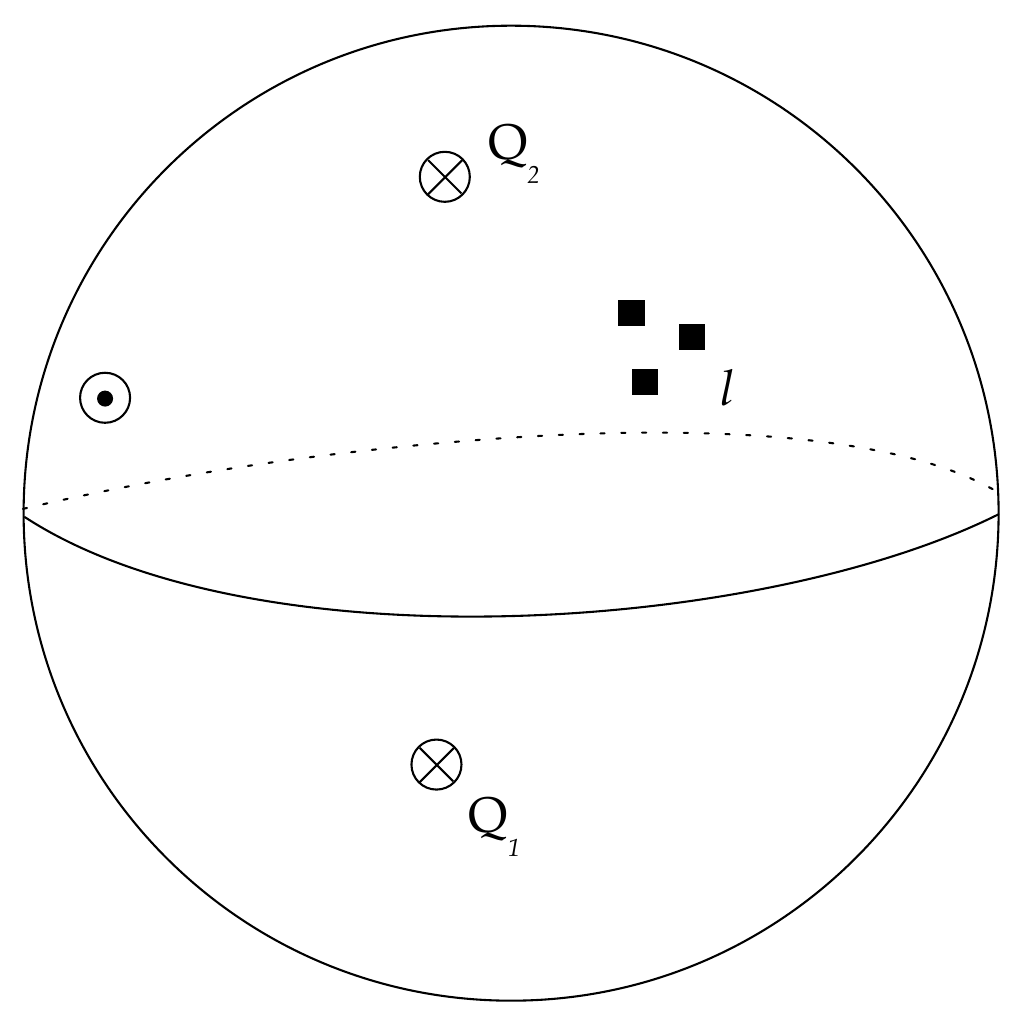}
        \caption{$\kappa:X^2\rightsquigarrow X$ induced by $\cM_{3+l}(X,\beta)$ for $l\geq 0$}
    \end{subfigure}
    \begin{subfigure}[t]{0.4\textwidth}
        \includegraphics[width=\textwidth]{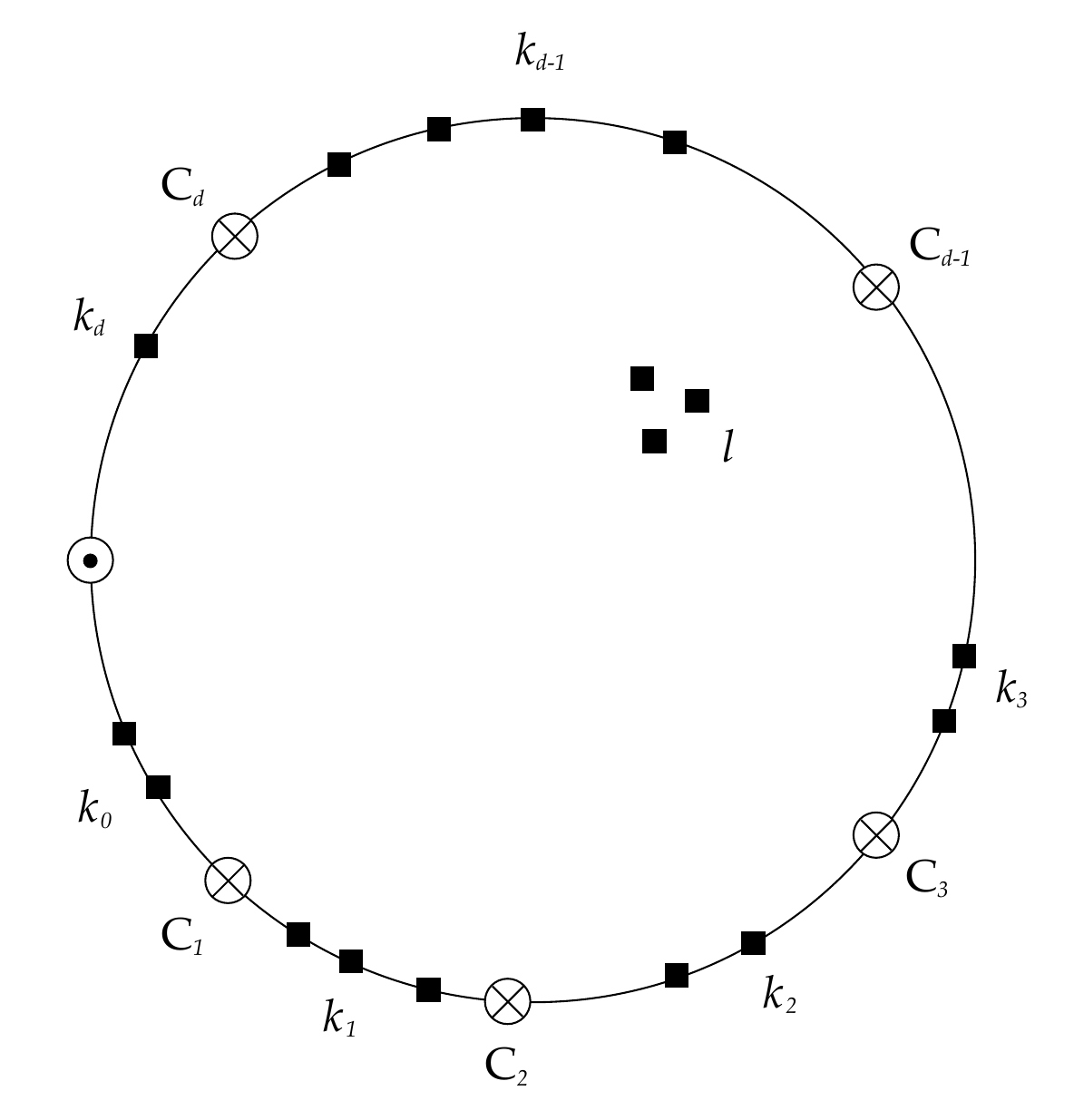}
        \caption{$\kappa:L^d\rightsquigarrow L$ induced by $\cM_{d+k_0+\ldots +k_d,l}(L,\beta)$
        for $k_0,\ldots ,k_d\geq 0$ and $l\geq 0$}
    \end{subfigure}
    \caption{Some correspondences induced by moduli of stable $J$-holomorphic spheres/disks. The
    interior/boundary marked points denoted $\blacksquare$ are constrained to map to $B/D$.}
    \label{FigBasicCorrespondences}
\end{figure}

\subsection{B-deformations}

Denote $C(X)$ the module of normalized smooth singular chains on $X$ with $\CC$ coefficients.
Throughout the paper, fix a cycle $B\in C(X)$ concentrated in even degrees and call it
bulk deformation parameter. Denote $B=\sum_jB_j$ its decomposition into homogeneous
components, indexed by their codegrees $j=2n-|B_j|$.

\begin{definition}\label{DefBCup}
The B-deformed cup product on $C(X)\otimes\Lambda$ is
$$
Q_1\star_B Q_2 = \sum_{l\geq 0}\sum_{\beta\in\pi_2(X)}q^{\omega(\beta)}
\frac{1}{l!}(Q_1\star_B Q_2)_\beta^l \quad ,$$
where the chain $(Q_1\star_B Q_2)_\beta^l=\kappa_{\#}(Q_1\times Q_2\times B^l)$
is obtained via pull-push from the correspondence $\kappa$ in Figure \ref{FigBasicCorrespondences}(A).
\end{definition}

\begin{proposition}(see e.g. McDuff-Salamon \cite{MS})
The B-deformed cup product descends to homology, and makes the corresponding
$\Lambda$-module an algebra with unit $[X]$ denoted $\QH^B(X)$.
\end{proposition}

\subsection{BMC-deformations}

Denote $C(L)$ the module of normalized smooth singular chains on $L$ with $\CC$ coefficients.
Throughout the paper, fix a chain $D\in C(L)\otimes \Lambda_{>0}$ concentrated in odd degrees and call it
Maurer-Cartan deformation parameter. Denote $D=\sum_iD_i$ its decomposition into homogeneous
components, indexed by their codegrees $i=n-|D_i|$.

\begin{definition}\label{DefBMCAinftyMaps}
For $d\geq 0$ the BMC-deformed $A_\infty$ products on $C(L)\otimes\Lambda$ are
$$m^d_{B,D}(C_1,\ldots ,C_d)=
\sum_{k_0,\ldots ,k_d\geq 0}\sum_{l\geq 0}\sum_{\beta\in\pi_2(X,L)}q^{\omega(\beta)}
\frac{1}{l!}q^\beta_{d+k_0+\ldots +k_d,l}(D^{k_0},C_1,D^{k_1},\ldots ,C_d,D^{k_d},B^l) \quad .$$
When $\beta\neq 0$ the chain $q^\beta_{d+k_0+\ldots +k_d,l}(D^{k_0},C_1,D^{k_1},\ldots ,C_d,D^{k_d},B^l)=
\kappa_{\#}(D^{k_0}\times C_1\times D^{k_1}\times \cdots \times C_d\times D^{k_d}\times B^l)$
is obtained via pull-push from the correspondence $\kappa$ in Figure \ref{FigBasicCorrespondences}(B).
For $\beta=0$ set $q^0_{\bullet,\bullet}=0$ except for $q^0_{1,0}=\partial$.
\end{definition}

The chain $D\in C(L)\otimes \Lambda_{>0}$ satisfies the projective Maurer-Cartan equation
if $m^0_{B,D}(1)=\lambda L$ for some $\lambda\in\Lambda$, in which case $\lambda$
is called the curvature. Denote $\MC^B(L)$ the set of all solutions $D$ of the projective
Maurer-Cartan equation.

\begin{proposition}\label{PropBMCHF}(see Fukaya-Oh-Ohta-Ono \cite{FOOO})
The maps $m^\bullet_{B,D}$ are an $A_\infty$ algebra structure on the $\Lambda$-module $C(L)$.
Moreover if $D\in\MC^B(L)$ then $(m^1_{B,D})^2=0$, and the BMC=deformed product $m^2_{B,D}$
descends to homology and makes the corresponding $\Lambda$-module an algebra with unit $[L]$
denoted $\HF^B(L,D)$.
\end{proposition}

\section{BMC deformation of cap product}\label{SecBMCcap}

This section introduces a new coupled bulk and Maurer-Cartan (BMC for short) deformation
$\cap_{B,D}$ of the quantum cap product operation. This constructions relies on the
correspondence axiom (3) from Section \ref{SecAxioms}.

\subsection{Chain-level definition}

The definition of the BMC-deformed quantum cap product $\cap_{B,D}$ is given at the
chain-level as follows.

\begin{definition}\label{DefBMCCap}
The BMC-deformed cap product is
$$Q\cap_{B,D} C = \sum_{k_0,k_1\geq 0}\sum_{l\geq 0}\sum_{\beta\in\pi_2(X,L)}q^{\omega(\beta)}
\frac{1}{l!}(Q\cap_{B,D}C)_\beta^{k_0,k_1;l} \quad .$$
When $\beta\neq 0$ the chain $(Q\cap_{B,D}C)_\beta^{k_0,k_1;l}=\kappa_{\#}
(D^{k_0}\times C\times D^{k_1}\times Q\times B^l)$ is obtained via pull-push from the correspondence
$\kappa$ in Figure \ref{FigCapCorrespondences} (A).
For $\beta=0$ set $(\cdot\cap_{B,D}\cdot)_0^{\bullet,\bullet;\bullet}=0$ except for
$(\cdot\cap_{B,D}\cdot)_0^{0,0;0}=\cap$.
\end{definition}

\ytableausetup{boxsize=1em}
\begin{figure}[H]
  \centering
    \begin{subfigure}[t]{0.4\textwidth}
        \includegraphics[width=\textwidth]{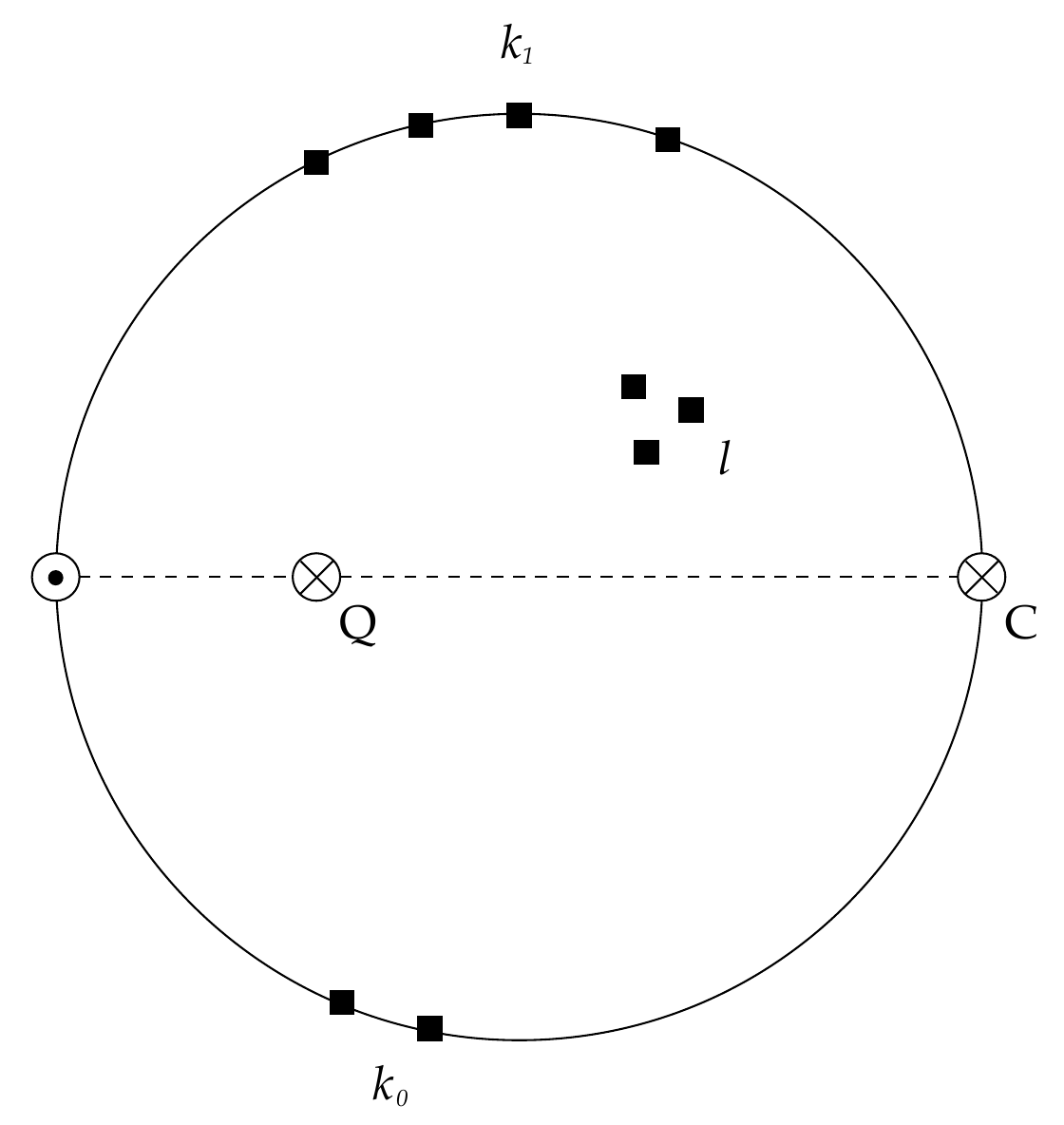}
        \caption{$\kappa:X\times L\rightsquigarrow L$ induced by
        $\cG'_{2+k_0+k_1,1+l}(L,\beta)$ for $k_0,k_1\geq 0$ and $l\geq 0$}
    \end{subfigure}
    \quad 
    \begin{subfigure}[t]{0.4\textwidth}
        \includegraphics[width=\textwidth]{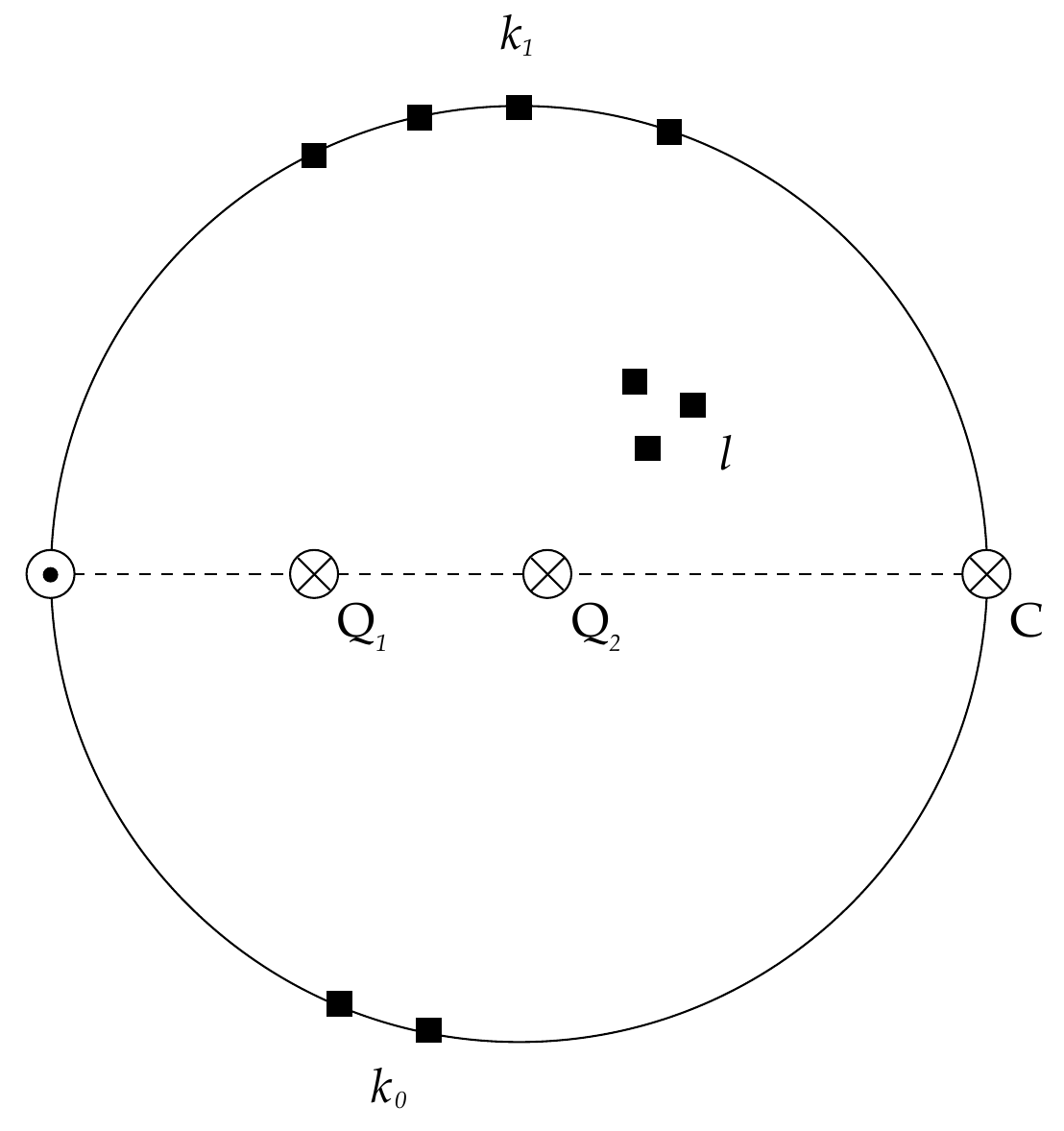}
        \caption{$\kappa:X^2\times L\rightsquigarrow L$ induced by
        $\cG''_{2+k_0+k_1,2+l}(L,\beta)$ for $k_0,k_1\geq 0$ and $l\geq 0$}
    \end{subfigure}
    \caption{Some correspondences induced by moduli of stable $J$-holomorphic spheres/disks with
    geodesic constraints. The interior/boundary marked points denoted $\blacksquare$ are constrained to map to $B/D$.}
    \label{FigCapCorrespondences}
\end{figure}

Recorded below is a degree formula that will be used later for arguing that certain
chains are degenerate.

\begin{lemma}\label{LemmaBMCCapDegree}
If $Q,C, B, D$ are homogeneous chains, then
$$|(Q\cap_{B,D}C)_\beta^{k_0,k_1;l}|=|Q|+|C|-2n+\mu(\beta)+k_0+k_1+2l-(k_0+k_1)(n-|D|)-l(2n-|B|) \quad .$$
\end{lemma}

\subsection{Relation with Maurer-Cartan equation}

A priori, the BMC-deformed quantum cap product $\cap_{B,D}$ is only defined at the chain level.
The following discussion explains why it descends to homology under the assumption
that $D$ satisfies the projective Maurer-Cartan equation, i.e. $D\in\MC^B(L)$.

\begin{definition}\label{DefGeodesicObstruction}
Define:
$$h^+_{B,D}(C,E,Q) = \sum_{l\geq 0}\sum_{k_0\geq 0}\sum_{\substack{k_1^L\geq 0\\ k_1^R\geq 0}}\sum_{\beta\in\pi_2(X,L)}
q^{\omega(\beta)}\frac{1}{l!}h^+_{B,D}(C,E,Q)_\beta^{k_0;k_1^L,k_1^R;l} \quad \textrm{and}$$
$$h^-_{B,D}(C,E,Q) = \sum_{l\geq 0}\sum_{\substack{k_0^L\geq 0\\ k_0^R\geq 0}}\sum_{k_1\geq 0}\sum_{\beta\in\pi_2(X,L)}
q^{\omega(\beta)}\frac{1}{l!}h^-_{B,D}(C,E,Q)_\beta^{k_0^L,k_0^R;k_1;l} \quad ,$$
where the chain $h^+_{B,D}(C,E,Q)_\beta^{k_0;k_1^L,k_1^R;l}=\kappa_{\#}
(D^{k_0}\times C\times D^{k_1^R}\times E\times D^{k_1^L}\times Q\times B^l)$ is obtained
via pull-push from the correspondence
$\kappa$ in
Figure \ref{FigLittleHCorrespondences} (A), and the chain
$h^-_{B,D}(C,E,Q)_\beta^{k_0^L,k_0^R;k_1;l}=\kappa_{\#}
(D^{k_0^L}\times E\times D^{k_0^R} \times C\times D^{k_1}\times Q\times B^l)$ from the one in
Figure \ref{FigLittleHCorrespondences} (B).
\end{definition}

\ytableausetup{boxsize=1em}
\begin{figure}[H]
  \centering
    \begin{subfigure}[t]{0.4\textwidth}
        \includegraphics[width=\textwidth]{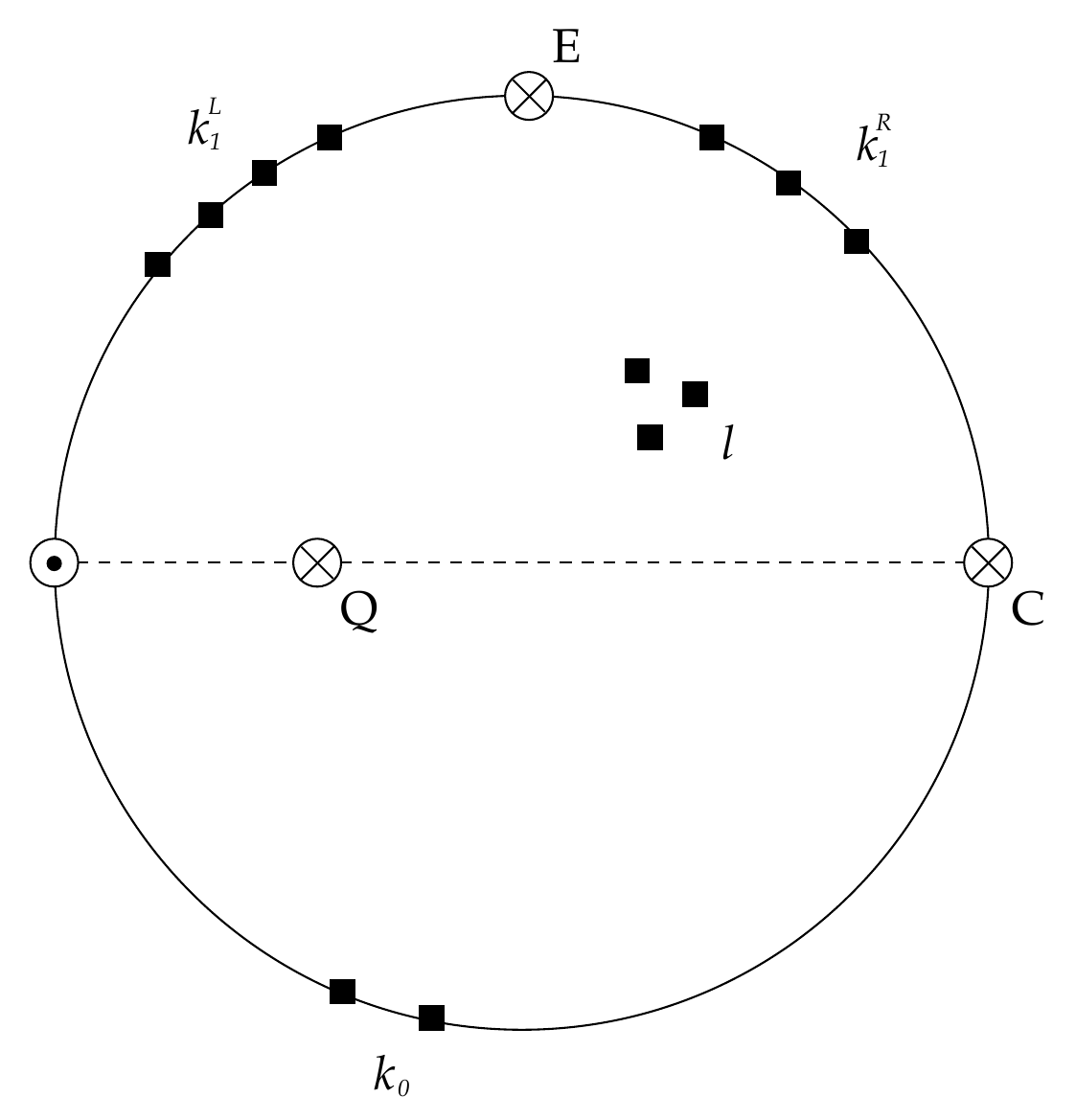}
        \caption{$\kappa:L^2\times X\rightsquigarrow L$ induced by
        $\cG'_{3+k_0+k_1^L+k_1^R,1+l}(L,\beta)$ for $k_0,k_1^L,k_1^R\geq 0$ and $l\geq 0$}
    \end{subfigure}
    \quad 
    \begin{subfigure}[t]{0.4\textwidth}
        \includegraphics[width=\textwidth]{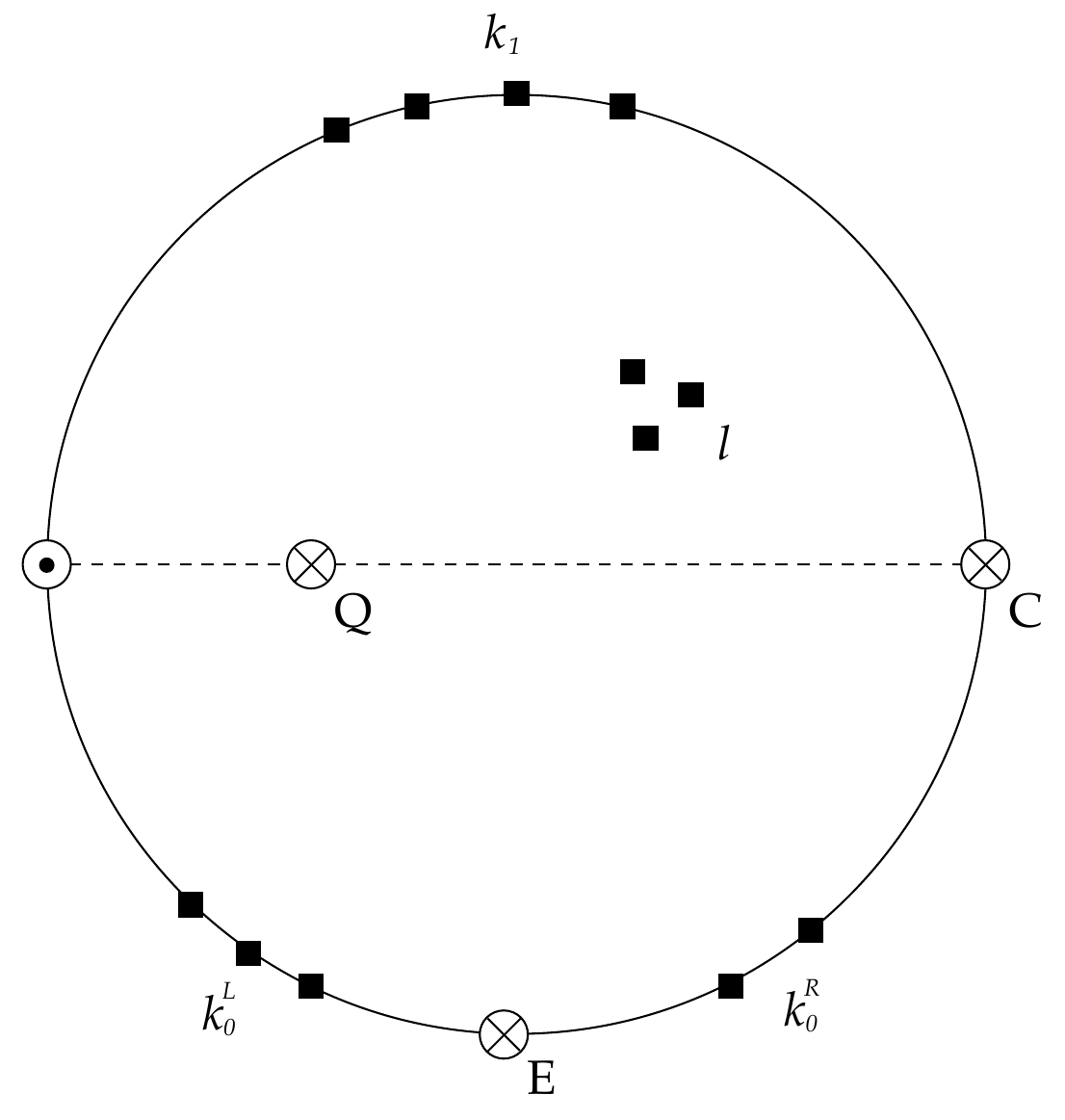}
        \caption{$\kappa:L^2\times X\rightsquigarrow L$ induced by
        $\cG'_{3+k_0^L+k_0^R+k_1,1+l}(L,\beta)$ for $k_0^L,k_0^R,k_1\geq 0$ and $l\geq 0$}
    \end{subfigure}
    \caption{Some correspondences induced by moduli of stable $J$-holomorphic disks with single
    geodesic constraints. The interior/boundary marked points denoted $\blacksquare$ are constrained to map to $B/D$.}
    \label{FigLittleHCorrespondences}
\end{figure}

\begin{lemma}\label{LemmaDegenerateGeodesicObstruction}
If $|E|=n$ then the chains $h^\pm_{B,D}(C,E,Q)$ are degenerate.
\end{lemma}

\begin{proof}
Assume that all chains are homogeneous. One can see that
$$|h^+_{B,D}(C,E,Q)_\beta^{k_0;k_1^L,k_1^R;l}|=n+\mu(\beta)-|C|^\vee -|E|^\vee -|Q|^\vee
-(k_0+k_1^L+k_1^R)(|D|^\vee -1) -l(|B|^\vee -2)+1 \quad .$$
The constraint of having an input boundary point mapping
to $L$ is vacuous, so one has
$$h^+_{B,D}(C,E,Q)=(Q\cap_{B,D}C)_\beta^{k_0,k_1^L+k_1^R;l} \quad .$$
Since the chain on the right has degree one less than what computed above when $|E|=n$ (i.e. $|E|^\vee=0$)
thanks to Lemma \ref{LemmaBMCCapDegree}, the chain on the left
is degenerate. A similar argument applies to $h^-_{B,D}(C,E,Q)$.
\end{proof}

\begin{proposition}\label{PropBMCCapOnHomology}
Up to signs one has
$$m^1_{B,D}(Q\cap_{B,D}C)= (\partial Q)\cap_{B,D} C + Q \cap_{B,D} m^1_{B,D}(C) + h^\pm_{B,D}(C,m^0_{B,D}(1),Q) \quad ,$$
and when $D\in\MC^B(L)$ the BMC-deformed cap product descends to a map $\QH^B(X)\otimes \HF^B(L,D)\to \HF^B(L,D)$.
\end{proposition}

\begin{proof}
In the formula defining the BMC-deformed differential (Definition \ref{DefBMCAinftyMaps} with $k=1$),
separate the classical ($\beta = 0$) and quantum ($\beta\neq 0$) parts:
$$m^1_{B,D}(\cdot)=q^0_{1,0}(\cdot) +
\sum_{k_0,k_1\geq 0}\sum_{l\geq 0}\frac{1}{l!}q^0_{1+k_0+k_1,l}(D^{k_0},\cdot ,D^{k_1},B^l)+$$
$$
\underbracket{
\sum_{\substack{\beta\in\pi_2(X,L)\\ \beta\neq 0}}\sum_{k_0,k_1\geq 0}\sum_{l\geq 0}
T^{\omega(\beta)}\frac{1}{l!}q^\beta_{1+k_0+k_1,l}(D^{k_0},\cdot ,D^{k_1},B^l)
}_{\tilde{m}^1_{B,D}(\cdot)} \quad .$$
The classical part is $q^0_{1,0}=\partial$, and the claim
to be proven can be rewritten as
$$\partial (Q\cap_{B,D}C) =
\underbracket{\tilde{m}^1_{B,D}(Q\cap_{B,D}C)}_{(1)}+
\underbracket{(\partial Q)\cap_{B,D}C + Q\cap_{B,D}(\partial C)}_{(2)}+$$
$$
\underbracket{Q\cap_{B,D}\tilde{m}^1_{B,D}(C)}_{(3)}+
\underbracket{h^+(C,m^0_{B,D}(1),Q)}_{(4)}+
\underbracket{h^-(C,m^0_{B,D}(1),Q)}_{(5)} \quad .$$

The terms (1)-(5) arise from the following degenerations of disks contributing
to $Q\cap_{B,D}C$ (see Figure \ref{FigBMCCapDegenerations}).
\begin{enumerate}[(1)]

\item Using Definition \ref{DefBMCCap} and Definition \ref{DefBMCAinftyMaps}, one can rewrite this term as
$$\sum_{\substack{\beta[1],\beta[2]\\ \beta[1]\neq 0}}\sum_{\substack{k_0[1]\\ k_1[1]}}
\sum_{l[1]}\sum_{\substack{k_0[2]\\ k_1[2]}}\sum_{l[2]}T^{\omega(\beta[1])}T^{\omega(\beta[2])}
\frac{1}{l[1]!l[2]!}$$
$$q^{\beta[1]}_{1+k_0[1]+k_1[1],l[1]}(D^{k_0[1]},(Q\cap_{B,D}C)_{\beta[2]}^{k_0[2],k_1[2];l[2]},
D^{k_1[1]},B^{l[1]}) \quad .$$
Each chain in the sum corresponds to disk bubbling as in Figure
\ref{FigBMCCapDegenerations} (A), where the interior marked point mapping to $Q$ moves along the geodesic constraint
all the way up to the input boundary point mapping to $C$ . If a disk contributes
$T^{\omega(\beta)}(Q\cap_{B,D}C)^{k_0,k_1,l}_\beta$ to $Q\cap_{B,D}C$, then after this kind
of degeneration its class breaks up $\beta=\beta[1]+\beta[2]$, while the
boundary marked points mapping to $D$ split into two groups $k_0=k_0[1]+k_1[2]$, with $k_0[\bullet]$
and $k_1[\bullet]$ respectively on the lower and upper arcs of the two disks. Similarly, the interior
marked points mapping to $B$ split into two groups $l=l[1]+l[2]$. Degenerations inducing different
partitions of the $l$ interior points into groups of size $l[1]$ and $l[2]$ contribute the same chain,
and since there are $l!(l[1]!l[2]!)^{-1}$ such partitions overall one gets the coefficient
$$T^{\omega(\beta)}\frac{1}{l!}\frac{l!}{l[1]!l[2]!}=T^{\omega(\beta[1])}T^{\omega(\beta[2])}\frac{1}{l[1]!l[2]!} \quad .$$

\item These terms correspond to degenerating the condition of mapping marked points to
$Q$ and $D$ to conditions of mapping them to $\partial Q$ and $\partial D$. These are codimension
one thanks to Lemma \ref{LemmaBMCCapDegree}.

\item Using Definition \ref{DefBMCCap} and Definition \ref{DefBMCAinftyMaps}, one can rewrite this term as
$$\sum_{\substack{\beta[1],\beta[2]\\ \beta[1]\neq 0}}\sum_{\substack{k_0[1]\\ k_1[1]}}
\sum_{l[1]}\sum_{\substack{k_0[2]\\ k_1[2]}}\sum_{l[2]}T^{\omega(\beta[1])}T^{\omega(\beta[2])}
\frac{1}{l[1]!l[2]!}$$
$$(Q\cap_{B,D}q^{\beta[1]}_{1+k_0[1]+k_1[1],l[1]}(D^{k_0[1]},C,D^{k_1[1]},B^{l[1]}))_{\beta[2]}^{k_0[2],k_1[2];l[2]} \quad .$$
Each chain in the sum corresponds to disk bubbling as in Figure
\ref{FigBMCCapDegenerations} (B), where the interior marked point mapping to $Q$ moves along the geodesic constraint
all the way up to the output boundary point. If a disk contributes
$T^{\omega(\beta)}(Q\cap_{B,D}C)^{k_0,k_1,l}_\beta$ to $Q\cap_{B,D}C$, then after this kind
of degeneration its class breaks up $\beta=\beta[1]+\beta[2]$, while the
boundary marked points mapping to $D$ split into two groups $k_0=k_0[1]+k_1[2]$, with $k_0[\bullet]$
and $k_1[\bullet]$ respectively on the lower and upper arcs of the two disks. Similarly, the interior
marked points mapping to $B$ split into two groups $l=l[1]+l[2]$. Degenerations inducing different
partitions of the $l$ interior points into groups of size $l[1]$ and $l[2]$ contribute the same chain,
and since there are $l!(l[1]!l[2]!)^{-1}$ such partitions overall one gets the coefficient
$$T^{\omega(\beta)}\frac{1}{l!}\frac{l!}{l[1]!l[2]!}=T^{\omega(\beta[1])}T^{\omega(\beta[2])}\frac{1}{l[1]!l[2]!} \quad .$$

\item Using Definition \ref{DefGeodesicObstruction} and Definition \ref{DefBMCAinftyMaps}, one can rewrite this term as
$$\sum_{\substack{\beta[1]\\ \beta[2]}}\sum_{k[1]}\sum_{k_0[2]}\sum_{\substack{k_1^L[2]\\ k_1^R[2]}}
\sum_{\substack{l[1]\\ l[2]}}T^{\omega(\beta[1])}T^{\omega(\beta[2])}
\frac{1}{l[1]!l[2]!}$$
$$h^+_{B,D}(C,q^{\beta[1]}_{k[1],l[1]}(D^{k[1]},B^{l[1]}),Q)_{\beta[2]}^{k_0[2];k_1^L[2],k_1^R[2];l[2]} \quad .$$
Each chain in the sum corresponds to disk bubbling as in Figure
\ref{FigBMCCapDegenerations} (C), where a group of boundary marked points mapping to $D$ on the upper arc of the disk with
geodesic constraint come together. If a disk contributes
$T^{\omega(\beta)}(Q\cap_{B,D}C)^{k_0,k_1,l}_\beta$ to $Q\cap_{B,D}C$, then after this kind
of degeneration its class breaks up $\beta=\beta[1]+\beta[2]$, while the
boundary marked points mapping to $D$ that are on the upper arc split into three groups $k_1=k_1^L[2]+k[1]+k_1^R[2]$, with $k_1^L[2]$
and $k_1^R[2]$ respectively on the left and right of the attaching point of the disk bubble, while $k[1]$ are on the new disk.
The interior marked points mapping to $B$ split into two groups $l=l[1]+l[2]$. Degenerations inducing different
partitions of the $l$ interior points into groups of size $l[1]$ and $l[2]$ contribute the same chain,
and since there are $l!(l[1]!l[2]!)^{-1}$ such partitions overall one gets the coefficient
$$T^{\omega(\beta)}\frac{1}{l!}\frac{l!}{l[1]!l[2]!}=T^{\omega(\beta[1])}T^{\omega(\beta[2])}\frac{1}{l[1]!l[2]!} \quad .$$

\item Using Definition \ref{DefGeodesicObstruction} and Definition \ref{DefBMCAinftyMaps}, one can rewrite this term as
$$\sum_{\substack{\beta[1]\\ \beta[2]}}\sum_{k[1]}\sum_{\substack{k_0^L[2]\\ k_0^R[2]}}\sum_{k_1[2]}
\sum_{\substack{l[1]\\ l[2]}}T^{\omega(\beta[1])}T^{\omega(\beta[2])}
\frac{1}{l[1]!l[2]!}$$
$$h^-_{B,D}(C,q^{\beta[1]}_{k[1],l[1]}(D^{k[1]},B^{l[1]}),Q)_{\beta[2]}^{k_0^L[2],k_0^R[2];k_1[2];l[2]} \quad .$$
Each chain in the sum corresponds to disk bubbling as in Figure
\ref{FigBMCCapDegenerations} (D), where a group of boundary marked points mapping to $D$ on the lower arc of the disk with
geodesic constraint come together. If a disk contributes
$T^{\omega(\beta)}(Q\cap_{B,D}C)^{k_0,k_1,l}_\beta$ to $Q\cap_{B,D}C$, then after this kind
of degeneration its class breaks up $\beta=\beta[1]+\beta[2]$, while the
boundary marked points mapping to $D$ that are on the upper arc split into three groups $k_0=k_0^L[2]+k[1]+k_0^R[2]$, with $k_0^L[2]$
and $k_0^R[2]$ respectively on the left and right of the attaching point of the disk bubble, while $k[1]$ are on the new disk.
The interior marked points mapping to $B$ split into two groups $l=l[1]+l[2]$. Degenerations inducing different
partitions of the $l$ interior points into groups of size $l[1]$ and $l[2]$ contribute the same chain,
and since there are $l!(l[1]!l[2]!)^{-1}$ such partitions overall one gets the coefficient
$$T^{\omega(\beta)}\frac{1}{l!}\frac{l!}{l[1]!l[2]!}=T^{\omega(\beta[1])}T^{\omega(\beta[2])}\frac{1}{l[1]!l[2]!} \quad .$$

\end{enumerate}

By Gromov compactness, the only other possible degenerations of a disk contributing to $Q\cap_{B,D}C$
consist of sphere bubbling, or disk bubbling arising from moving the interior marked point
mapping to $Q$ all the way up to the upper or lower arcs. None of these degenerations shows
up in $\partial(Q\cap_{B,D}C)$, because the first has codimension two while the second
violates the geodesic constraint. Finally, the BMC-defomred cap product descends to homology
whenever $D\in\MC^B(L)$ because in this case $m^0_{B,D}(1)=W^B(D)L$, and hence the terms
$h^{\pm}_{B,D}(Q,m^0_{B,D}(1),C)$ are degenerate chains thanks to Lemma \ref{LemmaDegenerateGeodesicObstruction}.

\end{proof}

\ytableausetup{boxsize=1em}
\begin{figure}[H]
  \centering
    \begin{subfigure}[t]{0.4\textwidth}
        \includegraphics[width=\textwidth]{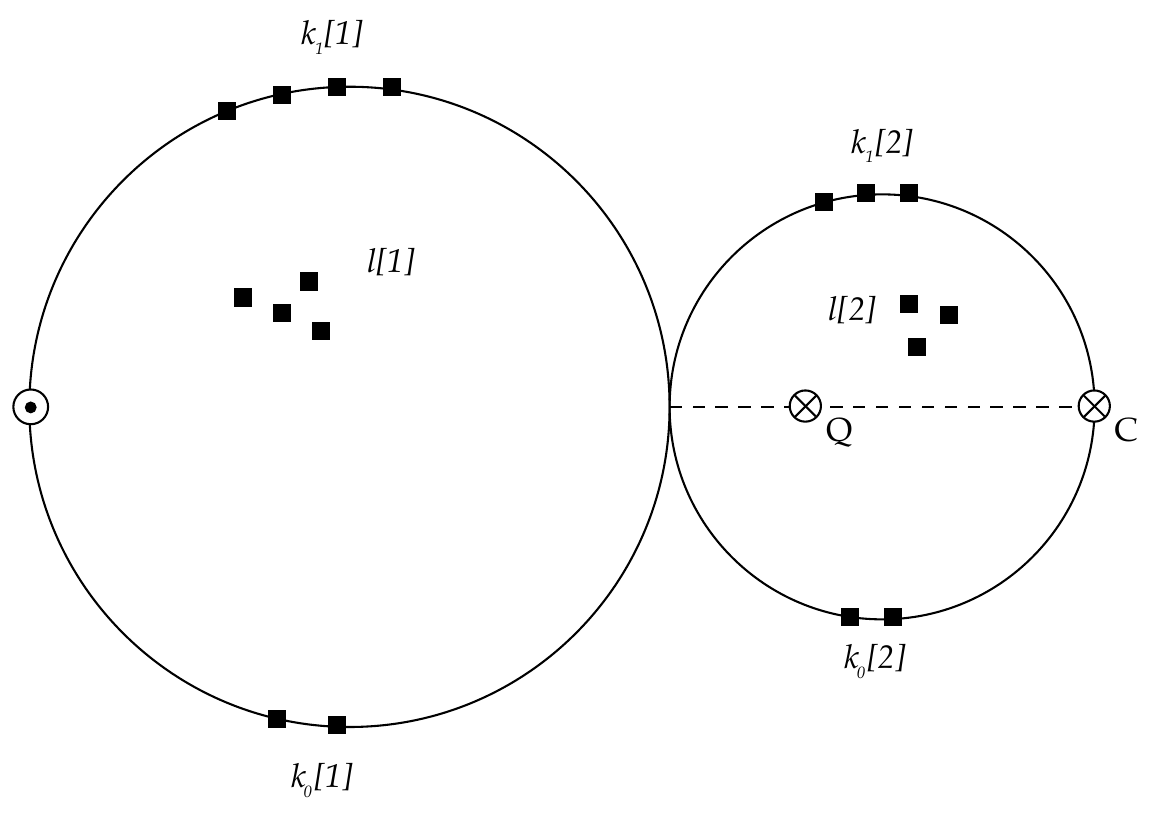}
        \caption{Term (1)}
    \end{subfigure}
    \begin{subfigure}[t]{0.4\textwidth}
        \includegraphics[width=\textwidth]{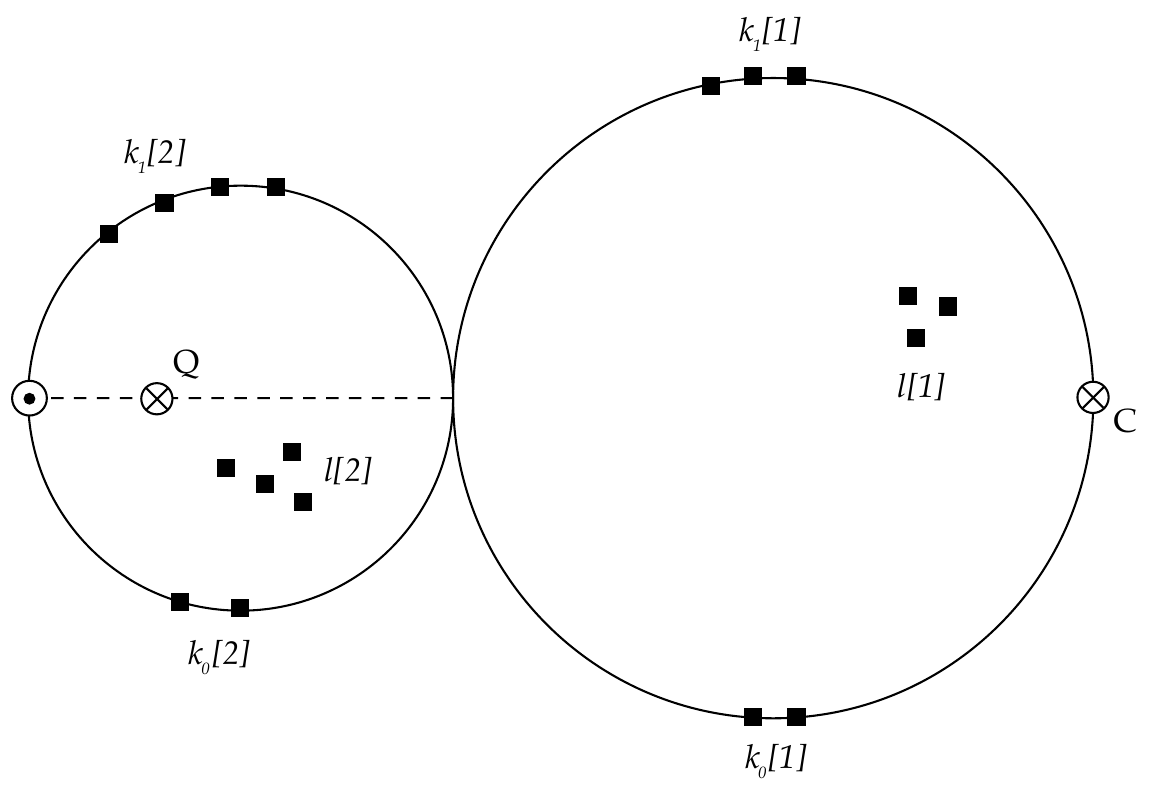}
        \caption{Term (3)}
    \end{subfigure}
    \begin{subfigure}[t]{0.3\textwidth}
        \includegraphics[width=\textwidth]{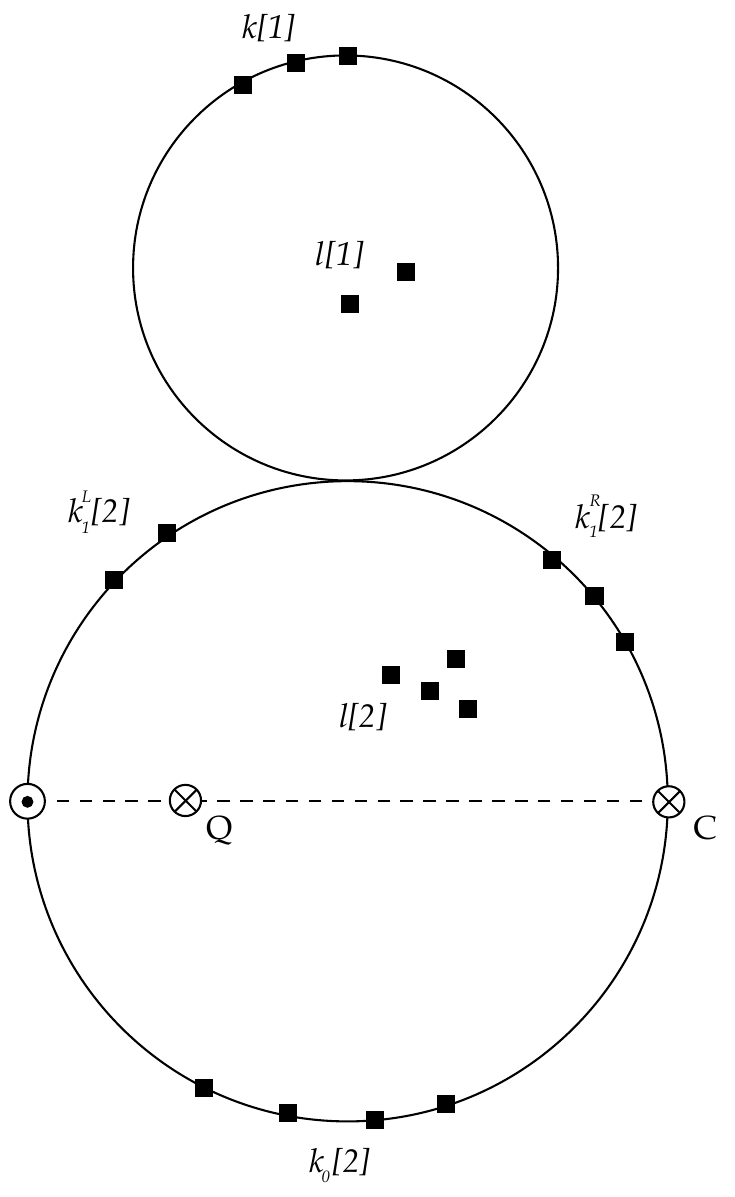}
        \caption{Term (4)}
    \end{subfigure}
    \quad\quad
    \begin{subfigure}[t]{0.3\textwidth}
        \includegraphics[width=\textwidth]{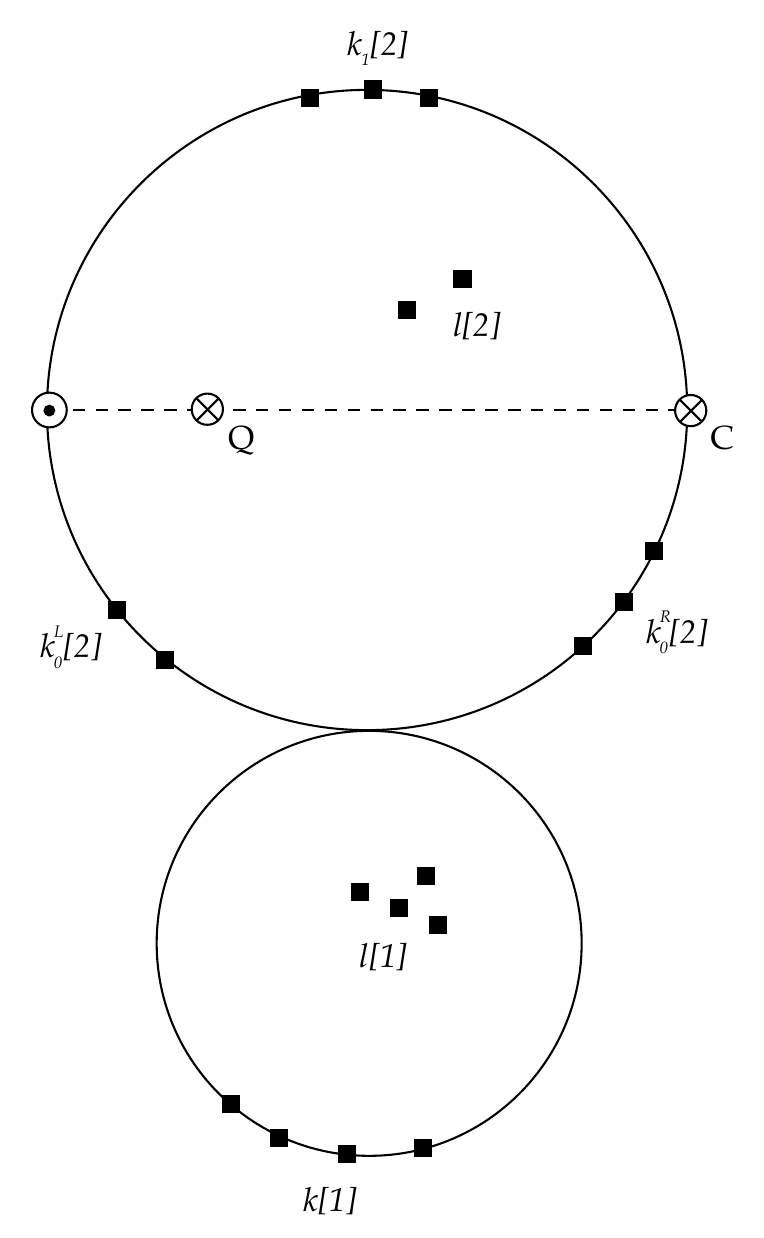}
        \caption{Term (5)}
    \end{subfigure}
    \caption{Degenerations of disks contributing to $Q\cap_{B,D}C$.}
    \label{FigBMCCapDegenerations}
\end{figure}

\section{Module structure}\label{SecBMCmod}

This section explains why the quantum cup product $\cap_{B,D}$ of Section \ref{SecBMCcap}
makes $\HF^B(L,D)$ into a module over $\QH^B(X)$. This relies on the correspondence axiom
(4) from Section \ref{SecAxioms}.

\subsection{Double geodesic constraints}

The following chains will play an important role in checking that $\cap_{B,D}$ indeed
gives a module structure (Proposition \ref{PropModuleStructure}).

\begin{definition}\label{DefBMCCapDouble}
The double BMC-deformed cap product is
$$\Theta_{B,D}(Q_1,Q_2,C)=\sum_{k_0,k_1\geq 0}\sum_{l\geq 0}\sum_{\beta\in\pi_2(X,L)}q^{\omega(\beta)}
\frac{1}{l!}(\Theta_{B,D}(Q_1,Q_2,C))_\beta^{k_0,k_1;l} \quad ,$$
where the chain $(Q\cap_{B,D}C)_\beta^{k_0,k_1;l}=\kappa_{\#}
(D^{k_0}\times C\times D^{k_1}\times Q_1\times Q_2\times B^l)$ is obtained via pull-push from the
correspondence $\kappa$ in
Figure \ref{FigCapCorrespondences} (B).
\end{definition}

\ytableausetup{boxsize=1em}
\begin{figure}[H]
  \centering
    \begin{subfigure}[t]{0.4\textwidth}
        \includegraphics[width=\textwidth]{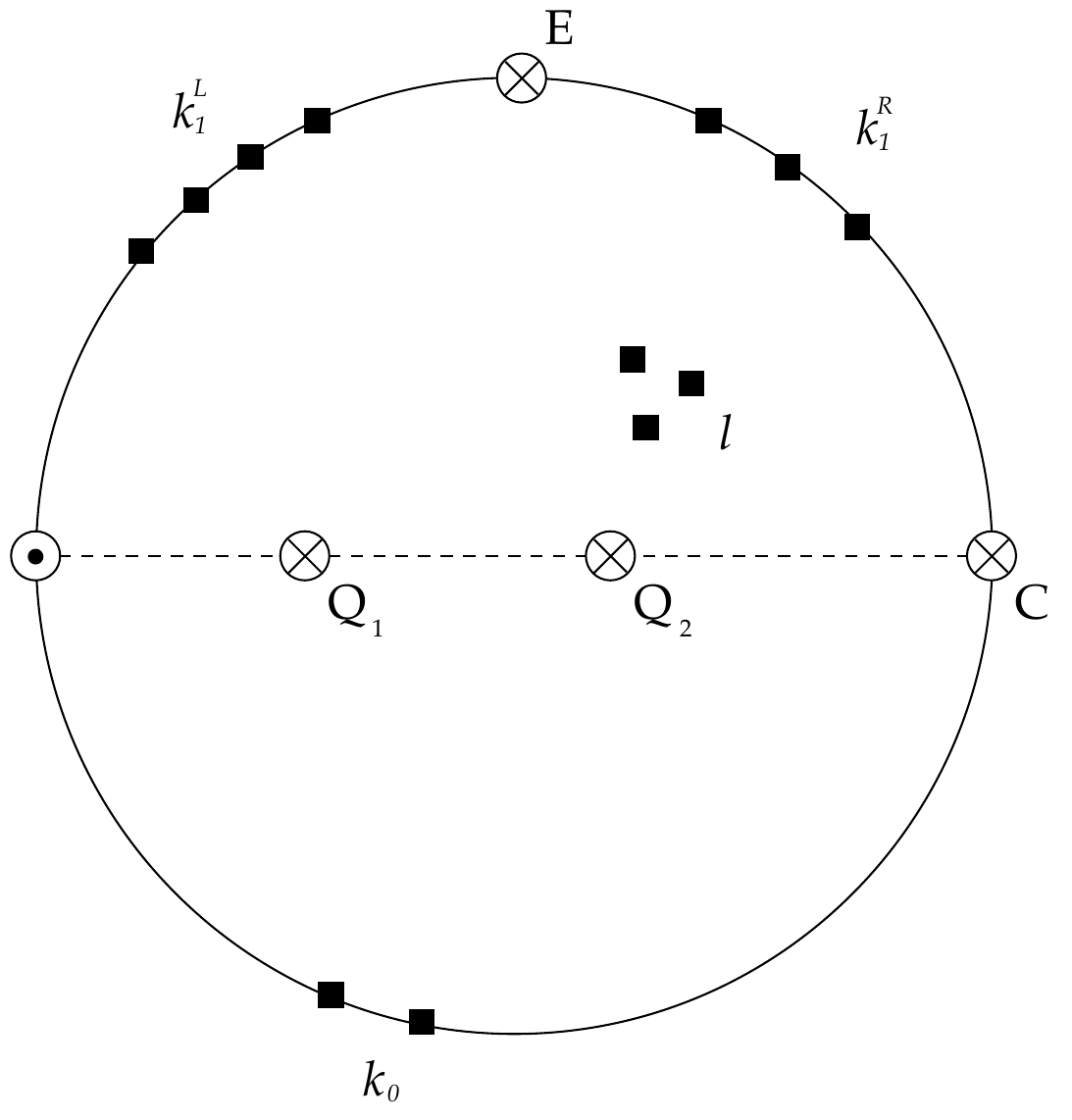}
        \caption{$\kappa:L^2\times X^2\rightsquigarrow L$ induced by
        $\cG''_{3+k_0+k_1^L+k_1^R,2+l}(L,\beta)$ for $k_0,k_1^L,k_1^R\geq 0$ and $l\geq 0$}
    \end{subfigure}
    \quad 
    \begin{subfigure}[t]{0.4\textwidth}
        \includegraphics[width=\textwidth]{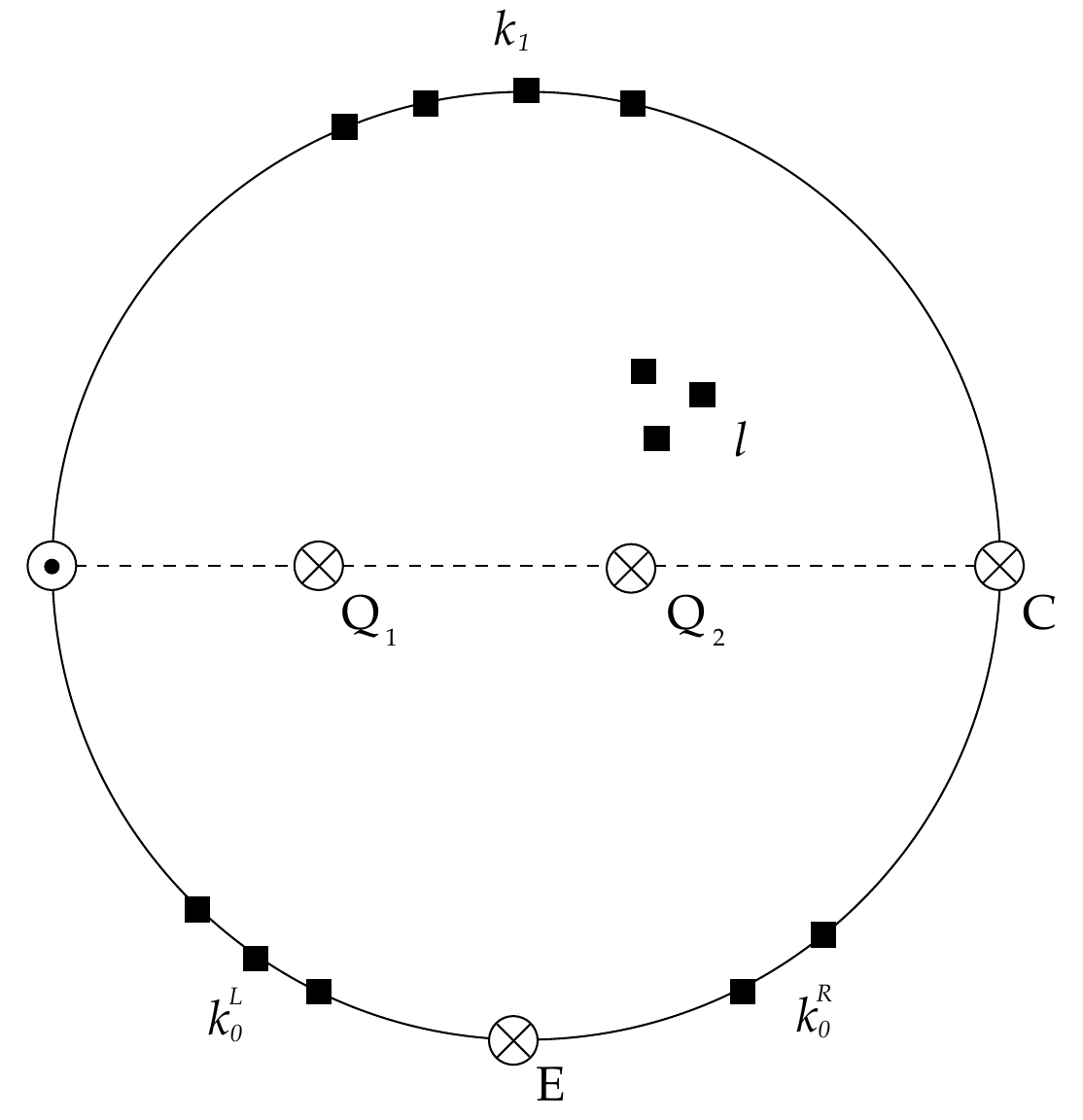}
        \caption{$\kappa:L^2\times X^2\rightsquigarrow L$ induced by
        $\cG''_{3+k_0^L+k_0^R+k_1,2+l}(L,\beta)$ for $k_0^L,k_0^R,k_1\geq 0$ and $l\geq 0$}
    \end{subfigure}
    \caption{Some correspondences induced by moduli of stable $J$-holomorphic disks with double
    geodesic constraints. The interior/boundary marked points denoted $\blacksquare$ are constrained to map to $B/D$.}
    \label{FigBigHCorrespondences}
\end{figure}

\begin{definition}\label{DefGeodesicObstructionDouble}
Define:
$$H^+_{B,D}(C,E,Q_1,Q_2) = \sum_{l\geq 0}\sum_{k_0\geq 0}\sum_{\substack{k_1^L\geq 0\\ k_1^R\geq 0}}\sum_{\beta\in\pi_2(X,L)}
q^{\omega(\beta)}\frac{1}{l!}H^+_{B,D}(C,E,Q_1,Q_2)_\beta^{k_0;k_1^L,k_1^R;l} \quad \textrm{and}$$
$$H^-_{B,D}(C,E,Q_1,Q_2) = \sum_{l\geq 0}\sum_{\substack{k_0^L\geq 0\\ k_0^R\geq 0}}\sum_{k_1\geq 0}\sum_{\beta\in\pi_2(X,L)}
q^{\omega(\beta)}\frac{1}{l!}H^-_{B,D}(C,E,Q_1,Q_2)_\beta^{k_0^L,k_0^R;k_1;l} \quad ,$$
where the chain $H^+_{B,D}(C,E,Q_1,Q_2)_\beta^{k_0;k_1^L,k_1^R;l}=\kappa_{\#}
(D^{k_0}\times C\times D^{k_1^R}\times E\times D^{k_1^L}\times Q_1\times Q_2\times B^l)$ is obtained
via pull-push from the correspondence $\kappa$ in
Figure \ref{FigBigHCorrespondences} (A), and the chain
$H^-_{B,D}(C,E,Q_1,Q_2)_\beta^{k_0^L,k_0^R;k_1;l}=\kappa_{\#}
(D^{k_0^L}\times E\times D^{k_0^R} \times C\times D^{k_1}\times Q_1\times Q_2\times B^l)$ from the one in
Figure \ref{FigBigHCorrespondences} (B).
\end{definition}

\begin{lemma}\label{LemmaBMCCapDoubleDegree}
If $Q_1,Q_2,C, B, D$ are homogeneous chains, then
$$|(\Theta_{B,D}(C,Q_1,Q_2))_\beta^{k_0,k_1^L+k_1^R;l}|=
n+\mu(\beta)-|C|^\vee -|Q_1|^\vee -|Q_2|^\vee
-(k_0+k_1^L+k_1^R)(|D|^\vee -1) -l(|B|^\vee -2)+1 \quad .$$
\end{lemma}

\begin{lemma}\label{LemmaDegenerateGeodesicObstruction}
If $|E|=n$ then the chains $H^\pm_{B,D}(C,E,Q_1,Q_2)$ are degenerate.
\end{lemma}

\begin{proof}
Assume that all chains are homogeneous. One can see that
$$|H^+_{B,D}(C,E,Q_1,Q_2)_\beta^{k_0;k_1^L,k_1^R;l}|=n+\mu(\beta)-|C|^\vee -|E|^\vee -|Q_1|^\vee -|Q_2|^\vee
-(k_0+k_1^L+k_1^R)(|D|^\vee -1) -l(|B|^\vee -2)+2 \quad .$$
The constraint of having an input boundary point mapping
to $L$ is vacuous, so one has
$$H^+_{B,D}(C,E,Q_1,Q_2)=(\Theta_{B,D}(C,Q_1,Q_2))_\beta^{k_0,k_1^L+k_1^R;l} \quad .$$
Since the chain on the right has degree one less than what computed above when $|E|=n$ (i.e. $|E|^\vee=0$)
thanks to Lemma \ref{LemmaBMCCapDoubleDegree},
the chain on the left is degenerate. A similar argument applies to $H^-_{B,D}(C,E,Q_1,Q_2)$.
\end{proof}

\subsection{Degeneration to single geodesic constraints}

The following discussion explains how the double geodesic constraints introduced earlier
degenerate to single geodesic constraints entering the definition of BMC-deformed
quantum cap $\cap_{B,D}$, and the relation with the module structure.

\begin{lemma}\label{LemmaBoundaryBMCCapDouble}
Up to signs and degenerate chains one has
$$m^1_{B,D}(\Theta_{B,D}(Q_1,Q_2,C))=\Theta_{B,D}(Q_1,Q_2,m^1_{B,D}(C))+$$
$$Q_1\cap_{B,D}(Q_2\cap_{B,D}C)
+(Q_1\star_B Q_2)\cap_{B,D}C + H_{B,D}^\pm(C,m^0_{B,D}(1),Q_1,Q_2)\quad .$$
\end{lemma}

\begin{proof}

As in the proof of Proposition \ref{PropBMCCapOnHomology}, separate the classical
and quantum parts $m^1_{B,D}=\partial + \tilde{m}^1_{B,D}$. The claim to be proven can
then be rewritten as

$$
\partial(\Theta_{B,D}(Q_1,Q_2,C))=
\underbracket{\tilde{m}^1_{B,D}(\Theta_{B,D}(Q_1,Q_2,C))}_{(1)}+
\underbracket{\Theta_{B,D}(Q_1,Q_2,\partial C)}_{(2)}+
$$
$$
\underbracket{\Theta_{B,D}(Q_1,Q_2,\tilde{m}^1_{B,D}(C)}_{(3)}
+
\underbracket{Q_1\cap_{B,D}(Q_2\cap_{B,D}C)}_{(4)}+
\underbracket{(Q_1\star_B Q_2)\cap_{B,D}C}_{(5)}+
$$
$$
\underbracket{H^+_{B,D}(C,m^0_{B,D}(1),Q_1,Q_2)}_{(6)}+
\underbracket{H^-_{B,D}(C,m^0_{B,D}(1),Q_1,Q_2)}_{(7)} \quad .
$$

The terms above arise from the following degenerations of disks contributing to
$\Theta_{B,D}(Q_1,Q_2,C)$ (see Figure \ref{FigBMCCapDoubleDegenerations}).
\begin{enumerate}[(1)]

\item Using Definition \ref{DefBMCCapDouble} and Definition \ref{DefBMCAinftyMaps},
one can rewrite this term as

$$\sum_{\substack{k_0[1]\\ k_1[1]}}\sum_{l[1]}\sum_{\substack{k_0[2]\\ k_1[2]}}\sum_{l[2]}
\sum_{\substack{\beta[1]\neq 0\\ \beta[2]}}T^{\omega(\beta[1])}T^{\omega(\beta[2])}\frac{1}{l[1]!l[2]!}$$
$$
q^{\beta[1]}_{1+k_0[1]+k_1[1],l[1]}(D^{k_0[1]},
(\Theta_{B,D}(Q_1,Q_2,C))_{\beta[2]}^{k_0[2],k_1[2];l[2]},D^{k_1[1]},B^{l[1]}) \quad .$$

Each chain in the sum corresponds to disk bubbling as in Figure
\ref{FigBMCCapDoubleDegenerations} (A), where the interior marked points mapping to $Q_1,Q_2$ move along the geodesic constraint
all the way up to the input boundary point mapping to $C$ . If a disk contributes
$T^{\omega(\beta)}(\Theta_{B,D}(Q_1,Q_2,C))^{k_0,k_1,l}_\beta$ to $\Theta_{B,D}(Q_1,Q_2,C)$, then after this kind
of degeneration its class breaks up $\beta=\beta[1]+\beta[2]$, while the
boundary marked points mapping to $D$ split into two groups $k_0=k_0[1]+k_1[2]$, with $k_0[\bullet]$
and $k_1[\bullet]$ respectively on the lower and upper arcs of the two disks. Similarly, the interior
marked points mapping to $B$ split into two groups $l=l[1]+l[2]$. Degenerations inducing different
partitions of the $l$ interior points into groups of size $l[1]$ and $l[2]$ contribute the same chain,
and since there are $l!(l[1]!l[2]!)^{-1}$ such partitions overall one gets the coefficient
$$T^{\omega(\beta)}\frac{1}{l!}\frac{l!}{l[1]!l[2]!}=T^{\omega(\beta[1])}T^{\omega(\beta[2])}\frac{1}{l[1]!l[2]!} \quad .$$

\item This term corresponds to degenerating the condition of mapping a marked point to
$C$ to the condition of mapping it to $\partial C$. This is codimension
one thanks to Lemma \ref{LemmaBMCCapDoubleDegree}.

\item Using Definition \ref{DefBMCCapDouble} and Definition \ref{DefBMCAinftyMaps},
one can rewrite this term as

$$\sum_{\substack{k_0[1]\\ k_1[1]}}\sum_{l[1]}\sum_{\substack{k_0[2]\\ k_1[2]}}\sum_{l[2]}
\sum_{\substack{\beta[1]\\ \beta[2]\neq 0}}T^{\omega(\beta[1])}T^{\omega(\beta[2])}\frac{1}{l[1]!l[2]!}$$
$$\left( \Theta_{B,D}(Q_1,Q_2,q^{\beta[2]}_{1+k_0[2]+k_1[2],l[2]}
(D^{k_0[2]},C,D^{k_1[2]},B^{l[2]})\right)_{\beta[1]}^{k_0[1],k_1[1];l[1]} \quad .$$

Each chain in the sum corresponds to disk bubbling as in Figure
\ref{FigBMCCapDoubleDegenerations} (B), where the interior marked points mapping to $Q_1,Q_2$ move along the geodesic constraint
all the way up to the output boundary point. If a disk contributes
$T^{\omega(\beta)}(\Theta_{B,D}(Q_1,Q_2,C))^{k_0,k_1,l}_\beta$ to $\Theta_{B,D}(Q_1,Q_2,C)$, then after this kind
of degeneration its class breaks up $\beta=\beta[1]+\beta[2]$, while the
boundary marked points mapping to $D$ split into two groups $k_0=k_0[1]+k_1[2]$, with $k_0[\bullet]$
and $k_1[\bullet]$ respectively on the lower and upper arcs of the two disks. Similarly, the interior
marked points mapping to $B$ split into two groups $l=l[1]+l[2]$. Degenerations inducing different
partitions of the $l$ interior points into groups of size $l[1]$ and $l[2]$ contribute the same chain,
and since there are $l!(l[1]!l[2]!)^{-1}$ such partitions overall one gets the coefficient
$$T^{\omega(\beta)}\frac{1}{l!}\frac{l!}{l[1]!l[2]!}=T^{\omega(\beta[1])}T^{\omega(\beta[2])}\frac{1}{l[1]!l[2]!} \quad .$$

\item Using Definition \ref{DefBMCCap}, one can rewrite this term as

$$\sum_{\substack{k_0[1]\\ k_1[1]}}\sum_{l[1]}\sum_{\substack{k_0[2]\\ k_1[2]}}\sum_{l[2]}
\sum_{\substack{\beta[1]\\ \beta[2]}}T^{\omega(\beta[1])}T^{\omega(\beta[2])}\frac{1}{l[1]!l[2]!}$$
$$\left( Q_1\cap_{B,D} (Q_2\cap_{B,D} C)_{\beta[2]}^{k_0[2],k_1[2];l[2]}\right)_{\beta[1]}^{k_0[1],k_1[1];l[1]} \quad .$$

Each chain in the sum corresponds to disk bubbling as in Figure
\ref{FigBMCCapDoubleDegenerations} (C), where one of the interior marked points mapping to $Q_1,Q_2$ move along the geodesic constraint
all the way up to the nearest boundary point. If a disk contributes
$T^{\omega(\beta)}(\Theta_{B,D}(Q_1,Q_2,C))^{k_0,k_1,l}_\beta$ to $\Theta_{B,D}(Q_1,Q_2,C)$, then after this kind
of degeneration its class breaks up $\beta=\beta[1]+\beta[2]$, while the
boundary marked points mapping to $D$ split into two groups $k_0=k_0[1]+k_1[2]$, with $k_0[\bullet]$
and $k_1[\bullet]$ respectively on the lower and upper arcs of the two disks. Similarly, the interior
marked points mapping to $B$ split into two groups $l=l[1]+l[2]$. Degenerations inducing different
partitions of the $l$ interior points into groups of size $l[1]$ and $l[2]$ contribute the same chain,
and since there are $l!(l[1]!l[2]!)^{-1}$ such partitions overall one gets the coefficient
$$T^{\omega(\beta)}\frac{1}{l!}\frac{l!}{l[1]!l[2]!}=T^{\omega(\beta[1])}T^{\omega(\beta[2])}\frac{1}{l[1]!l[2]!} \quad .$$

\item Using Definition \ref{DefBCup} and Definition \ref{DefBMCCap}, one can rewrite this term as

$$\sum_{\substack{k_0[2]\\ k_1[2]}}\sum_{\substack{l[1]\\ l[2]}}\sum_{\substack{\beta[1]\\ \beta[2]}}
T^{\omega(\beta[1])}T^{\omega(\beta[2])}\frac{1}{l[1]!l[2]!}
\left( (Q_1\star_B Q_2)_{\beta[1]}^{l[1]}\cap_{B,D} C \right)_{\beta[2]}^{k_0[2],k_1[2];l[2]} \quad .$$

Each chain in the sum corresponds to sphere bubbling as in Figure \ref{FigBMCCapDoubleDegenerations} (D), where the two
interior marked points mapping to $Q_1,Q_2$ come together along the geodesic constraint.
If a disk contributes
$T^{\omega(\beta)}(\Theta_{B,D}(Q_1,Q_2,C))^{k_0[2],k_1[2],l}_\beta$ to $\Theta_{B,D}(Q_1,Q_2,C)$, then after this kind
of degeneration its class breaks up $\beta=\beta[1]+\beta[2]$ with $\beta[1]\in\pi_2(X)$ and
$\beta[2]\in\pi_2(X,L)$, while the interior
marked points mapping to $B$ split into two groups $l=l[1]+l[2]$. Degenerations inducing different
partitions of the $l$ interior points into groups of size $l[1]$ and $l[2]$ contribute the same chain,
and since there are $l!(l[1]!l[2]!)^{-1}$ such partitions overall one gets the coefficient
$$T^{\omega(\beta)}\frac{1}{l!}\frac{l!}{l[1]!l[2]!}=T^{\omega(\beta[1])}T^{\omega(\beta[2])}\frac{1}{l[1]!l[2]!} \quad .$$

\item Using Definition \ref{DefBMCAinftyMaps} with $k=0$ and Definition \ref{DefGeodesicObstructionDouble},
one can rewrite this term as

$$\sum_{\substack{l[1]\\ l[2]}}\sum_{\substack{k_0[1]\\ k_0[2]}}\sum_{\substack{k_1^L[1]\\ k_1^R[1]}}
\sum_{\substack{\beta[1]\\ \beta[2]}}T^{\omega(\beta[1])}T^{\omega(\beta[2])}\frac{1}{l[1]!l[2]!}
H^+_{B,D}(C,q^{\beta[2]}_{k_0[2],l[2]}(D^{k_0[2]},B^{l[2]}),Q_1,Q_2)_{\beta[1]}^{k_0[1];k_1^L[1],k_1^R[1];l[1]} \quad .$$

Each chain in the sum corresponds to disk bubbling as in Figure \ref{FigBMCCapDoubleDegenerations} (E), where a group of boundary
marked points mapping to $D$ on the upper arc of the disk with geodesic constraint come together.
If a disk contributes $T^{\omega(\beta)}(\Theta_{B,D}(Q_1,Q_2,C))^{k_0[2],k_1[2],l}_\beta$ to $\Theta_{B,D}(Q_1,Q_2,C)$, then after this kind
of degeneration its class breaks up $\beta=\beta[1]+\beta[2]$, while the
boundary marked points mapping to $D$ that are on the upper arc split into three groups $k_1=k_1^L[2]+k[1]+k_1^R[2]$, with $k_1^L[2]$
and $k_1^R[2]$ respectively on the left and right of the attaching point of the disk bubble, while $k[1]$ are on the new disk.
The interior marked points mapping to $B$ split into two groups $l=l[1]+l[2]$. Degenerations inducing different
partitions of the $l$ interior points into groups of size $l[1]$ and $l[2]$ contribute the same chain,
and since there are $l!(l[1]!l[2]!)^{-1}$ such partitions overall one gets the coefficient
$$T^{\omega(\beta)}\frac{1}{l!}\frac{l!}{l[1]!l[2]!}=T^{\omega(\beta[1])}T^{\omega(\beta[2])}\frac{1}{l[1]!l[2]!} \quad .$$
 
\item Using Definition \ref{DefBMCAinftyMaps} with $k=0$ and Definition \ref{DefGeodesicObstructionDouble},
one can rewrite this term as

$$\sum_{\substack{l[1]\\ l[2]}}\sum_{\substack{k_0^L[1]\\ k_0^R[1]}}\sum_{\substack{k_1[1]]\\ k_0[2]}}
\sum_{\substack{\beta[1]\\ \beta[2]}}T^{\omega(\beta[1])}T^{\omega(\beta[2])}\frac{1}{l[1]!l[2]!}
H^-_{B,D}(C,q^{\beta[2]}_{k_0[2],l[2]}(D^{k_0[2]},B^{l[2]}),Q_1,Q_2)_{\beta[1]}^{k_0^L[1],k_0^R[1];k_1[1];l[1]} \quad .$$

Each chain in the sum corresponds to disk bubbling as in Figure \ref{FigBMCCapDoubleDegenerations} (F), where a group of boundary
marked points mapping to $D$ on the lower arc of the disk with geodesic constraint come together.
If a disk contributes $T^{\omega(\beta)}(\Theta_{B,D}(Q_1,Q_2,C))^{k_0[2],k_1[2],l}_\beta$ to $\Theta_{B,D}(Q_1,Q_2,C)$, then after this kind
of degeneration its class breaks up $\beta=\beta[1]+\beta[2]$, while the
boundary marked points mapping to $D$ that are on the lower arc split into three groups $k_0=k_0^L[1]+k[2]+k_0^R[2]$, with $k_0^L[1]$
and $k_0^R[1]$ respectively on the left and right of the attaching point of the disk bubble, while $k[2]$ are on the new disk.
The interior marked points mapping to $B$ split into two groups $l=l[1]+l[2]$. Degenerations inducing different
partitions of the $l$ interior points into groups of size $l[1]$ and $l[2]$ contribute the same chain,
and since there are $l!(l[1]!l[2]!)^{-1}$ such partitions overall one gets the coefficient
$$T^{\omega(\beta)}\frac{1}{l!}\frac{l!}{l[1]!l[2]!}=T^{\omega(\beta[1])}T^{\omega(\beta[2])}\frac{1}{l[1]!l[2]!} \quad .$$

\end{enumerate}

By Gromov compactness, the only other possible degenerations of a disk contributing to $\Theta_{B,D}(Q_1,Q_2,C)$
consist of sphere bubbling, or disk bubbling arising from moving the interior marked points
mapping to $Q_1,Q_2$ all the way up to the upper or lower arcs. None of these degenerations shows
up in $\partial(\Theta_{B,D}(Q_1,Q_2,C))$, because the first has codimension two while the second
violates the geodesic constraint.

\end{proof}

\begin{proposition}\label{PropModuleStructure}
When $D\in\MC^B(L)$ the following properties hold:
\begin{enumerate} 
	\item $[X]\cap_{B,D}[C] = [C]$ ;
	\item $[Q_1]\cap_{B,D}([Q_2]\cap_{B,D}[C]) = ([Q_1]\star_B[Q_2])\cap_{B,D}[C]$ .
\end{enumerate}
\end{proposition}

\begin{proof}

For part (1), separating the classical ($\beta=0$) and quantum ($\beta\neq 0$)
parts of the BMC-deformed cap product in Definition \ref{DefBMCCap} one gets:

$$X\cap_{B,D}C = \sum_{k_0,k_1\geq 0}\sum_{l\geq 0}T^{\omega(\beta)}\frac{1}{l!}(X\cap_{B,D}C)_0^{k_0,k_1;l}
+ \sum_{k_0,k_1\geq 0}\sum_{l\geq 0}\sum_{\substack{\beta\in\pi_2(X,L)\\ \beta\neq 0}}
T^{\omega(\beta)}\frac{1}{l!}(X\cap_{B,D}C)_\beta^{k_0,k_1;l} \quad .$$

In the classical part, $(X\cap_{B,D}C)_0^{k_0,k_1;l}=0$ for $(k_0,k_1,l)\neq (0,0,0)$ and
$(X\cap_{B,D}C)_0^{0,0;0}=X\cap C = C$. Regarding the quantum
part, observe that the constraint of mapping an interior point to $X$ is vacuous, hence
removing it one gets

$$(X\cap_{B,D}C)_\beta^{k_0,k_1;l} = q^\beta_{1+k_0+k_1,l}(D^{k_0},C,D^{k_1},B^l) \quad \textrm{for any} \; \beta\neq 0\quad .$$

When all chains are homogeneous, from Lemma \ref{LemmaBMCCapDegree}
one gets
$$|(X\cap_{B,D}C)_\beta^{k_0,k_1;l}|=|q^\beta_{1+k_0+k_1,l}(D^{k_0},C,D^{k_1},B^l)|+1 \quad ,$$
hence the chain on the left is degenerate for all $\beta\neq 0$.
Part (2) follows from the identity

$$m^1_{B,D}(\Theta_{B,D}(Q_1,Q_2,C))=\Theta_{B,D}(Q_1,Q_2,m^1_{B,D}(C))+$$
$$Q_1\cap_{B,D}(Q_2\cap_{B,D}C)
+(Q_1\star_B Q_2)\cap_{B,D}C + H_{B,D}^\pm(C,m^0_{B,D}(1),Q_1,Q_2)\quad .$$

proved in Lemma \ref{LemmaBoundaryBMCCapDouble} and Lemma \ref{LemmaDegenerateGeodesicObstruction}, which can be used because
$m^0_{B,D}(1)=W^B(D)L$ thanks to the assumption $D\in\MC^B(L)$.
\end{proof}

\ytableausetup{boxsize=1em}
\begin{figure}[H]
  \centering
    \begin{subfigure}[t]{0.3\textwidth}
        \includegraphics[width=\textwidth]{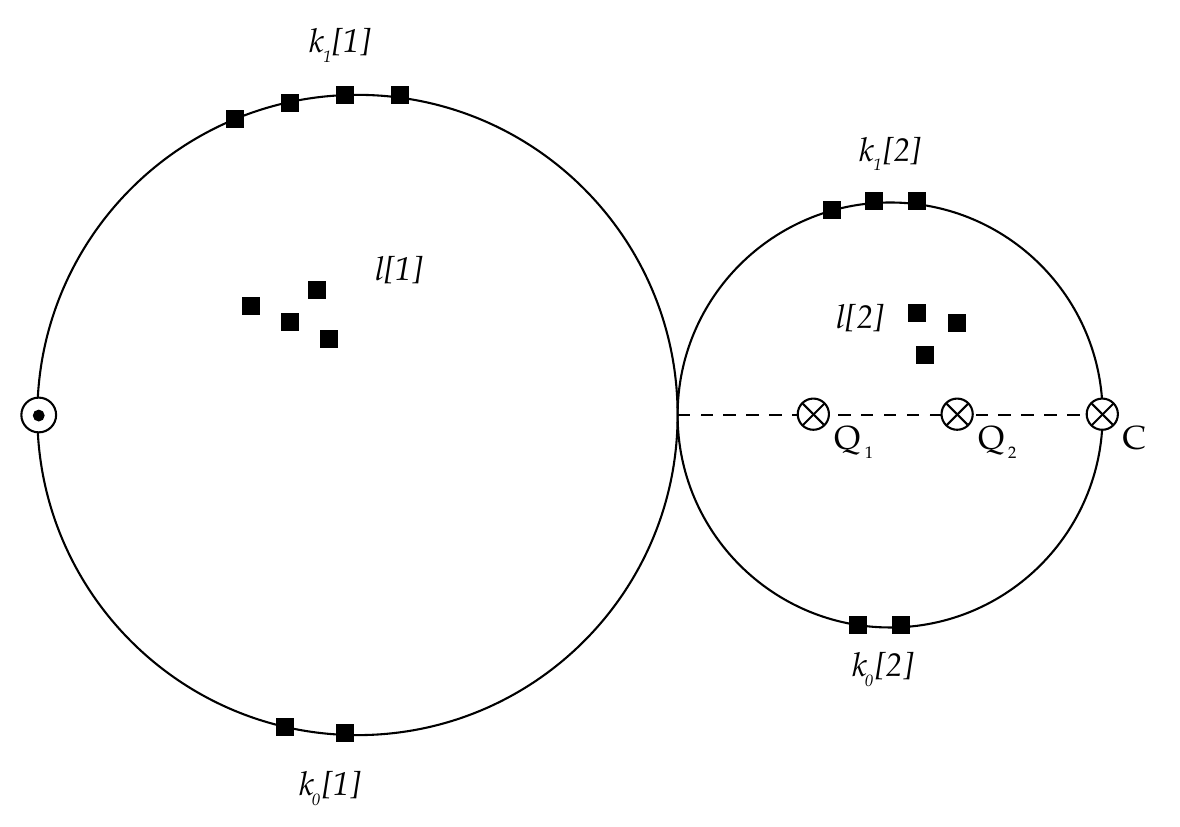}
        \caption{Term (1)}
    \end{subfigure}
    \begin{subfigure}[t]{0.3\textwidth}
        \includegraphics[width=\textwidth]{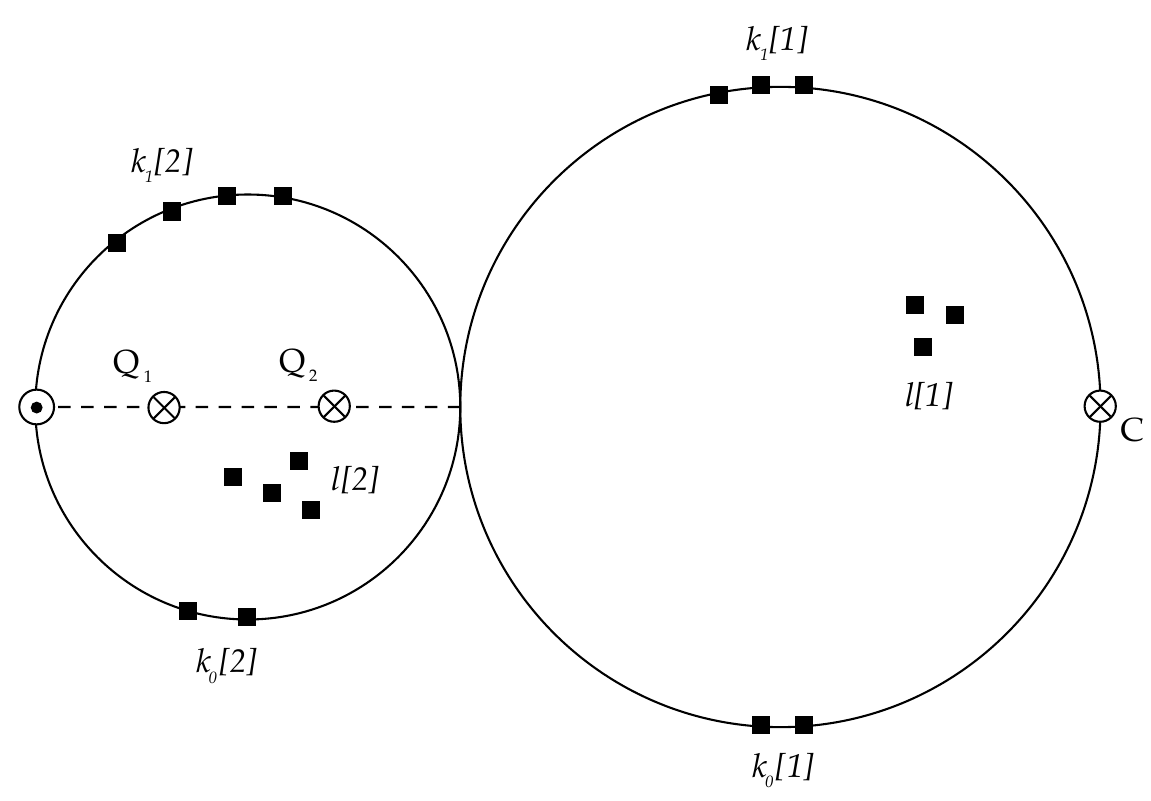}
        \caption{Term (3)}
    \end{subfigure}
    \begin{subfigure}[t]{0.3\textwidth}
        \includegraphics[width=\textwidth]{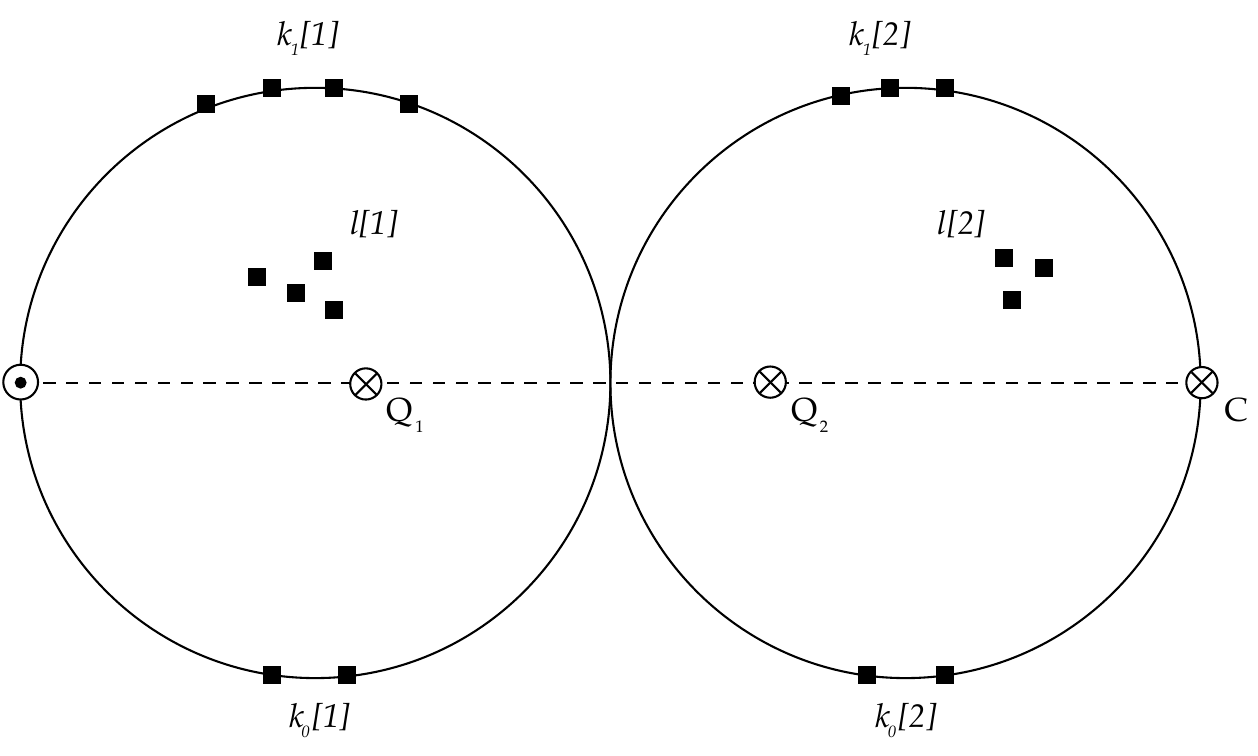}
        \caption{Term (4)}
    \end{subfigure}\\
    \quad\quad
    \begin{subfigure}[t]{0.3\textwidth}
        \includegraphics[width=\textwidth]{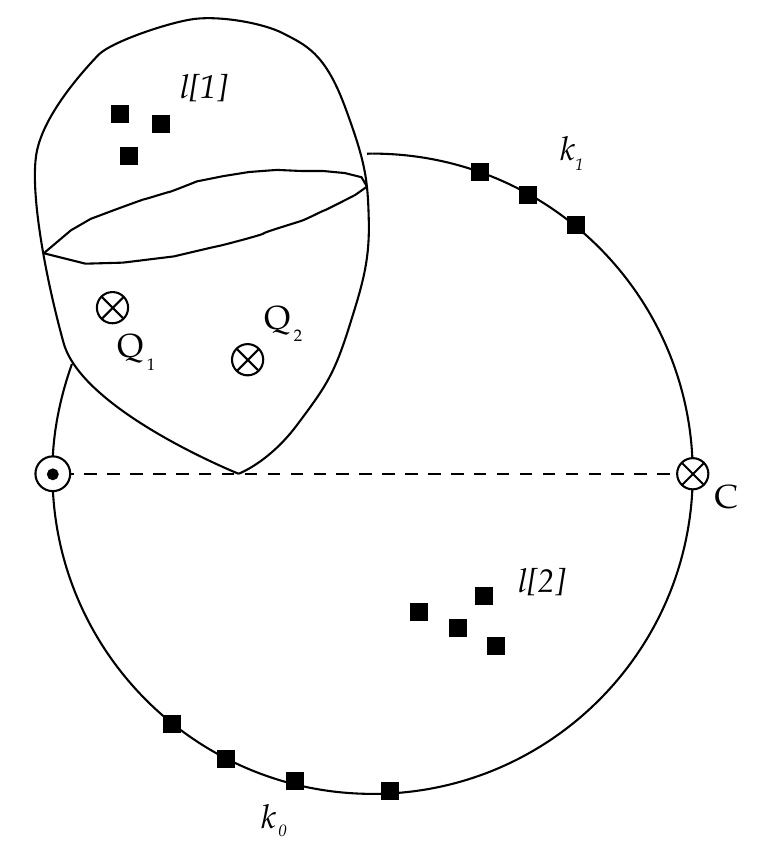}
        \caption{Term (5)}
    \end{subfigure}
    \begin{subfigure}[t]{0.3\textwidth}
        \includegraphics[width=\textwidth]{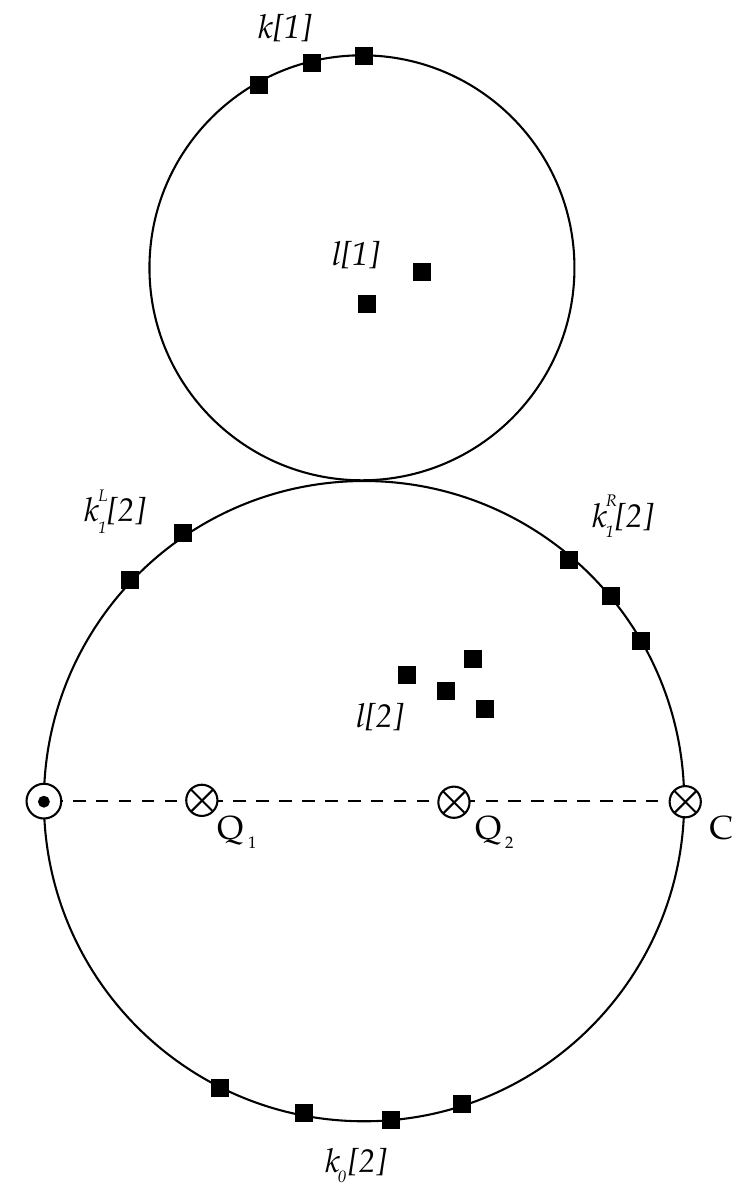}
        \caption{Term (6)}
    \end{subfigure}
    \begin{subfigure}[t]{0.3\textwidth}
        \includegraphics[width=\textwidth]{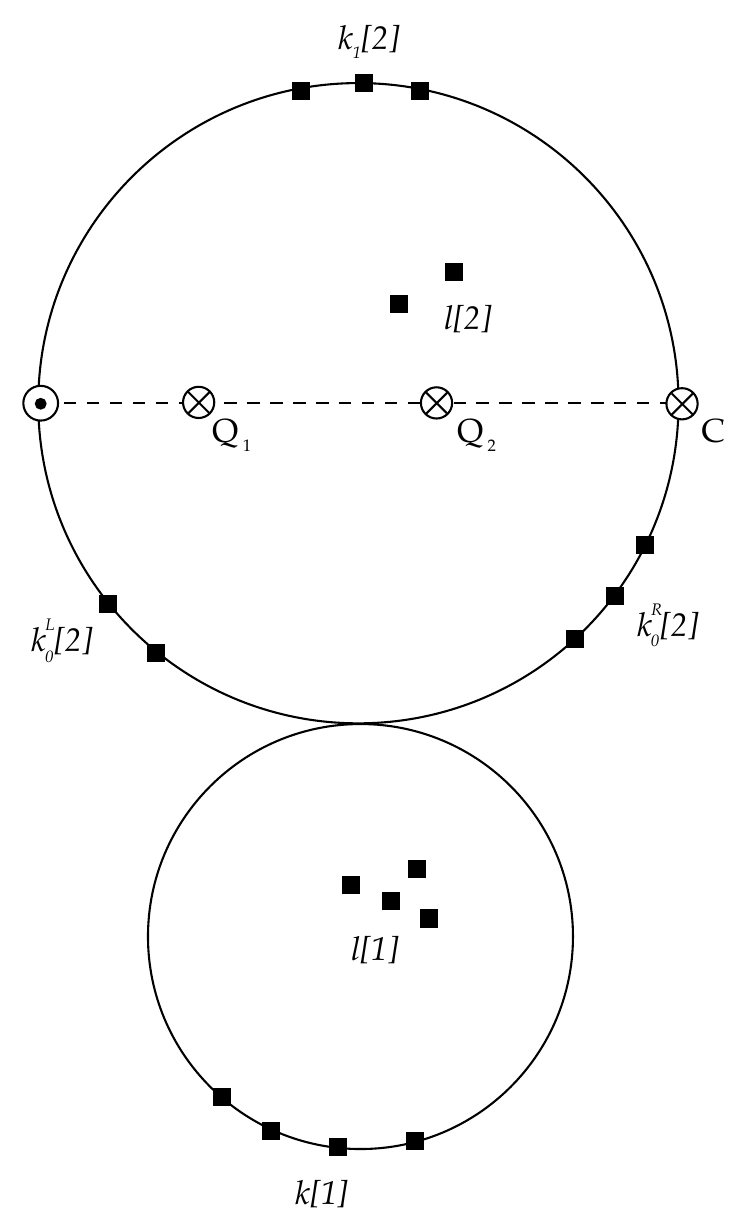}
        \caption{Term (7)}
    \end{subfigure}
    \caption{Degenerations of disks contributing to $\Theta_{B,D}(Q_1,Q_2,C)$.}
    \label{FigBMCCapDoubleDegenerations}
\end{figure}

\section{Curvatures are in the Dubrovin spectrum}\label{SecSpectrum}

This section contains the main result of the article, showing that the curvature of
the $A_\infty$ algebra $\CF^B(L,D)$ is in the spectrum of Dubrovin's operator $K^B$.
The proof makes use of the forgetful axioms (5)-(6) from Section \ref{SecAxioms}.

\begin{notation}
Decompose the chains $D=\sum_iD_i$ and $B=\sum_jB_j$ into homogeneous components,
with indices corresponding to cohomological degrees: $i=n-|D_i|$ and $j=2n-|B_j|$. Denote
$\bi(\bullet)=(i_1,\ldots ,i_\bullet)$ and $\bj(\bullet)=(j_1,\ldots ,j_\bullet)$
unspecified sequences of length $\bullet$ in the set of indices appearing in the sums above.
For a fixed length $\bullet$, sums over all such sequences are denoted
$\sum_{\bi(\bullet)} \quad \textrm{and} \quad \sum_{\bj(\bullet)}$.
Write $D_{\bi(\bullet)}=(D_{i_1},\ldots ,D_{i_\bullet})$ and $B_{\bj(\bullet)}=(B_{j_1},\ldots ,B_{j_\bullet})$
for the corresponding sequences of homogeneous components. Also define the following quantities, called total shifted degrees:
$|\bi(\bullet)|=(i_1-1)+\ldots +(i_\bullet -1)$ and $|\bj(\bullet)|=(j_1-2)+\ldots +(j_\bullet -2)$.
\end{notation}

\begin{lemma}\label{LemmaDegreeQMaps}
$|q^\beta_{l_0,l}(D_{\bi(l_0},B_{\bj(l)})|=n+\mu(\beta)-2-|\bi(l_0)|-|\bj(l)|$
\end{lemma}

\begin{proposition}\label{PropMC}
Suppose that $D\in\MC^B(L)$. Then for any $d\in\ZZ$ with $d\neq n$ one has
$$
		\sum_{l_0,l\geq 0}\sum_{\bi(l_0),\bj(l)}
		\sum_{\substack{\beta\in\pi_2(X,L)\\
		\mu(\beta)=2+d-n
		\\+|\bi(l_0)|+|\bj(l)|}}
		T^{\omega(\beta)}\frac{1}{l!}q^\beta_{l_0,l}(D_{\bi(l_0)},B_{\bj(l)})=0 \quad .
$$
\end{proposition}

\begin{proof}
The assumption $D\in\MC^B(L)$ means that $D$ solves the $B$-deformed weak Maurer-Cartan
equation, i.e. $m^0_{B,D}(1)=W^B(D)L$ for some $W^B(D)\in\Lambda$; see Section \ref{SecBMC}.
By Definition \ref{DefBMCAinftyMaps} with $k=0$, this is the same as
$$\sum_{l_0\geq 0}\sum_{l\geq 0}\sum_{\beta\in\pi_2(X,L)}T^{\omega(\beta)}\frac{1}{l!}
q^\beta_{l_0,l}(D^{l_0},B^l)=W^B(D)L \quad .$$
Breaking up $D$ and $B$ into homogeneous components, this becomes
$$\sum_{l_0\geq 0}\sum_{l\geq 0}\sum_{\beta\in\pi_2(X,L)}
\sum_{\bi(l_0),\bj(l)}T^{\omega(\beta)}\frac{1}{l!}
q^\beta_{l_0,l}(D_{\bi(l_0)},B_{\bj(l)})=W^B(D)L \quad .$$
The degree of the right hand side is $n$, while the left hand side is a sum of terms
of different degrees. Using Lemma \ref{LemmaDegreeQMaps} to group terms with the same
degree $d\in\ZZ$ one gets:
$$\sum_{d\in\ZZ}\sum_{l_0\geq 0}\sum_{l\geq 0}\sum_{\bi(l_0),\bj(l)}
\sum_{\substack{\beta\in\pi_2(X,L)\\
		\mu(\beta)=2+d-n
		\\+|\bi(l_0)|+|\bj(l)|}}
T^{\omega(\beta)}\frac{1}{l!}
q^\beta_{l_0,l}(D_{\bi(l_0)},B_{\bj(l)})=W^B(D)L \quad ,$$
and the claim follows by degree comparison with the right hand side.
\end{proof}

When $\mu(\beta)=2+|\bi(l_0)|+|\bj(l)|$, writing $q^\beta_{l_0,l}(D_{\bi(l_0)},B_{\bj(l)})=
n^\beta_{l_0,l}(D_{\bi(l_0)},B_{\bj(l)})L$ one can interpret $n^\beta_{l_0,l}(D_{\bi(l_0)},B_{\bj(l)})\in\QQ$
as the number of $J$-holomorphic disks of class $\beta$ through a given basepoint of $L$,
with $l_0$ boundary points mapping to $D$ and $l$ interior points mapping to $B$. Examining
the $d=n$ case in the proof of Proposition \ref{PropMC}, one gets a formula for $W^B(D)$
in terms of such numbers.

\begin{corollary}\label{CorBMCDiskPotential}
$$W^B(D)=\sum_{l_0,l\geq 0}\sum_{\bi(l_0),\bj(l)}
		\sum_{\substack{\beta\in\pi_2(X,L)\\
		\mu(\beta)=2\\
		+|\bi(l_0)|+|\bj(l)|}}
		\frac{1}{l!}n^\beta_{l_0,l}(D_{\bi(l_0)},B_{\bj(l)})T^{\omega(\beta)}$$
\end{corollary}

The following rescaling of the Maurer-Cartan deformation chain $D$ will play a crucial
role in the proof of Theorem \ref{ThmCurvatureInSpectrum}.

\begin{definition}\label{DefDScaled}
For any $D = \sum_iD_i$, its rescaling is $\tD = \sum_i\frac{i-1}{2}D_i$.
\end{definition}

\begin{lemma}\label{LemmaM1DScaled}
$$m^1_{B,D}(\tD)=\sum_{l\geq 0}\sum_{k\geq 0}\sum_{\bi(k)}\sum_{\beta\in\pi_2(X,L)}T^{\omega(\beta)}
\frac{|\bi(k)|}{2}\frac{1}{l!}q^\beta_{k,l}(D_{\bi(k)},B^l)$$
\end{lemma}

\begin{proof}
Using Definition \ref{DefBMCAinftyMaps} with $k=1$ and Definition \ref{DefDScaled} 
$$m^1_{B,D}(\tD)=
\sum_i\sum_{l_0,l_1}\sum_{l\geq 0}\sum_{\beta\in\pi_2(X,L)}
T^{\omega(\beta)}\frac{i-1}{2}\frac{1}{l!}q^\beta_{1+l_0+l_1,l}(D^{l_0},D_i,D^{l_1},B^l)=$$
$$\sum_i\sum_{k\geq 0}\sum_{l_0+l_1=k}\sum_{l\geq 0}\sum_{\beta\in\pi_2(X,L)}
T^{\omega(\beta)}\frac{i-1}{2}\frac{1}{l!}q^\beta_{1+k,l}(D^{l_0},D_i,D^{l_1},B^l) \quad .$$
Breaking up $D$ into homogeneous components and rearranging the sum one gets
$$m^1_{B,D}(\tD)=\sum_{\beta\in\pi_2(X,L)}T^{\omega(\beta)}\sum_{l\geq 0}\frac{1}{l!}\sum_{k\geq 0}\sum_{\bi(k)}\sum_i\sum_{l_0+l_1=k}
\frac{i-1}{2}q^\beta_{1+k,l}(D_{i_1},\ldots ,D_{i_{l_0}},D_i,D_{i_{l_0+1}},\ldots ,D_{i_k},B^l) \quad .$$
Observe that for any $k\geq 0$ one can rewrite
$$\sum_{\bi(k)}\sum_i\sum_{l_0+l_1=k}
\frac{i-1}{2}q^\beta_{1+k,l}(D_{i_1},\ldots ,D_{i_{l_0}},D_i,D_{i_{l_0+1}},\ldots ,D_{i_k},B^l)=
\sum_{\bi(1+k)}\frac{|\bi(1+k)|}{2}q^\beta_{1+k,l}(D_{\bi(1+k)},B^l) \quad ,$$
thus obtaining
$$m^1_{B,D}(\tD)=\sum_{\beta\in\pi_2(X,L)}T^{\omega(\beta)}\sum_{l\geq 0}\frac{1}{l!}
\sum_{k\geq 0}\sum_{\bi(1+k)}\frac{|\bi(1+k)|}{2}q^\beta_{1+k,l}(D_{\bi(1+k)},B^l) \quad .$$
The last expression is the one claimed after shifting the index $k$, because
$|\bi(0)|=0$.
\end{proof}

\begin{proposition}\label{PropKCapL}
For any $D\in\MC^B(L)$, one has $K^B \cap_{B,D} [L] = W^B(D)[L]$ .
\end{proposition}

\begin{proof}
Assume that $L\subset X\setminus Z$ for some divisor $Z$ Poincar\'{e} dual to a lift
of $c_1$ in $H^2(X,L)$, so that $\beta\cdot[Z]=\mu(\beta)/2$ for every $\beta\in\pi_2(X,L)$.
Decomposing the bulk-deformation cycle $B=\sum_jB_j$ into homogeneous components
$$K^B\cap_{B,D}[L]=[Z]\cap_{B,D}[L]-\sum_{j}\frac{j-2}{2}[B_j]\cap_{B,D}[L] \quad .$$
Recall that $m^0_{B,D}(1)=W^B(D)L$ thanks to the assumption
$D\in\MC^B(L)$, so it suffices to check that $K^B\cap_{B,D}L-m^0_{B,D}(1)$ is $m^1_{B,D}$-exact
at the chain level. We verify that the rescaled chain $\tD$ (Definition \ref{DefDScaled})
is a primitive. In the definition of BMC-deformed cap product (Definition \ref{DefBMCCap}),
one can group together contributions with the same number $l_0+l_1=k$ of boundary points
mapping to $D$ and get
$$Z\cap_{B,D}L = \sum_{l_0,l_1\geq 0}\sum_{l\geq 0}\sum_{\beta\in\pi_2(X,L)}T^{\omega(\beta)}\frac{1}{l!}(Z\cap_{B,D}L)^{l_0,l_1;l}_\beta$$
$$=\sum_{\beta\in\pi_2(X,L)}T^{\omega(\beta)}\sum_{l\geq 0}\frac{1}{l!}\sum_{k\geq 0}\sum_{l_0+l_1=k}(Z\cap_{B,D}L)^{l_0,l_1;l}_\beta \quad .$$
From the axioms of Section \ref{SecAxioms}, removing the vacuous boundary constraint of
an input marked point mapping to $L$ yields an equality of chains
$$\sum_{l_0+l_1=k}(Z\cap_{B,D}L)_\beta^{l_0,l_1;l}=\frac{\mu(\beta)}{2}q^\beta_{k,l}(D^k,B^l) \quad .$$
Using this one can rewrite
$$Z\cap_{B,D}L=\sum_{k\geq 0}\sum_{l\geq 0}\sum_{\beta\in\pi_2(X,L)}T^{\omega(\beta)}\frac{\mu(\beta)}{2}\frac{1}{l!}q^\beta_{k,l}(D^k,B^l) \quad .$$
Definition \ref{DefBMCAinftyMaps} with $k=0$ then gives
$$K^B\cap_{B,D}L-m^0_{B,D}(1)=\sum_{k\geq 0}\sum_{l\geq 0}\sum_{\beta\in\pi_2(X,L)}T^{\omega(\beta)}
\left(\frac{\mu(\beta)}{2}-1\right)\frac{1}{l!}q^\beta_{k,l}(D^k,B^l) - \sum_j\frac{j-2}{2}B_j\cap_{B,D}L\quad .$$
For each $k\geq 0$ and $l\geq 0$, break up $D=\sum_iD_i$ into homogenous components and isolate the contributions
of disk classes $\beta\in\pi_2(X,L)$ with Maslov index $\mu(\beta)=2+|\bi(k)|+|\bj(l)|$, which are the
ones contributing to $W^B(D)$ by Corollary \ref{CorBMCDiskPotential}:
$$K^B\cap_{B,D}L-m^0_{B,D}(1)=$$
$$\underbracket{\sum_{k\geq 0}\sum_{l\geq 0}\sum_{\bi(k)}\sum_{\bj(l)}
\sum_{\substack{\beta\in\pi_2(X,L)\\ \mu(\beta)=2+|\bi(k)|+|\bj(l)|}}T^{\omega(\beta)}\frac{|\bi(k)|}{2}\frac{1}{l!}
q^\beta_{k,l}(D_{\bi(k)},B_{\bj(l)})}_\textrm{(1)}$$
$$ +
\underbracket{\sum_{k\geq 0}\sum_{l\geq 0}\sum_{\bi(k)}\sum_{\bj(l)}
\sum_{\substack{\beta\in\pi_2(X,L)\\ \mu(\beta)=2+|\bi(k)|+|\bj(l)|}}T^{\omega(\beta)}\frac{|\bj(k)|}{2}\frac{1}{l!}
q^\beta_{k,l}(D_{\bi(k)},B_{\bj(l)})}_\textrm{(2)}$$
$$ +
\underbracket{\sum_{k\geq 0}\sum_{l\geq 0}\sum_{\bi(k)}\sum_{\bj(l)}
\sum_{\substack{\beta\in\pi_2(X,L)\\ \mu(\beta)\neq 2+|\bi(k)|+|\bj(l)|}}T^{\omega(\beta)}\left(\frac{\mu(\beta)}{2}-1\right)\frac{1}{l!}
q^\beta_{k,l}(D_{\bi(k)},B_{\bj(l)})}_\textrm{(3)}$$
$$ -
\underbracket{\sum_j\frac{j-2}{2}B_j\cap_{B,D}L}_\textrm{(4)} \quad .$$
Term (1) resembles $m^1_{B,D}(\tD)$ from
Lemma \ref{LemmaM1DScaled}, but only includes the contributions of disk classes
$\beta\in\pi_2(X,L)$ with Maslov index $\mu(\beta)=2+|\bi(k)|+|\bj(l)|$. Term (3) can be rewritten as
$$(3) = \sum_{\substack{d\in\ZZ\\ d\neq n}}
\sum_{k\geq 0}\sum_{l\geq 0}\sum_{\bi(k)}\sum_{\bj(l)}
\smashoperator[r]{\sum_{\substack{\beta\in\pi_2(X,L)\\ \mu(\beta)= 2+d-n\\+|\bi(k)|+|\bj(l)|}}}T^{\omega(\beta)}
\left(\frac{d-n}{2}+\frac{|\bi(k)|}{2}+\frac{|\bj(l)|}{2} \right)\frac{1}{l!}q^\beta_{k,l}(D_{\bi(k)},B_{\bj(l)})=$$
$$
\sum_{\substack{d\in\ZZ\\ d\neq n}}
\frac{d-n}{2}
\underbracket{\left(\sum_{k\geq 0}\sum_{l\geq 0}\sum_{\bi(k)}\sum_{\bj(l)}
\smashoperator[r]{\sum_{\substack{\beta\in\pi_2(X,L)\\ \mu(\beta)= 2+d-n\\+|\bi(k)|+|\bj(l)|}}}T^{\omega(\beta)}
\frac{1}{l!}q^\beta_{k,l}(D_{\bi(k)},B_{\bj(l)})\right)}_\textrm{(3a)}$$
$$+
\underbracket{\sum_{\substack{d\in\ZZ\\ d\neq n}}
\left(\sum_{k\geq 0}\sum_{l\geq 0}\sum_{\bi(k)}\sum_{\bj(l)}
\smashoperator[r]{\sum_{\substack{\beta\in\pi_2(X,L)\\ \mu(\beta)= 2+d-n\\+|\bi(k)|+|\bj(l)|}}}T^{\omega(\beta)}
\frac{|\bi(k)|}{2}\frac{1}{l!}q^\beta_{k,l}(D_{\bi(k)},B_{\bj(l)})\right)}_\textrm{(3b)}$$
$$+
\underbracket{\sum_{\substack{d\in\ZZ\\ d\neq n}}
\left(\sum_{k\geq 0}\sum_{l\geq 0}\sum_{\bi(k)}\sum_{\bj(l)}
\smashoperator[r]{\sum_{\substack{\beta\in\pi_2(X,L)\\ \mu(\beta)= 2+d-n\\+|\bi(k)|+|\bj(l)|}}}T^{\omega(\beta)}
\frac{|\bj(l)|}{2}\frac{1}{l!}q^\beta_{k,l}(D_{\bi(k)},B_{\bj(l)})\right)}_\textrm{(3c)} \quad . $$
Observe that $(3a)=0$ thanks to the assumption
$D\in\MC^B(L)$ and Proposition \ref{PropMC}, and that $(1)+(3b)=m^1_{B,D}(\tD)$, so that overall
$$K^B\cap_{B,D}L-m^0_{B,D}(1)=(1)+(2)+(3a)+(3b)+(3c)+(4)=$$
$$m^1_{B,D}(\tD)
+\underbracket{
\sum_{k\geq 0}\sum_{l\geq 0}\sum_{\bi(k)}\sum_{\bj(l)}
\sum_{\beta\in\pi_2(X,L)}T^{\omega(\beta)}
\frac{|\bj(l)|}{2}\frac{1}{l!}q^\beta_{k,l}(D_{\bi(k)},B_{\bj(l)})}_\textrm{(2)+(3c)}
-\underbracket{\sum_j\frac{j-2}{2}B_j\cap_{B,D}L}_\textrm{(4)} \quad .$$
To conclude, it suffices to check that $(2)+(3c)=(4)$. Using the definition of BMC-deformed
cap product (Definition \ref{DefBMCCap}) and grouping together the terms with the same
number $l_0+l_1=k$ of input boundary marked points mapping to $D$ one gets
$$(4) = \sum_{\beta\in\pi_2(X,L)}T^{\omega(\beta)}
\sum_j\frac{j-2}{2}\sum_{l\geq 0}\frac{1}{l!}\sum_{k\geq 0}\sum_{l_0+l_1=k}(B_j\cap_{B,D}L)^{l_0,l_1;l}_\beta \quad.$$
From the axioms of Section \ref{SecAxioms}, removing the vacuous boundary constraint of
an input boundary marked point mapping to $L$ yields an equality of chains
$$\sum_{l_0+l_1=k}(B_j\cap_{B,D}L)^{l_0,l_1;l}_\beta=\frac{1}{l+1}\sum_{t=0}^lq^\beta_{k,l+1}(D^k,B^t,B_j,B^{l-t}) \quad ,$$
and substitution in the previous equation gives
$$(4)=
\sum_{\beta\in\pi_2(X,L)}T^{\omega(\beta)}
\sum_{k\geq 0}\sum_{l\geq 0}\frac{1}{(l+1)!}\sum_j\frac{j-2}{2}\sum_{t=0}^lq^\beta_{k,l+1}(D^k,B^t,B_j,B^{l-t}) \quad .$$
Breaking up $B$ into homogeneous components yields
$$(4)=
\sum_{\beta\in\pi_2(X,L)}T^{\omega(\beta)}
\sum_{k\geq 0}\sum_{l\geq 0}\frac{1}{(l+1)!}\sum_{\bj(l)}\sum_j\frac{j-2}{2}\sum_{t=0}^lq^\beta_{k,l+1}(D^k,B_{\bj(t)},B_j,B_{\bj(l-t)}) \quad ,$$
and rewriting the sum
$$\sum_{\bj(l)}\sum_j\sum_{t=0}^l\frac{j-2}{2}q^\beta_{k,l+1}(D^k,B_{\bj(t)},B_j,B_{\bj(l-t)})=
\sum_{\bj(l+1)}\frac{|\bj(l+1)|}{2}q^\beta_{k,l+1}(D^k,B_{\bj(l+1)})$$
one gets
$$(4)=\sum_{\beta\in\pi_2(X,L)}T^{\omega(\beta)}
\sum_{k\geq 0}\sum_{l\geq 0}\frac{1}{(l+1)!}\sum_{\bj(l+1)}\frac{|\bj(l+1)|}{2}q^\beta_{k,l+1}(D^k,B_{\bj(l+1)}) \quad ,$$
which is precisely $(2)+(3c)$ after a shift of the index $l$, using the fact that $\bj(0)=0$.

\end{proof}

\begin{remark}
The fact that $Z\cap_{B,D}L-m^0_{B,D}(1)$ is only exact, as opposed to zero,
is a major difference between Maurer-Cartan deformations where $D$ has cohomological
degree 1 and general ones.
\end{remark}

\begin{theorem}\label{ThmCurvatureInSpectrum}
For any $D\in\MC^B(L)$, $\HF^B(L,D)\neq 0$ implies that $W^B_L(D)$ is an eigenvalue of
the Dubrovin operator $K^B \star_B - : \QH^B(X) \to \QH^B(X)$ .
\end{theorem}

\begin{proof}
The assumption $D\in\MC^B(L)$ guarantees that the curvature $W^B(D)\in\Lambda$ is defined,
and the BMC-deformed cap product of chains descends to a map on homology
$\cap_{B,D}:\QH^B(X)\otimes\HF^B(L,D)\to\HF^B(L,D)$ (Proposition \ref{PropBMCCapOnHomology}).
The curvature $W^B(D)$ is an eigenvalue of $K^B \star_B - : \QH^B(X) \to \QH^B(X)$ if and only if
$(K^B-W^B(D)[X])\star_B - : \QH^B(X) \to \QH^B(X)$ is not invertible. Assume by contradiction
that this is not the case. Then such operator is surjective, and there is some $[Q]\in\QH^B(X)$
such that $(K^B-W^B(D)[X])\star_B[Q]=[X]$. Observe that $(K^B-W^B(D)[X])\cap_{B,D}[L]=0$
(Proposition \ref{PropModuleStructure} and \ref{PropKCapL}). Capping
this equation by $[Q]$ and using again the module structure of $\HF^B(L,D)$ over $\QH^B(X)$
(Proposition \ref{PropModuleStructure}) one gets
$$0 = [Q]\cap_{B,D}((K^B-W^B(D)[X])\cap_{B,D}[L]) =$$
$$=([Q]\star_B (K^B - W^B(D)[X]))\cap_{B,D}[L]=
[X]\cap_{B,D} [L] = [L] \quad .$$
This is against the fact $\HF^B(L,D)\neq 0$ by assumption, because $[L]$ is the unit of
this algebra (Proposition \ref{PropBMCHF}).
\end{proof}

\section{Decoupling phenomenon}\label{SecDecoupling}

This section discusses an example where the spectrum of some finite energy truncation
of the Dubrovin operator $K^B$ decouples for explicit bulk parameters $B$, and explains
how this is related to decoupling of the full Dubrovin operator.

\subsection{Classical cohomology}

Let $X=\Gr(2,4)$ be the Grassmannians of complex planes in $\CC^4$.
This is a $4$-dimensional smooth Fano variety, that can be realized as a quadric
hypersurface in $\PP^5$ under the Pl\"{u}cker embedding. As symplectic structure $\omega$
take the Fubini-Study form on $\PP^5$ restricted to $X$.
The cohomology $\H(X;\CC)$ has dimension $6$ and is concentrated in even degrees,
with a basis given by classes $\sigma_d=\operatorname{PD}([X_d])$ Poincar\'{e} dual
to Schubert cycles $X_d\subset\Gr(2,4)$. The labels $d\subset 2\times 2$ are Young diagrams
in a grid with two rows and two columns. The codegree $|X_d|^\vee=|\sigma_d|$ is twice
the number of boxes in the corresponding Young diagram $d$, and we order the basis elements
as follows:
\vspace{0.3em}
\ytableausetup{boxsize=0.3em}
\begin{center}
\begin{tabular}{ c | c | c | c | c | c | c }
degree & 0 & 2 & 4 & 4 & 6 & 8 \\
\hline
class & $\sigma_0=\sigma_{\emptyset}$ & $\sigma_1=\sigma_{\ydiagram{1,0}}$ & $\sigma_2=\sigma_{\ydiagram{2,0}}$ & $\sigma_3=\sigma_{\ydiagram{1,1}}$ & $\sigma_4=\sigma_{\ydiagram{2,1}}$ & $\sigma_5=\sigma_{\ydiagram{2,2}}$
\end{tabular}
\end{center}
\vspace{0.3em}
The first Chern class of $X$ is $c_1 = 4\sigma_{\ydiagram{1,0}}$. Cup products of Schubert
classes can be computed combinatorially from their Young diagrams, using the
Littlewood-Richardson rule. We record for future reference the intersection matrix $g$
of $X$ and note that $g=g^{-1}$:

$$
\begin{pmatrix}
0 & 0 & 0 & 0 & 0 & 1\\
0 & 0 & 0 & 0 & 1 & 0\\
0 & 0 & 1 & 0 & 0 & 0\\
0 & 0 & 0 & 1 & 0 & 0\\
0 & 1 & 0 & 0 & 0 & 0\\
1 & 0 & 0 & 0 & 0 & 0
\end{pmatrix}
$$

\subsection{Big $q$-deformation and Dubrovin operator}

Denote $t=(t_0,t_1,t_2,t_3,t_4,t_5)\in\CC^6$ the coordinates on the complex cohomology
$\H(X;\CC)$ induced by the basis of Schubert classes. Following Fulton-Pandharipande
\cite{FP}, the genus $0$ Gromov-Witten potential of $X$ is a formal power series
$\Phi\in\QQ[[t]]$ given by the sum $\Phi=\Phi_c+\hat{\Phi}$
of a classical part and a quantum part. Denoting $n=(n_0,n_1,n_2,n_3,n_4,n_5)\in\NN^5$
and $|n|=n_0+n_1+n_2+n_3+n_4+n_5$, the classical part is:
$$\Phi_c = \sum_{|n|=3}\langle
\sigma_{\emptyset}^{n_0}\sigma_{\ydiagram{1,0}}^{n_1}\sigma_{\ydiagram{2,0}}^{n_2}
\sigma_{\ydiagram{1,1}}^{n_3}\sigma_{\ydiagram{2,1}}^{n_4}\sigma_{\ydiagram{2,2}}^{n_5},[X]\rangle
\frac{t_0^{n_0}t_1^{n_1}t_2^{n_2}t_3^{n_3}t_4^{n_4}t_5^{n_5}}{n_0!n_1!n_2!n_3!n_4!n_5!} \quad .$$
$$=\frac{1}{2}t_1^2t_3 + \frac{1}{2}t_1^2t_2 + \frac{1}{2}t_0t_3^2 + \frac{1}{2}t_0t_2^2
+t_0t_1t_4 + \frac{1}{2}t_0^2t_5 \quad ,$$
where we used the Littlewood-Richardson rule to simplify the expression.
Denoting $\hat{n}=(n_2,n_3,n_4,n_5)\in\NN^4$ and $|\hat{n}|= n_2 + n_3 + 2n_4 + 3n_5$, the
quantum part is:
$$\hat{\Phi}=\sum_{d=1}^\infty\left(\sum_{|\hat{n}|=4d+1}N(\hat{n})\frac{t_2^{n_2}
t_3^{n_3}t_4^{n_4}t_5^{n_5}}{n_2!n_3!n_4!n_5!}\right)e^{dt_1} \quad .$$
Here $N(\hat{n})\in\QQ$ is a Gromov-Witten number, counting the stable
$J$-holomorphic spheres of class $\beta=d\operatorname{PD}(\sigma_{\ydiagram{1,0}})$
with $n_2$ points mapping to $\operatorname{PD}(\sigma_{\ydiagram{2,0}})$, $n_3$ points
mapping to $\operatorname{PD}(\sigma_{\ydiagram{1,1}})$, $n_4$ points mapping to
$\operatorname{PD}(\sigma_{\ydiagram{2,1}})$, $n_5$ points mapping to $\operatorname{PD}(\sigma_{\ydiagram{2,2}})$.
These curves are rigid thanks to the dimension formula for the moduli space of $J$-holomorphic
spheres of class $\beta$ and
$$|\hat{n}|=n_2\left(\frac{1}{2}|\sigma_{\ydiagram{2,0}}|-1\right)+n_3\left(\frac{1}{2}|\sigma_{\ydiagram{1,1}}|-1\right)+
n_4\left(\frac{1}{2}|\sigma_{\ydiagram{2,1}}|-1\right)+n_5\left(\frac{1}{2}|\sigma_{\ydiagram{2,2}}|-1\right)$$
$$=\dim_\CC(X)+\langle c_1,\beta\rangle-3=4d+1 \quad .$$
The product structure on the big quantum cohomology $\QH^t(X)$ is given by
$$\sigma_i\star_t\sigma_j = \sum_{e,f}\Phi_{ije}g^{ef}\sigma_f \quad ,$$
where $\Phi_{ije}$ denotes the third-order partial derivative of $\Phi$ with respect
to $t_i,t_j,t_e$, and $g^{ef}$ are the entries of the inverse intersection matrix $g^{-1}$.
Dubrovin's operator acts on $\QH^t(X)$ multiplying by
$$K^t = 4\sigma_{\ydiagram{1,0}} + t_0\sigma_{\emptyset} - t_2\sigma_{\ydiagram{2,0}}
-t_3\sigma_{\ydiagram{1,1}} -2t_4\sigma_{\ydiagram{2,1}} -3t_5\sigma_{\ydiagram{2,2}} \quad .$$
Setting $e^{t_1}=q$, we can think of this as $t$-dependent family of operators
with spectra in the Novikov field $\Lambda$. For the purposes of studying the simplicity
of such spectra, we can set $t_0=0$ thanks to the following.

\begin{lemma}
The spectra of $K^t$ and $K^{\hat{t}}$ differ by a translation of $t_0\in\CC$. In particular,
one is simple if and only if the other is.
\end{lemma}

\begin{proof}
Since $K^t = K^{\hat{t}} +t_0\sigma_\emptyset = K^{\hat{t}} + t_0\operatorname{Id}$, the
characteristic polynomials $p_{K^t}(\lambda)$ of $K^t$ and $p_{K^{\hat{t}}}(\lambda)$
of $K^{\hat{t}}$ are related by $p_{K^t}(\lambda)=p_{K^{\hat{t}}}(\lambda-t_0)$.
\end{proof}

From now on, denote $\hat{t}=(t_2,t_3,t_4,t_5)\in\CC^4$, and think of $K^{\hat{t}}$
as a family of operators with spectra in the Novikov field $\Lambda$.

\subsection{Finite energy truncations}

After fixing a basis, represent the Dubrovin operator with a matrix
$M^{\hat{t}}$ with entries in $\Lambda$. Decompose
$M^{\hat{t}}=\sum_{d\geq 0}M^{\hat{t}}_dq^d$,
where $M^{\hat{t}}_d$ are matrices with entries in $\CC$ containing the coefficients
of $q^d$. The sum runs over Novikov exponents $d\geq0$ since the Gromov-Witten
potential $\Phi$ and its third-order derivatives do not contain negative powers of the
formal parameter $q$.

\begin{definition}\label{DefFiniteEnergyTruncation}
The truncation of $K^{\hat{t}}$ at energy $\alpha\in\RR_{\geq 0}$ is the
$\Lambda$-linear operator $K^{\hat{t}}_{<\alpha}$ associated to the matrix
$$M^{\hat{t}}_{<\alpha}=\sum_{0\leq d<\alpha}M^{\hat{t}}_dq^d \quad .$$
\end{definition}

\begin{remark}
While suppressed in the notation, the finite
energy truncations $K^{\hat{t}}_{<\alpha}$ do depend on the basis chosen to represent
$K^{\hat{t}}$ with a matrix. To see this, observe for example that the following matrices
are similar over $\Lambda$ but have different finite energy truncations:
$$
\begin{pmatrix}
q^{-1} & 0 \\
0 & q^{-2}
\end{pmatrix}
\begin{pmatrix}
0 & q \\
q & 0
\end{pmatrix}
\begin{pmatrix}
q & 0 \\
0 & q^2
\end{pmatrix}
=
\begin{pmatrix}
0 & q^2 \\
1 & 0
\end{pmatrix}
$$
\end{remark}

Simplicity of the
spectrum of $K^{\hat{t}}$ and its truncations are can be related as follows.
Denote by $\Delta^{\hat{t}} = \operatorname{Res}(p_{K^{\hat{t}}},p_{K^{\hat{t}}}')\in\Lambda$
the discriminant of $K^{\hat{t}}$, i.e. the resultant of the characteristic polynomial
of $K^{\hat{t}}$ and its derivative. Also denote $\Delta^{\hat{t}}_{<\alpha}\in\Lambda$
the discriminant of the truncation $K^{\hat{t}}_{<\alpha}$.

\begin{proposition}\label{PropSimpleTruncations}
The following implications hold:
\begin{enumerate}
	\item if $K^{\hat{t}}$ has simple spectrum, then there is some $\alpha>\operatorname{val}(\Delta^{\hat{t}})$
	such that $K^{\hat{t}}_{<\alpha}$ has simple spectrum ;
	\item if $\Delta^{\hat{t}}_{<\alpha}\neq 0$ modulo $q^\alpha$, then $K^{\hat{t}}$ has simple spectrum .
\end{enumerate}
\end{proposition}

\begin{proof}
(1) If $K^{\hat{t}}$ has simple spectrum then $\Delta^{\hat{t}}\neq 0$, and so it has finite valuation
$\operatorname{val}(\Delta^{\hat{t}})<\infty$. Now observe that $\Delta^{\hat{t}} = \Delta^{\hat{t}}_{<\alpha}$
modulo $q^\alpha$ because the discriminant is a polynomial function of the matrix entries.
Since the left hand side is nonzero modulo $q^\alpha$ for $\alpha>\operatorname{val}(\Delta^{\hat{t}})$,
it must be the case that $\Delta^{\hat{t}}_{<\alpha}\neq 0$, hence $K^{\hat{t}}_{<\alpha}$ has simple spectrum.
\\
(2) Using again $\Delta^{\hat{t}} = \Delta^{\hat{t}}_{<\alpha}$
modulo $q^\alpha$ and the assumption that the right hand side is nonzero, it follows
that $\Delta^{\hat{t}}\neq 0$, hence $K^{\hat{t}}$ has simple spectrum.
\end{proof}

Crucially, finite energy truncations are determined by finitely many
Gromov-Witten numbers, and these are typically computable from a small set using
the recursion relations given by the WDVV equations. In our case the Gromov-Witten numbers
$N(\hat{n})\in\QQ$ with $|\hat{n}|=1$ are written below (see e.g. Di Francesco-Itzykson
\cite[Proposition 9]{DI}), and can be obtained by recursion from $N(0,0,1,1)=1$:

\begin{center}
\begin{tabular}{ l l l l }
$N(5,0,0,0)=0$ & $N(4,1,0,0)=0$ & $N(3,2,0,0)=1$ & $N(3,0,1,0)=0$\\
$N(2,1,1,0)=1$ & $N(2,0,0,1)=0$ & $N(1,1,0,1)=1$ & $N(1,0,2,0)=1$\\
$N(0,0,1,1)=1$\\
\end{tabular}
\end{center}

Note that these numbers are non-negative integers, because
the target is a homogeneous variety; see \cite[Lemma 14]{FP}. Below is the matrix
of the truncation $(K^{\hat{t}})_{< 2}$ with respect to the basis of
Schubert classes. Its calculation only involves the Gromov-Witten numbers $N(\hat{n})\in\QQ$ with $|\hat{n}|=1$.

\begin{center}
\setlength{\tabcolsep}{6pt}
\renewcommand{\arraystretch}{2}
\begin{tabular}{ cm{0.1\textwidth} cm{0.1\textwidth} cm{0.1\textwidth} cm{0.1\textwidth} cm{0.1\textwidth} cm{0.1\textwidth} }
$0$ &
$(-2t_2t_3+2t_4+4t_4t_5)q$ &
$3t_3q$ &
$3t_2q$ &
$4q$ &
$0$\\
$4$ &
$(1/2t_2^2t_3+t_2t_4+t_5)q$ &
$(2t_2t_3+2t_4)q$ &
$t_2^2q$ &
$3t_2q$ &
$4q$\\
$-t_2$ &
$4$ &
$(1/2t_2t_3^2+t_3t_4)q$ &
$(1/2t_2^2t_3+t_2t_4+t_5)q$ &
$(2t_2t_3+2t_4)q$ &
$3t_3q$\\
$-t_3$ &
$4$ &
$(1/2t_2^2t_3+t_2t_4+t_5)q$ &
$1/6t_2^3q$ &
$t_2^2q$ &
$3t_2q$\\
$-2t_4$ &
$-t_2-t_3+(1/12t_2^3t_3^2-1/2t_2^2t_3t_4+2t_2t_3t_5-1/2t_2t_4^2-t_4t_5-3t_4t_5^2)q$ &
$4$ &
$4$ &
$(1/2t_2^2t_3+t_2t_4+t_5)q$ &
$(-2t_2t_3+2t_4+4t_4t_5)q$\\
$-3t_5$ &
$-2t_4$ &
$-t_2$ &
$-t_3$ &
$4$ &
$0$\\
\end{tabular}
\end{center}

\subsection{Decoupling}

The Grassmannian $\Gr(2,4)$ is perhaps one of the simplest examples where the
the spectrum of the Dubrovin operator is not simple for all bulk-parameters.

\begin{lemma}(see for example \cite[Proposition 1.12]{C20})
The eigenvalues of $K^0=c_1$ are $\pm 4\sqrt{2}$, $\pm 4\sqrt{2}i$, $0$. Moreover $0$
has algebraic multiplicity $2$.
\end{lemma}

This section analyzes the spectrum of the finite energy truncation 
$K^{\hat{t}}_{< 2}$ computed earlier one bulk-parameter at the time, i.e. allowing only
one bulk-parameter to be non-vanishing. This corresponds to analyzing the
quantum contributions of $J$-holomorphic spheres with incidence condition on a single
Schubert class.

\begin{itemize}

\item \textbf{Schubert class $\sigma_2=\sigma_{\ydiagram{2,0}}$: $t_3=t_4=t_5=0$.} The
truncation $K^{\hat{t}}_{< 2}$ specializes to the $t_2$-dependent family of operators

$$
\begin{pmatrix}
0 &
0 &
0 &
3t_2q &
4q &
0\\
4 &
0 &
0 &
t_2^2q &
3t_2q &
4q\\
-t_2 &
4 &
0 &
0 &
0 &
0\\
0 &
4 &
0 &
1/6t_2^3q &
t_2^2q &
3t_2q\\
0 &
-t_2 &
4 &
4 &
0 &
0\\
0 &
0 &
-t_2 &
0 &
4 &
0\\
\end{pmatrix}
$$

whose discriminant is $\Delta^{t_2}_{<2}=(-1208925819614629174706176 t_2^2)q^8 + o(q^8)$. In
particular, these have simple spectrum for $t_2\neq 0$.

\item \textbf{Schubert class $\sigma_3=\sigma_{\ydiagram{1,1}}$, $t_2=t_4=t_5=0$.} The
truncation $K^{\hat{t}}_{<2}$ specializes to the $t_3$-dependent family of operators

$$
\begin{pmatrix}
0 &
0 &
3t_3q &
0 &
4q &
0\\
4 &
0 &
0 &
0 &
0 &
4q\\
0 &
4 &
0 &
0 &
0 &
3t_3q\\
-t_3 &
4 &
0 &
0 &
0 &
0\\
0 &
-t_3 &
4 &
4 &
0 &
0\\
0 &
0 &
0 &
-t_3 &
4 &
0\\
\end{pmatrix}
$$

whose discriminant is $\Delta^{t_3}_{<2}=(-1208925819614629174706176t_3^2)q^8 + o(q^8)$.
In particular, these have simple spectrum for $t_3\neq 0$.

\item \textbf{Schubert class $\sigma_4=\sigma_{\ydiagram{2,1}}$, $t_2=t_3=t_5=0$.} The
truncation $K^{\hat{t}}_{<2}$ specializes to the $t_4$-dependent family of operators

$$
\begin{pmatrix}
0 &
2t_4q &
0 &
0 &
4q &
0\\
4 &
0 &
2t_4q &
0 &
0 &
4q\\
0 &
4 &
0 &
0 &
2t_4q &
0\\
0 &
4 &
0 &
0 &
0 &
0\\
-2t_4 &
0 &
4 &
4 &
0 &
2t_4q\\
0 &
-2t_4 &
0 &
0 &
4 &
0\\
\end{pmatrix}
$$

whose discriminant is $\Delta^{t_4}_{<2}=0$. In particular, these have non-simple spectrum
for all $t_4\in\CC$.

\item \textbf{Schubert class $\sigma_5=\sigma_{\ydiagram{2,2}}$, $t_2=t_3=t_4=0$.} The
truncation $K^{\hat{t}}_{<2}$ specializes to the $t_5$-dependent family of operators

$$
\begin{pmatrix}
0 &
0 &
0 &
0 &
4q &
0\\
4 &
t_5q &
0 &
0 &
0 &
4q\\
0 &
4 &
0 &
t_5q &
0 &
0\\
0 &
4 &
t_5q &
0 &
0 &
0\\
0 &
0 &
4 &
4 &
t_5q &
0\\
-3t_5 &
0 &
0 &
0 &
4 &
0\\
\end{pmatrix}
$$

whose discriminant is $\Delta^{t_5}_{<2}=(1888946593147858085478400t_5^2)q^9 + o(q^9)$.
In particular, these have simple spectrum for $t_5\neq 0$.

\end{itemize}

The calculations above have been used
to produce Figure \ref{FigSpectralMovie}, and can be summarized in the following theorem.

\begin{theorem}\label{ThmDecouplingGr24}
Consider the Grassmannian $X=\Gr(2,4)$ and the energy $\alpha=2$. The truncation
of $(K^B)_{<2}$ has simple spectrum when $B\neq 0$ is supported on the Schubert
cycles $X_{\ydiagram{2,0}}, X_{\ydiagram{1,1}}, X_{\ydiagram{2,2}}\subset\Gr(2,4)$, while
it doesn't when it is supported on the Schubert cycle $X_{\ydiagram{2,1}}\subset\Gr(2,4)$.
\end{theorem}

In view of Proposition \ref{PropSimpleTruncations}, this theorem suggests that
$K^B$ could have simple spectrum whenever $B$ is supported on the Schubert cycles
$X_{\ydiagram{2,0}}, X_{\ydiagram{1,1}}, X_{\ydiagram{2,2}}\subset\Gr(2,4)$.
However, one cannot conclude this with the given information, since $(\Delta^B)_{<2}=0$
modulo $q^2$ in each of these cases. Instead, the calculations with $B$ supported
on the Schubert cycle $X_{\ydiagram{2,1}}\subset\Gr(2,4)$
say that in this case either $K^B$ has non-simple spectrum, or
$\operatorname{val}(\Delta^B)\geq 2$.

\bibliographystyle{abbrv}
\bibliography{biblio}

\end{document}